\documentclass[11pt]{article} % use larger type; default would be 10pt

\usepackage[utf8]{inputenc} % set input encoding (not needed with XeLaTeX)

\usepackage[margin=1in]{geometry}% to change the page dimensions
\usepackage{graphicx} % support the \includegraphics command and options

\usepackage{booktabs} % for much better looking tables
\usepackage{array} % for better arrays (eg matrices) in maths=
\usepackage{paralist} % very flexible & customisable lists (eg. enumerate/itemize, etc.)
\usepackage{verbatim} % adds environment for commenting out blocks of text & for better verbatim
\usepackage{subfig} % make it possible to include more than one captioned figure/table in a single float
\usepackage{transparent}%Inkscape
\usepackage{color}
\usepackage{fancyhdr} % This should be set AFTER setting up the page geometry
\usepackage{sectsty}
\allsectionsfont{\sffamily\mdseries\upshape} % (See the fntguide.pdf for font help)
\usepackage[nottoc,notlof,notlot]{tocbibind} % Put the bibliography in the ToC
\usepackage[titles,subfigure]{tocloft} % Alter the style of the Table of Contents

\usepackage{amsmath,amsthm}
\usepackage{amssymb}
\usepackage{slashed}
\usepackage{accents}
\usepackage{wrapfig}
\newcommand{\ubar}[1]{\underaccent{\bar}{#1}}

\newtheorem{theorem}{Theorem}[section]
\newtheorem{proposition}[theorem]
{Proposition}
\newtheorem{lemma}{Lemma}[subsection]
\newtheorem{corollary}{Corollary}[theorem]
\newtheorem{definition}{Definition}[section]
\newtheorem{remark}{Remark}[section]

\DeclareMathOperator{\tr}{tr}

\title{The Isometric Embedding Problem in a Null Cone}
\author{Henri Roesch}
\date{\today} 

\begin{document}
\maketitle

\begin{abstract}
In the first part of this paper, we extend the result of Li-Wang \cite{li2020} on the linearized embedding problem to a compact manifold of arbitrary dimension. Using this, we then show that any metric perturbation of a embedded $n$-sphere is also isometrically embedded up to a solution of the homogenous Codazzi equation, irrespective of the ambient geometry. In the second part we specialize to dimension two, and study these results within an ambient Null Cone. Specifically, given a path of metrics on the 2-sphere and an initial isometric embedding, we develop a small parameter existence and uniqueness theorem for paths of isometric embeddings. In the final part, after imposing asymptotic decay conditions on the Null Cone, we show that any metric on the 2-sphere can be isometrically embedded up to a scaling factor. We then prove the existence of a foliation in a neighborhood of infinity.
\end{abstract}

\tableofcontents
\newpage
\section{Introduction}
In 1916, Weyl proposed the following problem: Given a smooth metric on a 2-sphere exhibiting positive Gauss curvature, can an isometric embedding be given into the three dimensional Euclidean space? Weyl suggested an approach using the method of continuity in \cite{weyl1}, for the analytic case he  also produced the openness part of the argument. Weyl also provided an estimate on the mean curvature for strictly convex surfaces as a $C^2$ a priori estimate toward the continuity argument. The analytic case was subsequently solved by Lewy \cite{lewy1}. The smooth case was then settled by Nirenberg in his famous paper \cite{nirenberg1}, another solution was also obtained independently by Alexandrov-Pogorelov \cite{pogorelov1}. Pogorelov then generalized Nirenberg's approach to the hyperbolic case in \cite{pogorelov2}, and other Riemannian manifolds in \cite{pogorelov3},\cite{pogorelov4}. Work of Guan-Li \cite{guanli1}, Hong-Zuily \cite{hongzuily1}, and Iaia \cite{iaia1} then generalized Weyl's estimate to the nonnegative Gauss curvature case. The analogous degenerate case in hyperbolic space has been considered by Chang-Xiao \cite{changxiao1}, and Lin-Wang \cite{linwang1}.\\
\indent In the theory of general relativity the isometric embedding problem holds great significance in the study of quasi-local mass quantities. Brown and York in \cite{brownyork1} proposed a definition of a quasi-local mass utilizing the solution to the Weyl problem. Their quasi-local mass measures the deviation between the extrinsic mean curvature of the physical region as it sits in a three dimensional `time-symmetric slice' of spacetime compared to the isometric embedding in Euclidean space. Important work by Shi-Tam in \cite{shitam1} was able to establish non-negativity of the Brown-York mass. Liu-Yau in \cite{liuyau1}, \cite{liuyau2} introduced a natural extension to the Brown-York counterpart that incorporated general time-slices in spacetime and were also able to show non-negativity, as would be expected of a meaningful mass quantity. A further generalization using a Hamilton-Jacobi formalism led Wang-Yau to define another quasi-local mass \cite{wangyau1},\cite{wangyau2},\cite{wangyau3}. Similarly to the Euclidean space reference of Brown-York, the Wang-Yau mass considers general co-dimension two spacelike regions in spacetime and utilizes the Weyl problem to find isometric embeddings within the flat Minkowski spacetime as reference. Non-negativity for the Wang-Yau mass was also shown in \cite{wangyau1}, along with many other physically important observations due to Chen-Wang-Yau \cite{chenwangyau1},\cite{chenwangyau2},\cite{chenwangyau3},\cite{chenwangyau4},\cite{chenwangyau5},\cite{chenwangyau6}, and Chen-M.Wang-Y.Wang-Yau \cite{chenwang2yau1},\cite{chenwang2yau2},\cite{chenwang2yau3}. As highlighted in the case of the Wang-Yau mass, investigations into the Weyl problem within more exotic ambient backgrounds is of particular interest.\\
\indent The setting of this paper is to initially study the isometric embedding problem within more general ambient geometries. Specifically, we concentrate particularly upon a null three dimensional ambient geometry, called a \textit{Null Cone}. Our analysis develops from recent work of Li-Wang \cite{li2020}, where the authors observe a dual relation between the linearized isometric embedding system and the homogenous linearized Gauss-Codazzi system. Li-Wang also provides a novel use of the maximum principle to obtain uniqueness of solutions, from which Fredholm theory provides solvability of the linearized problem. Their work provides an openness argument independent of any ambient Riemannian space and also without the requirement of infinitesimal rigidity as in earlier work of Polgorelov \cite{pogorelov2},\cite{pogorelov3}. Our interest in studying the isometric embedding problem in Null Cones similarly derives from a study of quasi-local energy. One way to understand the notion of total energy is to consider a hypersurface in spacetime ruled by null geodesics, representative of a congruence of light rays in general relativity. This congruence forms the Null Cone geometry. Similarly to the Liu-Yau mass, if any spherical cross-section $\Sigma\cong \mathbb{S}^2$ of this congruence is considered, the resulting co-dimension two surface exhibits an extrinsic mean curvature vector that one can use to define the Hawking Energy, $E_H(\Sigma)$. Taking a limit of this quasi-local energy along 2-spheres foliating a Null Cone then relates to the total energy of the spacetime whenever the induced metrics asymptotically rescale to a metric of constant Gauss curvature. We refer the reader to the works of Christodoulou-Klainermann \cite{christodoulou2014global}, Bieri \cite{bieri2010}, also Klainerman-Nicol\`{o} \cite{klainerman2003evolution}, Chru\'{s}ciel-Paetz \cite{chrusciel2014mass}, and Mars-Soria \cite{MS1}. From this limit, one can define the notion of the total energy of a Null Cone $\Omega$ called the Trautman-Bondi Energy, $E_{TB}(\Omega)$. This energy can be traced back to Trautman \cite{Trautman:1958zz}, predating, a coordinate based construction by Bondi et al. \cite{bondi1962gravitational,sachs1962gravitational}.  We refer the reader to \cite{bieri2016future} for more details, including how the Trautman-Bondi mass relates to the famous one of Arnowitt-Deser-Misner \cite{arnowitt2008republication} at spacelike infinity. 
\\
\indent In the first part of this paper, we establish the openness for a strictly convex surface in arbitrary ambient geometries. We do so by observing the openness argument of Li-Wang holds independently of the signature or degeneracy of the metric orthogonal to an embedded sphere. We also show that the elliptic operator employed in the Li-Wang result in dimension two can be transformed into a self-adjoint elliptic operator acting on covariant vectors fields, closely related to the transformation of the conformal Killing operator into the conformal vector Laplace operator. We are then able to solve the linearized problem up to a solution of the homogenous Codazzi equation in any dimension, on any compact manifold. In the second part of the paper we show, given an initial isometric embedding, any given path of metrics connected to the initial one admits small parameter existence and uniqueness of a corresponding path of isometric embeddings within a convex Null Cone. In the final part, we use this to show that any metric on a 2-sphere can be isometrically embedded within a Null Cone up to a scaling. Moreover, any convex Null Cone admits a foliation in a neighborhood of infinity by a given metric up to rescaling.
\subsection{Initial Setup and Main Results}
We assume the existence of a $C^1$ mapping $F: I\times \mathbb{S}^n\to \mathbb{R}^{n+1}$, for an interval $I\subset\mathbb{R}$, $0\in I$. We also assume the existence of a smooth 2-tensor $\sigma$ as a ``metric". Denoting by $r:=F|_{\{0\}\times\mathbb{S}^n}$, we assume an embedding $r:\mathbb{S}^n\to\mathbb{R}^{n+1}$ such that $\sigma(dF(\partial_t),X)=0$ for any $X\in \Gamma(r_\star(T\mathbb{S}^n))$. It will be helpful to identify $dF(\partial_t)|_{\{0\}\times\mathbb{S}^n}$ via the canonical isomorphism $T_p\mathbb{R}^{n+1}\to \mathbb{R}^{n+1}$ at any given $p\in \mathbb{R}^{n+1}$ with a function $\vec{\upsilon}:\mathbb{S}^n\to\mathbb{R}^{n+1}$. We will make no assumption on the sign of $\sigma(\vec{\upsilon},\vec{\upsilon})$. We also obtain the induced metric $\gamma:=r^\star(\sigma)$, and induced second fundamental form $h = \pounds_{\partial_t}(F^\star(\sigma))|_{\{0\}\times\mathbb{S}^n}$.\\\\ Throughout this paper, for $0<\beta<1$, $\mathfrak{n}\in \mathbb{N}_+$, we define the following index $(\mathfrak{n},\beta)$-H\"{o}lder spaces with the usual norms induced by the standard metric on $\mathbb{S}^n$:
\begin{itemize}
\item $C^{\mathfrak{n},\beta}(\mathbb{S}^n,\mathbb{R}^{n+1})$: the H\"{o}lder space of functions $\mathbb{S}^n\to\mathbb{R}^{n+1}$,
\item $C^{\mathfrak{n},\beta}(T\mathbb{S}^n)$: the H\"{o}lder space of sections of contravariant vector fields on $\mathbb{S}^n$, 
\item $C^{\mathfrak{n},\beta}(T^\star\mathbb{S}^n)$: the H\"{o}lder space of sections  of covariant vector fields on $\mathbb{S}^n$, 
\item $C^{\mathfrak{n},\beta}(\text{Sym}(T^\star\mathbb{S}^n\otimes T^\star\mathbb{S}^n))$: the H\"{o}lder space of sections of symmetric covariant 2-tensors on $\mathbb{S}^n$.
\end{itemize}
For convenience, we will denote by $||\cdot||_{\mathfrak{n},\beta}$ the norm as applied to any element of a given H\"{o}lder space above. When used, it will be clear from context the particular space in question.
We're ready to state our first result in Section 2.3, Theorem \ref{t1}: 
\begin{theorem}\label{T1}
Consider, 
\begin{align*}
	h\in C^{2,\beta}(\text{Sym}(T^\star\mathbb{S}^n\otimes T^\star\mathbb{S}^n)),\,\,\gamma&\in C^{3,\beta}(\text{Sym}(T^\star\mathbb{S}^n\otimes T^\star\mathbb{S}^n)),\,\,\vec{\upsilon}\in C^{2,\beta}(\mathbb{S}^n,\mathbb{R}^{n+1}).
\end{align*}
Then, provided $\mathcal{P}(h,h)>0$, there exists $\epsilon>0$ such that any $\tilde\gamma\in C^{2,\beta}(\text{Sym}(T^\star\mathbb{S}^n\otimes T^\star\mathbb{S}^n))$ satisfying $||\gamma-\tilde\gamma||_{2,\beta}\leq \epsilon$
admits a unique $a\in C^{1,\beta}(\text{Sym}(T^\star\mathbb{S}^n\otimes T^\star\mathbb{S}^n))$ satisfying $\mathcal{P}(a,h) = 0$, and $\nabla\cdot a = d\tr_\gamma a$, whereby the metric $\tilde\gamma+a$ is isometrically embedded in $(\mathbb{R}^{n+1},\sigma)$.
\end{theorem}
Here, $\nabla$ refers to the Levi-Civita covariant derivative induced on $\mathbb{S}^n$ by $\gamma$. The bilinear form $\mathcal{P}$ acting on $\text{Sym}(T^\star\mathbb{S}^n\otimes T^\star\mathbb{S}^n)$ is given by:
\begin{definition}\label{d1}
Given a compact manifold $\Sigma^n$ with metric $\gamma$, $\mathcal{P}$ is the bilinear form acting on tensors $a,b\in\text{Sym}(T_p^\star\Sigma\otimes T_p^\star\Sigma)$ according to:
$$\mathcal{P}(a,b) := \tr_\gamma a\tr_\gamma b - \langle a,b\rangle$$
whereby, $\langle \cdot,\cdot\rangle$ represents the metric induced on $\text{Sym}(T^\star\Sigma\otimes T^\star\Sigma)$ by $\gamma$.	\end{definition}
Using an appropriate local orthonormal frame for the metric $\gamma$ at a point $p\in \Sigma$, we may realize some $h_p\in \text{Sym}(T_p^\star\Sigma\otimes T_p^\star\Sigma) $ as $h_p = \text{diag}(\kappa_1,\cdots,\kappa_n)$. We observe:
$$\mathcal{P}(h_p,h_p) = 2\sum_{i<j}\kappa_i\kappa_j.$$
Regarding the $\mathbb{S}^2$ case, Li-Wang proves (Theorem 11, \cite{li2020}) the following result. Given $h\in \Gamma(\text{Sym}(T^\star\mathbb{S}^2\otimes T^\star\mathbb{S}^2))$ such that $\mathcal{P}(h,h)>0$, then any $a\in \Gamma(\text{Sym}(T^\star\mathbb{S}^2\otimes T^\star\mathbb{S}^2))$ satisfying $\mathcal{P}(h,a) = 0$ and $\nabla\cdot a  = d\tr_\gamma a$ must necessarily be the trivial section. We will present the Li-Wang result within our context in the next section. We are then able prove the following openness result in Section 3.3 and 3.4 (Theorems \ref{t3}, \ref{t4}, and Corollary \ref{c5}):
\begin{theorem}\label{T2}
Suppose $(\mathcal{A},\sigma)$ is a convex Null Cone. For $\mathfrak{n}\geq 2$, consider also a continuous path of metrics $t\to\gamma_t\in C^{\mathfrak{n}+1,\beta}(\text{Sym}(T^\star\mathbb{S}^2\otimes T^\star\mathbb{S}^2))$, $t\in(-b,b)$, with the following properties:
\begin{enumerate}
\item There exists an isometric embedding $r_0:(\mathbb{S}^2,\gamma_0)\hookrightarrow(\mathcal{A},\sigma)$,
\item $t\to\gamma_t\in C^{\mathfrak{n},\beta}(\text{Sym}(T^\star\mathbb{S}^2\otimes T^\star\mathbb{S}^2))$ is $\mathfrak{m}$-times continuously Fr\'{e}chet differentiable within $C^{\mathfrak{n},\beta}(\text{Sym}(T^\star\mathbb{S}^2\otimes T^\star\mathbb{S}^2))$.
\end{enumerate}
Then, there exists $0<\epsilon<b$, and a family of isometric embeddings $r_t:(\mathbb{S}^2,\gamma_t)\hookrightarrow (\mathcal{A},\sigma)$, $t\in(-\epsilon,\epsilon)$, with an associated path $t\to\vec{r}_t\in C^{\mathfrak{n}+1,\beta}(\mathbb{S}^2,\mathbb{R}^3)$ that is $\mathfrak{m}$-times continuously Fr\'{e}chet differentiable within $C^{\mathfrak{n},\beta}(\mathbb{S}^2,\mathbb{R}^3)$. Moreover, $\epsilon$ is independent of $\mathfrak{n},\mathfrak{m}$, and the path is unique with first Fr\'{e}chet derivative: 
\begin{align*}
\dot{\vec{r}}_t &= d\vec{r}_t(\tau_t^\#)+\phi_t\frac{\vec{r}_t}{|\vec{r}_t|},\\
L_{h(\vec{r}_t)}(\tau_t) = \dot{\gamma}_t-&\frac12\tr_{h(\vec{r}_t)}(\dot{\gamma}_t)h(\vec{r}_t),\,\,\,\phi_t=\frac{\tr_{\gamma_t}(\dot{\gamma}_t)-2\nabla^{\gamma_t}\cdot\tau_t}{\tr_{\gamma_t}h(\vec{r}_t)},
\end{align*}
whereby $\tau_t\perp\text{Ker}(L_{h(\vec{r}_t)})$ with respect to the $L^2$ inner product induced by $\gamma_t$.
\end{theorem}
Here, convexity of a Null Cone refers to the property that a geodesic foliation of a null-cone by convex leaves exist. We will discuss this further in the final section and show how this property translates into the fact that -any- cross-section of a Null Cone is convex. The operator $L_{h(\vec{r}_t)}:C^{m,\beta}(T^\star\mathbb{S}^2)\to C^{m-1,\beta}(\text{Sym}(T^\star\mathbb{S}^2\otimes T^\star\mathbb{S}^2)))$ is the elliptic operator considered in the work of Li-Wang. It depends on the second fundamental form, $h(\vec{r}_t)\in C^{m,\beta}(\text{Sym}(T^\star\mathbb{S}^2\otimes T^\star\mathbb{S}^2))$, associated to the isometric embedding $r_t:(\mathbb{S}^2,\gamma_t)\hookrightarrow (\mathcal{A},\sigma)$. \\
\indent Using this result we will be able to complete the continuity argument toward isometrically embedding any metric up to a scale factor within a Null Cone. In order to do so we need to assume certain decay at infinity (see \ref{a1}-\ref{a4} in Section 4.2) yielding our final result in Section 4.2 and 4.3 (Theorems \ref{t5}, \ref{t6}):
\begin{theorem}\label{T3}
	Suppose $(\mathcal{A},\sigma)$ is a smooth, convex Null Cone satisfying the decay assumptions (\ref{a1}-\ref{a4}). Then, for $s$ sufficiently large, we can isometrically embed the metric $s^2\gamma$, for any metric $\gamma\in C^{\mathfrak{n},\beta}(\text{Sym}(T^\star\mathbb{S}^2\otimes T^\star\mathbb{S}^2))$, $\mathfrak{n}\geq 3$. Moreover, we can foliate $(\mathcal{A},\sigma)$ in a neighborhood of infinity with a family of isometric embeddings associated to the path $t\to s^2e^t\gamma$, $t\in[0,\infty)$, and the foliation is an asymptotically $C^{1,\beta}$ geodesic foliation.
\end{theorem}
Any null hypersurface geometry consists of a manifold with a degenerate metric. In the case of a Null Cone this degeneracy exists in the normal flow direction for a global foliation by 2-spheres with expanding area form. \textit{Asymptotic flatness} therefore refers to the decay of the intrinsic and extrinsic geometry along this foliation towards the data observed for a geodesic foliation of the standard lightcone in Minkowski space. Up to a rescaling, these leaves converge to a Riemannian structure on the 2-sphere, the regularity here is identified as $C^{1,\beta}$, see Definition \ref{d5}.
\section{The Linearized Isometric Embedding Problem}
For convenience, we will assume all tensors in this section are smooth. When dealing with less regular spaces in later sections, we will address any discrepancies whenever necessary.\\\\
\indent Given an open interval $I\subset\mathbb{R}$ such that $0\in I$, and a compact manifold $\Sigma^n$ we assume the existence of a smooth mapping (of smooth manifolds) $F:I\times\Sigma^n\to\mathcal{M}$. Moreover, for each $s\in I$, $F_s:=F(s,\cdot)$ induces an embedding $\Sigma\hookrightarrow_{F_s}\mathcal{M}$. We place no constraints on the dimension of the manifold $\mathcal{M}$, regarding the metric $g$ we only ask that $dF(\partial_s)|_{\Sigma}\in \Gamma(T^\perp\Sigma)$, where we identify $\Sigma$ with its image $ F_0(\Sigma)$, and that the induced metric $\gamma:=g|_\Sigma$ is Riemannian. For yet another open interval $J\subset\mathbb{R}$ such that $0\in J$, we assume the existence of a smooth mapping $V:J\times \Sigma\to T\Sigma$, such that $V(t,p)\in T_p\Sigma$, for $t\in J$, $p\in\Sigma$. We also take the mapping $V$ such that $V(0,\cdot)$ is the trivial section. Now for each $t\in J$, the section $V_t:=V(t,\cdot)\in\Gamma(T\Sigma)$ generates a flow $\tilde\Psi_t:\mathbb{R}\times \Sigma\to \Sigma$ such that $\lambda\to\tilde\Psi_t(\lambda,p)$ is the maximal integral curve of $V_t$ through $p \in \Sigma$, and $\tilde\Psi_t(0,p)=p$. It is well known that for each $\lambda\in\mathbb{R}$, the map $\tilde\Psi_t(\lambda,\cdot):\Sigma\to\Sigma$ is a diffeomorphism, and we will denote $\Psi(t,p):=\tilde\Psi_t(1,p)$. From standard results for systems of ODE one concludes that the mapping $\Psi: J\times \Sigma\to\Sigma$ is also smooth (we will also denote $\Psi_t(p):=\Psi(t,p)$, whereby $\Psi_t:\Sigma\to\Sigma$ is again a diffeomorphism). In a local coordinate chart $(x^i,\mathcal{U})$ of a point $p\in\Sigma$, we may use the Taylor Approximation Theorem to conclude that $\Psi^i(t,p) = t\tau^i(p)+O(t^2)$ whereby $\tau = dV(\partial_t)$, which by a natural isomorphism, may be viewed as $\tau\in\Gamma(T\Sigma)$. Finally, we will take a smooth function $\varphi: J\times \Sigma\to I$, such that $\varphi(0,\cdot) \equiv 0$, and $\phi(p) = \frac{d}{dt}|_{t=0}\varphi(t,p)$. With slight abuse of notation, denoting temporarily $\gamma = F^\star(g)$, we have in a coordinate neighborhood $(x^i,\mathcal{U})$ of some $p\in\Sigma$:
\begin{align*}
\frac{d}{dt}|_{t=0}(\Psi^\star_t\gamma)(\varphi_t,\Psi_t(x))_{ij} &=\frac{d}{dt}|_{t=0}\Big(\gamma(\varphi_t,\Psi_t(x))_{kl}\frac{\partial\Psi^k_t}{\partial x^i}\frac{\partial\Psi^l_t}{\partial x^j}\Big)\\
&=h_{ij}\phi+\gamma_{ij,k}\tau^k+\gamma_{ik}\tau^k,_j+\gamma_{jl}\tau^l,_i\\
&=h_{ij}\phi+\nabla_i\tau_j+\nabla_j\tau_i\\
&=2\text{Sym}(\nabla\tau)_{ij}+\phi h_{ij}
\end{align*}
whereby $h:=\pounds_{\partial_s}(F^\star g)|_\Sigma$. Since $dF(\partial_s)|_\Sigma\in\Gamma(T^\perp\Sigma)$, we may recognize the symmetric 2-tensor $h$ as a certain component of the second fundamental form of $\Sigma$. Given the pairing $(\gamma,h)$ on $\Sigma$, we refer to the operator
$$(\tau,\phi)\to 2\text{Sym}(\nabla\tau)+\phi h$$
as the \textit{linearized embedding} operator. Since we're ultimately interested in studying a path of metrics on $\Sigma$ for which an isometric embedding exists into $(\mathcal{M},g)$ we will need to first understand the linearized embedding operator.
\subsection{A new elliptic operator}
It turns out an elliptic operator of 1-forms on $\Sigma$ can be constructed that we will use to characterize the image of the linearized embedding operator. In order to construct this operator we will need the symmetric bilinear form $\mathcal{P}$ of Definition \ref{d1}. Before constructing this operator it will be helpful to identify a sub-bundle of $\text{Sym}(T^\star\Sigma\otimes T^\star\Sigma)$ related to the 2-tensor $h$. We can partially motivate this sub-bundle by assuming the triple $(\Sigma,\gamma,h)$ can be isometrically embedded in $\mathbb{R}^{n+1}$, whereby $h$ corresponds with the induced second fundamental form. Then, from the Gauss equation (see, for example \cite{O}, pg. 100), one identifies the scalar curvature of $\Sigma\subset \mathbb{R}^{n+1}$:
$$R = (\tr_\gamma h)^2 - |h|_\gamma^2.$$
If we take a path of isometries $t\to F_t$, $F_t:\Sigma\hookrightarrow\mathbb{R}^{n+1}$, $F_0 = id$, then we observe a variation of the second fundamental form $a:=\pounds_{\partial_t}h$, satisfying:
$$\mathcal{P}(a,h) = (\tr_\gamma a)(\tr_\gamma h) - \langle a, h\rangle = 0$$
on $\Sigma$. So, with respect to some given $h\in\Gamma(\text{Sym}(T^\star\Sigma\otimes T^\star\Sigma))$ satisfying $\mathcal{P}(h,h)>0$, we will denote the bundle of pointwise $h$-orthogonal symmetric 2-tensors with respect to $\mathcal{P}$ by 
$$\Xi_h:=\{a\in\text{Sym}(T^\star\Sigma\otimes T^\star\Sigma)|\mathcal{P}(a,h)=0\}.$$
\begin{lemma}\label{l3}
$-\mathcal{P}$ is a Lorentzian metric for the vector bundle $\text{Sym}(T^\star\Sigma\otimes T^\star\Sigma)$. 
\end{lemma}
\begin{proof}
We start by taking any symmetric 2-tensor $h$ such that $\mathcal{P}(h,h)>0$ (for example $\gamma$ itself will do). It suffices therefore to show that $-\mathcal{P}$ is an inner product on $\Xi_h|_p$. To do so, we start by noticing that $\mathcal{P}(h,h)>0$ implies $|\frac{h}{\tr_\gamma h}|^2<1$, and $\mathcal{P}(a,h) = 0$ is equivalent to $\tr_\gamma a = \langle a,\frac{h}{\tr_\gamma h}\rangle$. Therefore,
$$\mathcal{P}(a,a) = (\tr_\gamma a)^2-|a|^2 = \langle a,\frac{h}{\tr_\gamma h}\rangle^2-|a|^2\leq|a|^2\Big(|\frac{h}{\tr_\gamma h}|^2-1\Big)\leq 0,$$
having used the Cauchy-Schwartz inequality to observe the first inequality. We also observe from the final inequality that equality is achieved if and only if $a=0$. 

\end{proof}
We will be able to construct an elliptic operator on 1-forms, $\Gamma(T^\star\Sigma)$, for any symmetric 2-tensor $h\in\Gamma(\text{Sym}(T^\star\Sigma\otimes T^\star\Sigma))$ satisfying $\mathcal{P}(h,h)>0$. We start by defining the operator $L_h:\Gamma(T^\star\Sigma)\to \Gamma(\Xi_h)$, given by
$$L_h(\tau) = 2\text{Sym}(\nabla\tau)-2\frac{\mathcal{P}(\text{Sym}(\nabla\tau),h)}{\mathcal{P}(h,h)}h,$$
noticing that this is precisely the $h$-orthogonal $\mathcal{P}$ projection
of $\text{Sym}(\nabla\tau)$. In the $n=2$ case, which we consider in the next section, $L_h$ is exactly the operator considered by Li-Wang (see (3.4), pg11 of \cite{li2020}). \\
\indent With $L_h$ in hand, we define the operator $\mathcal{L}_h:\Gamma(T^\star\Sigma)\to\Gamma(T^\star\Sigma)$:
$$\mathcal{L}_h(\tau) := \nabla\cdot \big(L_h(\tau)-\tr_\gamma (L_h(\tau))\gamma\big).$$
\begin{remark}
 We note that $\mathcal{P}(\gamma,a) = n\tr_\gamma a-\tr_\gamma a = (n-1)\tr_\gamma a$ for any symmetric 2-tensor $a\in\Gamma(\text{Sym}(T^\star\Sigma\otimes T^\star\Sigma))$. It follows that
$$L_\gamma(\tau) = 2\text{Sym}(\nabla\tau)-\frac{2}{n}(\nabla\cdot \tau) \gamma,$$
also known as the conformal Killing operator with respect to the metric $\gamma$. The 1-forms $\tau\in\text{Ker}(L_\gamma)$ are metrically equivalent to the global conformal Killing fields if $\Sigma$. Specifically, the flow diffeomorphism $\Psi_t:\Sigma\to\Sigma$ generated by $\tau\in\text{Ker}(L_\gamma)$ induce the conformal diffeomorphisms, $\Psi_t^\star(\gamma) = e^{2u_t}\gamma$, of $\Sigma$ connected to the identity. $\mathcal{L}_\gamma$ is metrically equivalent to the conformal vector Laplace operator on $\Gamma(T\Sigma)$.
\end{remark}
\begin{proposition}\label{p2}
Suppose $h\in\Gamma(\text{Sym}(T^\star\Sigma\otimes T^\star\Sigma))$ satisfies $\mathcal{P}(h,h)>0$. Then
\begin{enumerate}
\item $\mathcal{L}_h$ is an elliptic operator
\item $\mathcal{L}_h$ is a self adjoint operator with respect to the $L^2$ inner product induced by $\gamma$  
\item $\text{Ker}(\mathcal{L}_h) = \text{Ker}(L_h)$
\end{enumerate}
\end{proposition}
\begin{proof}
The principal symbol of $\mathcal{L}_h$ takes the form:
\begin{align*}
\sigma_\xi(\mathcal{L}_h)(\tau) &= \\
\Big(|\xi|^2\tau+\langle\tau,\xi\rangle\xi&-2\frac{(\tr_\gamma h)\langle \xi,\tau\rangle-h(\xi,\tau)}{\mathcal{P}(h,h)}h\Big)-\Big(2\langle\xi,\tau\rangle -2\frac{(\tr_\gamma h)\langle\xi,\tau\rangle-h(\xi,\tau)}{\mathcal{P}(h,h)}\tr_\gamma h\Big)\xi.
\end{align*}
We therefore observe, 
$$\langle\sigma_\xi(\mathcal{L}_h)(\tau), \tau\rangle=\Big(|\xi|^2|\tau|^2-\langle \tau,\xi\rangle^2\Big)+\frac{2}{\mathcal{P}(h,h)}\Big(h(\xi,\tau)-\tr_\gamma h\langle \xi,\tau\rangle\Big)^2,$$
and we notice both terms within parentheses are independently non-negative. For ellipticity, it suffices to show $\sigma(\mathcal{L}_h)(\xi)$ is an injective linear mapping whenever $\xi\neq0$, equivalently, $\sigma_\xi(\mathcal{L}_h)(\tau)=0$ implies $\tau=0$. Clearly the assumption $\sigma_\xi(\mathcal{L}_h)(\tau)=0$ implies $\tau = c\xi$ for some constant $c$ as a result of the first pair of parentheses above. We conclude from the second pair of parentheses that $c(h(\xi,\xi)-\tr_\gamma h|\xi|^2)=0$. From $\mathcal{P}(h,h)>0$ we may conclude (up to an appropriate sign change) that $\tr_\gamma h>0$. Therefore, from Cauchy-Schwartz it follows that 
$$h(\xi,\xi)-\tr_\gamma h|\xi|^2\leq (|h|-\tr_\gamma h)|\xi|^2<0$$
forcing $c=0$ as desired.\\\\
\indent For the second property we integrate over $\Sigma$ using the induced area form coming from $\gamma$, using the divergence theorem: 
\begin{align*}
\int\langle\mathcal{L}_h(\tau),\eta\rangle dA &= \int-\langle L_h(\tau),\nabla\eta\rangle+(\tr_\gamma L_h(\tau))(\nabla\cdot\eta) dA\\
&=-\int \mathcal{P}(L_h(\tau),\text{Sym}(\nabla\eta)) dA\\
&=-\frac12\int \mathcal{P}(L_h(\tau),L_h(\eta))dA.
\end{align*}
We immediately conclude from the permutation symmetry between $(\tau,\eta)$ within this last term that $\int\langle\mathcal{L}_h(\tau),\eta\rangle dA = \int\langle\tau,\mathcal{L}_h(\eta)\rangle dA$ and therefore that $\mathcal{L}_h$ is self adjoint.\\\\
\indent For the third property, $\text{Ker}(L_h)\subset \text{Ker}(\mathcal{L}_h)$ follows trivially. For the reverse inclusion we observe
$$\int \langle\mathcal{L}_h(\tau),\tau\rangle dA = -\frac12\int\mathcal{P}(L_h(\tau),L_h(\tau))dA.$$
Therefore, for any $\tau\in\text{Ker}(\mathcal{L}_h)$ we observe that $\int\mathcal{P}(L_h(\tau),L_h(\tau))dA = 0$, and we use Lemma \ref{l3} to conclude that $\mathcal{P}(L_h(\tau),L_h(\tau))\equiv 0$, thus $L_h(\tau)= 0$. 
\end{proof}
We now show that the collection of symmetric 2-tensors given by 
$$\mathcal{S}:=\{h\in\Gamma(\text{Sym}(T^\star\Sigma\otimes T^\star\Sigma))|\mathcal{P}(h,h)>0\}$$
consists of two disjoint open connected components. To show this, we start by decomposing an $h$ into its $\gamma$-linear and $\gamma$-orthogonal components with respect to $\mathcal{P}$:
$$h = \frac{\tr_\gamma h}{n}\gamma+\Big(h-\frac{\tr_\gamma h}{n}\gamma\Big)=: \frac{1}{n}(\tr_\gamma h)\gamma + \hat{h}.$$
We notice that this corresponds with the decomposition of $h$ into its trace and traceless (or tracefree) parts with respect to $\gamma$. We may now construct a path within $\mathcal{S}$, $t\to h_t\in \Gamma(\text{Sym}(T^\star\Sigma\otimes T^\star\Sigma))$, from $h=h_0$ to a multiple of $\gamma$ at $t=1$, such that $\mathcal{P}(h,h)=\mathcal{P}(h_t,h_t)$. One example is the path given by: 
 \[t\to \begin{cases} 
      (a\cosh(ct)-b\sinh(ct))e_0+(b\cosh(ct)-a\sinh(ct))e_1, & b> 0 \\
      a\gamma & b=0 
   \end{cases}
\]

whereby, $|\hat{h}|>0$ gives:
$$e_0:=\frac{\gamma}{\sqrt{n(n-1)}},\,e_1:=\frac{\hat{h}}{|\hat{h}|},\,a = \mathcal{P}(e_0,h) = \sqrt{\frac{n-1}{n}}(\tr_\gamma h),\,b=-\mathcal{P}(e_1,\hat{h}) = |\hat{h}|,\,c=\text{arc}\tanh(\frac{b}{a}).$$

The sign of $\tr_\gamma h=\frac{1}{n-1}\mathcal{P}(\gamma,h)$ then indicates whether $h$ is path connected to 
$\gamma$ or $-\gamma$, giving the disjoint union $\mathcal{S} = \mathcal{S}_+\cup\mathcal{S}_-$ respectively. We highlight that the analysis above is not specific to $\gamma$, one can similarly with the choice $e_0^h: = \frac{h}{\sqrt{\mathcal{P}(h,h)}}$ construct a path from any $\tilde h\in\mathcal{S}$ to $h\in\mathcal{S}$ fixing $\mathcal{P}({\tilde h},{\tilde h})$. Similarly, we conclude that $h,\tilde h\in\mathcal{S}$ occupy the same connected component if and only if $\mathcal{P}(h,\tilde h)>0$.
\begin{remark}
Since the symbol of $\mathcal{L}_h$ at a given $p\in\Sigma$ relates to that of $L_h$ via:
$$\sigma_\xi(\mathcal{L}_h)(\tau) = \langle \xi,\sigma_\xi(L_h)(\tau)\rangle-\tr_\gamma(\sigma_\xi(L_h)(\tau))\xi,$$
we note that the linear map $\sigma_\xi(\mathcal{L}_h)$ is a composition of $\sigma_\xi(L_h)$ by yet another linear mapping. Therefore, from the injectivity of $\sigma_\xi(\mathcal{L}_h)$ we have that $\sigma_\xi(L_h)$ is injective also.
\end{remark}
It will be helpful to denote the subset $\ubar\Xi_h\subset\Gamma(\Xi_h)$:
$$\ubar\Xi_h:=\{a\in\Gamma(\Xi_h)|\nabla\cdot a = d\tr_\gamma a\}.$$
\begin{proposition}\label{p3}
Suppose $h\in\Gamma(\text{Sym}(T^\star\Sigma\otimes T^\star\Sigma))$ satisfies $\mathcal{P}(h,h)>0$. Then, for every $q\in\Gamma(\Xi_h)$, we have a unique decomposition
$$q = L_h(\tau)+a$$
such that $a\in \ubar\Xi_h$. Moreover, $\{a,L_h(\tau)\}$ are mutually orthogonal with respect to the $L^2$ inner product induced by $\mathcal{P}$, namely $\int\mathcal{P}(L_h(t),a)dA = 0$.
\end{proposition}
\begin{proof}
Since $\mathcal{L}_h$ is a strong elliptic operator on a compact manifold, it is a Fredholm operator and the Fredholm Alternative applies. We conclude that a solution, $\tau$, for the equation
$$\mathcal{L}_h(\tau) = \eta$$
exists provided $\int\langle\eta,\mu\rangle dA=0$ whenever $\mu\in\text{Ker}(\mathcal{L}^\star_h)=\text{Ker}(\mathcal{L}_h)$, where $\mathcal{L}^\star_h$ denotes adjoint of $\mathcal{L}_h$ relative to the $L^2$ inner product induced by $\gamma$. Via standard elliptic regularity results, the regularity of $\mu$ will dictate the regularity of our solution $\tau$ provided it exists (see Chapter 5 ,\cite{giamart}). For the equation:
$$\mathcal{L}_h(\tau) = \nabla\cdot(q-(\tr_\gamma q)\gamma),$$
we know a smooth solution exists, since for any $\mu\in\text{Ker}(\mathcal{L}_h)=\text{Ker}(L_h)$, we observe
\begin{align*}
\int\langle \nabla\cdot (q - (\tr_\gamma q)\gamma),\mu\rangle dA&=\int -\langle q-(\tr_\gamma q)\gamma,\nabla\mu\rangle dA\\
&=-\int\langle q-(\tr_\gamma q)\gamma,\text{Sym}(\nabla\mu)\rangle dA\\
&=-\int\mathcal{P}(q,\text{Sym}(\nabla\mu))dA\\
&=-\frac12\int\mathcal{P}(q,L_h(\mu))dA \\
&=0.
\end{align*}
We conclude therefore that $\nabla\cdot\big((L_h(\tau)-q)-\tr_\gamma(L_h(\tau)-q)\gamma\big)=0$. So taking $a:=q-L_h(\tau)$ we've shown the desired decomposition. Substituting $q\to a$ in our calculation above, we also immediately observe that $\int\mathcal{P}(a,L_h(\tau))dA=-2\int\langle \nabla\cdot a-d\tr_\gamma a,\tau\rangle dA = 0$. All that remains is to prove uniqueness. We assume the existence of 1-forms $\tau,\tau'\in \Gamma(T^\star\Sigma)$, and symmetric 2-tensors $a,a'\in\ubar\Xi_h$, such that
$$L_h(\tau-\tau') = a-a'.$$
It follows that $\mathcal{L}_h(\tau-\tau') = \nabla\cdot (a-a')-d\tr_\gamma (a-a')=0$ meaning $\tau-\tau'\in\text{Ker}(\mathcal{L}_h) = \text{Ker}(L_h)$. Therefore $L_h(\tau-\tau')=0$, and we conclude $a=a'$ as desired.
\end{proof}
Returning to the linearized embedding equation, we observe that solving for the pair $(\tau,\phi)$ given a symmetric 2-tensor $\tilde q\in\Gamma(\text{Sym}(T^\star\Sigma\otimes T^\star\Sigma))$ in
\begin{equation}
2\text{Sym}(\nabla\tau)+\phi h = \tilde q
\end{equation}
is equivalent to solving for $\tau$ in the equation
\begin{equation}\label{e6}
L_h(\tau) = q
\end{equation}
whereby $q = \tilde q-\frac{\mathcal{P}(\tilde q,h)}{\mathcal{P}(h,h)}h$, and $\phi = \frac{\mathcal{P}(\tilde q,h)}{\mathcal{P}(h,h)}-2\frac{\mathcal{P}(\nabla\tau,h)}{\mathcal{P}(h,h)}=\frac{\tr_\gamma \tilde q}{\tr_\gamma h}-2\nabla\cdot\tau$. So from Proposition \ref{p3}, solvability of the linearized embedding equation follows for any $q\in\Gamma(\Xi_h)$ satisfying $\int\mathcal{P}(q,a)dA=0$ for any $a\in\ubar\Xi_h$ (i.e. $q\in\ubar\Xi_h^\perp$ with respect to the $L^2$ inner product induced by $\mathcal{P}$).
\begin{theorem}\label{t0}
Suppose $h\in\Gamma(\text{Sym}(T^\star\Sigma\otimes T^\star\Sigma))$ satisfies $\mathcal{P}(h,h)>0$, then:
$$\dim(\text{Ker}(L_h)) = \dim(\text{Ker}(L_\gamma)).$$
\end{theorem}
\begin{proof}
Since the operators $\mathcal{L}_h$, $\mathcal{L}_\gamma$ are Fredholm we know their kernels $\text{Ker}(L_h)$, $\text{Ker}(L_\gamma)$ are finite dimensional vector spaces. We will construct a linear isomorphism between them. We start by choosing some smooth path, $t\to h_t\in \Gamma(\text{Sym}(T^\star\Sigma\otimes T^\star\Sigma))$, connecting $h$ to a multiple of $\gamma$. Although not necessary, it will be more convenient to take the path $t\to h_t$ such that $\mathcal{P}(h_t,h_t)$ remains fixed. If we denote by $a_t:=\frac{d}{dt}h_t$, then along this path we have $2\mathcal{P}(a_t,h_t) = \frac{d}{dt}\mathcal{P}(h_t,h_t) = 0$, giving $a_t\in \Xi_{h_t}$. Given initial data, $\mu\in\Gamma(T^\star\Sigma)$, we now wish to show existence of a unique path $t\to\tau_t\in \Gamma(T^\star\Sigma)$, $t\in[0,1]$, that solves the initial value problem:
\begin{align*}
	\frac{d}{dt}\tau_t &= \eta_t,\,\,\, \mathcal{L}_{h_t}(\eta_t) = \nabla\cdot\big(2\frac{\mathcal{P}(\text{Sym}(\nabla\tau_t),h_t)}{\mathcal{P}(h,h)}(a_t - (\tr_\gamma a_t)\gamma)\big),\\
	\tau_0 &= \mu,\,\,\,\eta_t\perp\text{Ker}(\mathcal{L}_{h_t}),
\end{align*}
where again, orthogonality is with respect to the $L^2$ inner product induced by $\gamma$. To this end, we will consider the complete metric space of continuous paths $t\to\tau_t\in C^{2,\beta}(T^\star\Sigma)$, $t\in[0,\delta]$, with distance function:
$$||\tau_1-\tau_2||_{\infty,\delta}:=\sup_{t\in[0,\delta]}||(\tau_1)_t-(\tau_2)_t||_{2,\beta}.$$
Considering $\mu$ as a continuous path (independent of $t$), we wish to construct a contraction mapping acting on the closed ball $\mathcal{B}_\epsilon(\mu,\delta):=\{\tau\big|||\tau-\mu||_{\infty,\delta}\leq\epsilon\}$ for sufficiently small $\delta$. For arbitrary $\tau_t$, the mapping we will consider is given by $\mathcal{C}(\tau)_t := \mu+\int_0^t\mathcal{C}_0(\tau)_udu$, whereby $\mathcal{C}_0(\tau)_t$, for each $t\in[0,\delta]$, is given by the unique solution to the equation:
$$\mathcal{L}_{h_t}(\mathcal{C}_0(\tau)_t) = \nabla\cdot\big(2\frac{\mathcal{P}(\text{Sym}(\nabla\tau_t),h_t)}{\mathcal{P}(h,h)}(a_t - (\tr_\gamma a_t)\gamma)\big),\,\,\,\mathcal{C}_0(\tau)_t\perp\text{Ker}(\mathcal{L}_{h_t}).$$
\indent Generally, from interior Schauder estimates (see, \cite{giamart} Chapter 5), we have:
$$||\mathcal{C}_0(\tau)_t||_{2,\beta}\leq C(||h_t||_{1,\beta},||\gamma||_{2,\beta})(||\mathcal{L}_{h_t}(\mathcal{C}_0(\tau)_t)||_{0,\beta}+||\mathcal{C}(\tau)_t||_{0,\beta}).$$
Taking $\mathcal{C}_0(\tau)_t\perp\text{Ker}(\mathcal{L}_{h_t})$, a standard compactness argument then gives:
$$||\mathcal{C}_0(\tau)_t||_{2,\beta}\leq C(||h_t||_{1,\beta},||\gamma||_{2,\beta})||\mathcal{L}_{h_t}(\mathcal{C}_0(\tau)_t)||_{0,\beta}.$$
To show this we argue for a contradiction. Assuming this estimate does not hold, then for each $n\in\mathbb{N}_+$ we can find $\nu_n\in \text{Ker}(\mathcal{L}_{h_t})^\perp$ such that $\frac{1}{n}||\nu_n||_{2,\beta}>||\mathcal{L}_{h_t}(\nu_n)||_{0,\beta}$, so that the bounded sequence $\tilde\nu_n:=\frac{\nu_n}{||\nu_n||_{2,\beta}}\in C^{2,\beta}(T^\star\Sigma)$ causes $\mathcal{L}_{h_t}(\tilde\nu_n)\to 0$ in the $C^{0,\beta}(T^\star\Sigma)$ norm, as $n\to\infty$. By the Arz\`{e}la-Ascoli theorem, we know $C^{2,\beta}(T^\star\Sigma)$ is compactly embedded within $C^{0,\beta}(T^\star\Sigma)$. Specifically, up-to a re-indexing, we may assume $\{\tilde\nu_n\}_n$ converges in $C^{0,\beta}(T^\star\Sigma)$. Again from Schauder estimates, we have:
$$||\tilde\nu_n-\tilde\nu_m||_{2,\beta}\leq C(||\mathcal{L}_{h_t}(\tilde\nu_n-\tilde\nu_m)||_{0,\beta}+||\tilde\nu_n-\tilde\nu_m||_{0,\beta})\leq C(\frac{1}{n}+\frac{1}{m}+||\tilde\nu_n-\tilde\nu_m||_{0,\beta}),$$
so that $\{\tilde\nu_n\}_n$ is a Cauchy sequence in $C^{2,\beta}(T^\star\Sigma)$ and therefore also converges within $C^{2,\beta}(T^\star\Sigma)$, a Banach space. We conclude that the limit $\tilde\nu_\infty\in C^{2,\beta}(T^\star\Sigma)$ satisfies $||\tilde\nu_\infty||_{2,\beta}=1$, and $\mathcal{L}_{h_t}(\tilde\nu_\infty) = 0$. Since $\text{Ker}(\mathcal{L}_{h_t})^\perp$ is closed, we have $\tilde\nu_\infty\in\text{Ker}(\mathcal{L}_{h_t})^\perp\cap\text{Ker}(\mathcal{L}_{h_t})=\{0\}$ yet $||\tilde\nu_\infty||_{2,\beta} = 1$, yielding the desired contradiction.
\begin{lemma}\label{l0}
We obtain a continuous path $t\to\mathcal{C}_0(\tau)_t\in C^{2,\beta}(T^\star\Sigma)\cap\text{Ker}(\mathcal{L}_{h_t})^\perp$ from solving the elliptic equation:
$$\mathcal{L}_{h_t}(\mathcal{C}_0(\tau)_t) = \nabla\cdot\big(2\frac{\mathcal{P}(\text{Sym}(\nabla\tau_t),h_t)}{\mathcal{P}(h,h)}(a_t - (\tr_\gamma a_t)\gamma)\big),$$
for each $t$.
\end{lemma}
\begin{proof}
For $t_0\in[0,\delta]$, we wish to show $\displaystyle{\lim_{t\to t_0}||\mathcal{C}_0(\tau)_{t_0}-\mathcal{C}_0(\tau)_t||_{2,\beta} = 0}$. In order to do so, we denote by $\pi_t$ the projection map onto $\text{Ker}(\mathcal{L}_{h_t})$. It follows:
\begin{align*}
\mathcal{L}_{h_{t_0}}\Big(\mathcal{C}_0(\tau)_{t_0}-\big(\mathcal{C}_0(\tau)_t-\pi_{t_0}\mathcal{C}_0(\tau)_t\big)\Big)&=\mathcal{L}_{h_{t_0}}(\mathcal{C}_0(\tau)_{t_0})-\mathcal{L}_{h_{t}}(\mathcal{C}_0(\tau)_t)+\Big(\mathcal{L}_{h_{t}}-\mathcal{L}_{h_{t_0}}\Big)(\mathcal{C}_0(\tau)_{t})\\
&= \nabla\cdot\big(2\frac{\mathcal{P}\big(\text{Sym}(\nabla(\tau_{t_0}-\tau_t)),h_{t_0}\big)}{\mathcal{P}(h,h)}(a_{t_0} - (\tr_\gamma a_{t_0})\gamma)\big)\\
&+\nabla\cdot\big(2\frac{\mathcal{P}(\text{Sym}(\nabla\tau_t),h_{t_0}-h_t)}{\mathcal{P}(h,h)}(a_{t_0} - (\tr_\gamma a_{t_0})\gamma)\big)\\
&+\nabla\cdot\big(2\frac{\mathcal{P}(\text{Sym}(\nabla\tau_t),h_t)}{\mathcal{P}(h,h)}((a_t-a_{t_0}) - (\tr_\gamma (a_t-a_{t_0}))\gamma)\big)\\
&+ \Big(\mathcal{L}_{h_{t}}-\mathcal{L}_{h_{t_0}}\Big)(\mathcal{C}_0(\tau)_{t}).
\end{align*}
From continuity and $C^{2,\beta}$ boundedness of $\{\tau_t,h_t,a_t\}$, we therefore conclude from Schauder estimates that $|| \mathcal{C}_0(\tau)_{t_0}-\big(\mathcal{C}_0(\tau)_t-\pi_{t_0}\mathcal{C}_0(\tau)_t\big) ||_{2,\beta}\to 0$ as $t\to t_0$. Choosing a (smooth) basis $\{\nu_i\}_i\subset \text{Ker}(\mathcal{L}_{h_{t_0}})$, with unit $L^2$ norm, we decompose $\pi_{t_0}\mathcal{C}_0(\tau)_t = \sum_i\langle\nu_i,\mathcal{C}_0(\tau)_t\rangle_{L^2}\nu_i=\sum_i\langle\nu_i-\pi_{t}(\nu_i),\mathcal{C}_0(\tau)_t\rangle_{L^2}\nu_i$, for $\langle\cdot,\cdot\rangle_{L^2}$ the $L^2$ inner product induced by $\gamma$. We also observe,
$$\mathcal{L}_{h_t}(\nu_i-\pi_t(\nu_i))=\mathcal{L}_{h_t}(\nu_i) = \Big(\mathcal{L}_{h_t}-\mathcal{L}_{h_{t_0}}\Big)(\nu_i)$$
giving, via Schauder estimates, that $||\nu_i-\pi_t(\nu_i)||_{2,\beta}\to 0$, as $t\to t_0$. We therefore see:
\begin{align*}
||\pi_{t_0}\mathcal{C}_0(\tau)_t||_{2,\beta}&\leq\sum_i\Big|\langle\nu_i-\pi_t(\nu_i),\mathcal{C}_0(\tau)_t\rangle_{L^2}\Big|||\nu_i||_{2,\beta}\\
&\leq C\Big(\sum_i||\nu_i-\pi_t(\nu_i)||_{2,\beta}||\nu_i||_{2,\beta}\Big)||\mathcal{C}_0(\tau)_t||_{2,\beta},
\end{align*}
having used the Cauchy-Schwarz inequality, and compactness of $\Sigma$ in  obtaining the second inequality. Combining these facts, we observe also $||\pi_{t_0}\mathcal{C}_0(\tau)_t||_{2,\beta}\to0$, as $t\to t_0$. Consequently:
$$||\mathcal{C}_0(\tau)_{t_0}-\mathcal{C}_0(\tau)_t||_{2,\beta}\leq||\mathcal{C}_0(\tau)_{t_0}-\mathcal{C}(\tau)_t+\pi_{t_0}\mathcal{C}_0(\tau)_t||_{2,\beta}+||\pi_{t_0}\mathcal{C}_0(\tau)_t||_{2,\beta}\to 0,\,\,\text{as}\,\,\,t\to t_0.$$
\end{proof}
From Lemma \ref{l0}, we conclude that the mapping $\tau_t\to\mathcal{C}(\tau)_t$ is well defined. From our elliptic estimate above, and smoothness of the paths $t\to h_t$, $t\to a_t$, we also conclude with some constant $C_1$ independent of $t$, such that:
$$||\mathcal{C}_0(\tau)_t||_{2,\beta}\leq C_1||\tau_t||_{2,\beta}\leq C_1(\epsilon+||\mu||_{2,\beta}).$$
It follows that $||\mathcal{C}(\tau)_t-\mu||_{2,\beta}\leq \int_0^t||\mathcal{C}_0(\tau)_u||_{2,\beta}du\leq C_1\delta(\epsilon+||\mu||_{2,\beta})$, so that sufficiently small $\delta$ ensures $\mathcal{C}:\mathcal{B}_\epsilon(\mu,\delta)\to\mathcal{B}_\epsilon(\mu,\delta)$. From linearity of the elliptic equation:
$$
||\mathcal{C}_0(\tau_1)_t-\mathcal{C}_0(\tau_2)_t||_{2,\beta}\leq C_1||(\tau_1)_t-(\tau_2)_t||_{2,\beta},$$
giving $||\mathcal{C}(\tau_1)_t-\mathcal{C}(\tau_2)_t||_{2,\beta}\leq \int_0^t||\mathcal{C}_0(\tau_1)_u-\mathcal{C}_0(\tau_2)_u||_{2,\beta}du\leq C_1\delta||\tau_1-\tau_2||_{\infty,\delta}.$ Shrinking $\delta$, if necessary, we therefore conclude that $\mathcal{C}$ is a contraction mapping with a unique fixed point within $\mathcal{B}_\epsilon(\mu,\delta)$. From a standard continuity argument, we conclude that the initial value problem admits a unique solution for all $t\in[0,1]$.\\
\indent We now observe, for any solution $\tau_t$ of our initial value problem:
\begin{align*}
\frac{d}{dt}L_{h_t}(\tau_t) &= L_{h_t}(\eta_t)-\frac{\mathcal{P}(L_{h_t}(\tau_t),a_t)}{\mathcal{P}(h,h)}h_t-2\frac{\mathcal{P}(\text{Sym}(\nabla\tau_t),h_t)}{\mathcal{P}(h,h)}a_t\\
&=b_t-\frac{\mathcal{P}(L_{h_t}(\tau_t),a_t)}{\mathcal{P}(h,h)}h_t
\end{align*}
for some $b_t\in\ubar\Xi_{h_t}$. Consequently,
\begin{align*}
	\frac{d}{dt}\int\mathcal{P}(L_{h_t}(\tau_t),L_{h_t}(\tau_t))dA &= 2\int\mathcal{P}\big(b_t-\frac{\mathcal{P}(L_{h_t}(\tau_t),a_t)}{\mathcal{P}(h,h)}h_t ,L_{h_t}(\tau_t)\big)dA\\
	& = 2\int\mathcal{P}(b_t,L_{h_t}(\tau_t))dA\\
	&=-4\int\langle\nabla\cdot(b_t-(\tr_\gamma b_t)\gamma),\tau_t\rangle dA\\
	&=0.
\end{align*}
We conclude therefore, that $\int\mathcal{P}(L_{h_t}(\tau_t),L_{h_t}(\tau_t))dA = \int\mathcal{P}(L_h(\mu),L_h(\mu))dA$, for all $t\in[0,1]$. So solving with initial data $\mu\in \text{Ker}(\mathcal{L}_h)$, enforces $\int\mathcal{P}(L_{h_t}(\tau_t),L_{h_t}(\tau_t))dA = 0$, for all $t\in[0,1]$, and therefore $\tau_t\in \text{Ker}(\mathcal{L}_{h_t})$ for all $t$. Moreover, from linearity and uniqueness of our initial value problem, this induces a linear injection $\text{Ker}(\mathcal{L}_h)\to\text{Ker}(\mathcal{L}_\gamma)$. It is also easily seen that traversing the path $t\to h_{1-t}$ for $t\in[0,1]$ produces the inverse linear transformation $\text{Ker}(\mathcal{L}_\gamma)\to\text{Ker}(\mathcal{L}_h)$.

\end{proof}
\begin{corollary}\label{c0}
Suppose $h\in\Gamma(\text{Sym}(T^\star\mathbb{S}^n\otimes T^\star\mathbb{S}^n))$ satisfies $\mathcal{P}(h,h)>0$. If $n=2$, or for $n\geq 3$ the metric $\gamma$ is conformal to the standard round metric, then:
$$\dim(\text{Ker}(L_h))= \frac{(n+1)(n+2)}{2}.$$
\end{corollary}
\begin{proof}
We assume $\gamma = e^{2\varphi}\mathring\gamma$ for some smooth function $\varphi\in\mathcal{F}(\Sigma)$, and $\mathring\gamma$ the standard round metric on $\mathbb{S}^n$. The conformal killing operator enjoys conformal invariance of the form:
$$\nabla^i\tau^j+\nabla^j\tau^i-\frac{2}{n}\nabla\cdot\tau\gamma^{ij} = e^{-2\varphi}\big(\mathring\nabla^i\tau^j+\mathring\nabla^j\tau^i - \frac{2}{n}\mathring\nabla\cdot\tau\mathring\gamma^{ij}\big),$$
whereby $\mathring\nabla$ is the induce covariant derivative on $\mathbb{S}^n$ with respect to the standard round metric $\mathring\gamma$. Consequently, via a metric isomorphism, we conclude $\dim(\text{Ker}(L_\gamma)) = \dim(\text{Ker}(L_{\mathring\gamma}))$. It is a well known fact that the group of conformal diffeomorphisms of $(\mathbb{S}^n,\mathring\gamma)$ is isomorphic to $SO(n+1,1)$. Meaning that the dimension of the space of conformal vector fields, that generate the corresponding Lie algebra, is given by:
$$\dim(SO(n+1,1)) = \frac{(n+1)(n+2)}{2}.$$
Regarding the case $n=2$, the famous Uniformization Theorem (see Proposition \ref{p9} below) states, upto a diffeomorphism, any Riemannian metric on $\mathbb{S}^2$ is conformal to the standard round metric. 

\end{proof}

\subsection{The $\Sigma = \mathbb{S}^2$ case}
In the $\mathbb{S}^2$ case, with the help of a local orthonormal frame, we observe $\mathcal{P}(h,h)=2\frac{\det h}{\det\gamma}$. We may decompose any $h$ into trace and traceless parts (or $\mathcal{P}$-linear to $\gamma$ and $\mathcal{P}$-orthogonal to $\gamma$ parts)
$$h = \frac12(\tr_\gamma h)\gamma+\hat{h}.$$
We therefore observe
\begin{align*}
\frac{\det h}{\det \gamma}&=\frac12\mathcal{P}\Big(\frac{1}{\sqrt{2}}(\tr_\gamma h)\frac{\gamma}{\sqrt{2}}+|\hat{h}|\frac{\hat{h}}{|\hat{h}|}, \frac{1}{\sqrt{2}}(\tr_\gamma h)\frac{\gamma}{\sqrt{2}}+|\hat{h}|\frac{\hat{h}}{|\hat{h}|}\Big)\\
&=\frac14(\tr_\gamma h)^2-\frac12|\hat{h}|^2.
\end{align*}
Also, at a given $p\in\mathbb{S}^2$, with the aid of a local coordinate chart, we observe
$$(h^{-1})^{ij} = \frac{\frac12(\tr_\gamma h)\gamma^{ij}-\hat{h}^{ij}}{\frac14(\tr_\gamma h)^2-\frac12|\hat{h}|^2},$$
where indices are raised in $\hat{h}^{ij}$ using the inverse metric $\gamma^{ij}:=(\gamma^{-1})^{ij}$. We conclude therefore, for $h$ such that $\mathcal{P}(h,h)>0$, the operator $L_h(\tau)$ in dimension two is precisely the $h$-traceless part of $2\text{Sym}(\nabla\tau)$:
$$L_h(\tau) = 2\text{Sym}(\tau)-\tr_h(\nabla\tau)h.$$ 
Since $(\text{Sym}(T^\star\mathbb{S}^2\otimes T^\star\mathbb{S}^2),-\mathcal{P})$ is a three dimensional Lorentzian vector bundle we observe that the sub-bundle $\Xi_h$, corresponding with the $h$-traceless 2-tensors, is two dimensional. Therefore, since the principal symbol of $L_h$ is injective, it is an isomorphism. We conclude that $L_h$ as in-fact an elliptic operator. In \cite{li2020} (Theorem 11), Li-Wang show that the equation (\ref{e6}) is solvable for any $q\in\Gamma(\Xi_h)$. Their argument  combines various results from the invariance of index for perturbations of an elliptic operator, the cohomology of $\mathbb{S}^2$, the Maximal principle, and the unique continuation property of elliptic systems. Consequently, they're able to show that the cokernel of $L_h$ is trivial. Invariance of index then also yields $\dim(\text{Ker}(L_h))=6$. We've included an adaptation of their beautiful argument in an Appendix. Using Theorem \ref{t0}, we provide a different proof. 
\begin{proposition}\label{p4}
Given $h\in\Gamma(\text{Sym}(T^\star\mathbb{S}^2\otimes T^\star\mathbb{S}^2))$ such that $\mathcal{P}(h,h)>0$, then the operator $L_h:\Gamma(T^\star\mathbb{S}^2)\to\Gamma(\Xi_h)$ is surjective, equivalently $\ubar\Xi_h=\{0\}$.
\end{proposition}
\begin{proof}
We again connect the positive definite symmetric 2-tensor $h$ to a multiple of the metric, $\gamma$, by a smooth path $t\to h_t$, within $\mathcal{S}$. For each $t\in[0,1]$, we observe that $L_{h_t}:\Gamma(T^\star\Sigma)\to\Gamma(\Xi_{h_t})$ is an elliptic operator. We also have the $h$-orthogonal $\mathcal{P}$ projection $\pi:\text{Sym}(T^\star\mathbb{S}^2\otimes T^\star\mathbb{S}^2)\to\Xi_h$: 
$$\pi(a) = a-\frac12\tr_h(a)h.$$
We now show $\pi|_{\Xi_{h_t}}$ is injective, thus a bundle isomorphism. Assuming $a = \frac12\tr_h(a)h$, for $a\in\Xi_{h_t}$, we observe $0 =\tr_{h_t}(a)= \frac12\tr_h(a)\tr_{h_t}(h)$. Recalling $\frac12\tr_{h_t}(h) = \frac{\mathcal{P}(h_t,h)}{\mathcal{P}(h_t,h_t)}\neq 0$, we therefore conclude $a = \frac12\tr_h(a)h = 0$, and $\pi$ is injective. The operator $\pi\circ L_{h_t}:\Gamma(T^\star\mathbb{S}^2)\to\Gamma(\Xi_h)$ observes the principal symbol $\sigma_\xi(\pi\circ L_{h_t}) = \pi\circ\sigma_\xi(L_{h_t})$. Since this is a composition of injective linear maps, it is an isomorphism. Consequently, the resulting family of operators $\{\pi\circ L_{h_t}\}_{t\in[0,1]}$ are elliptic, and smoothly connects $L_h$ at $t=0$, to $\pi\circ L_\gamma$ at $t=1$. From ellipticity of $\pi\circ L_{h_t}$ for each $t$, and compactness of $\mathbb{S}^2$, these operators are Fredholm. Therefore, the cokernel:
$$ \text{Coker}(\pi\circ L_{h_t}):=((\pi\circ L_{h_t})(\Gamma(T^\star\mathbb{S}^2)))^\perp,$$
where orthogonality is with respect to the $H^1$ inner product induced by $\gamma$ on $\Xi_h$, is a finite dimensional vector space. Standard results apply on the invariance of index for small perturbations of elliptic operators (see, for example Theorem 19.2.2 \cite{hormander2007analysis}). Therefore, we conclude by a simple continuity argument that $\text{index}(L_h) = \text{index}(\pi\circ L_\gamma)$ whereby:
$$\text{index}(L_h):=\dim(\text{Ker}(L_h))-\dim(\text{Coker}(L_h)).$$
Theorem \ref{t0} now implies $\dim(\text{Coker}(\pi\circ L_\gamma)) = \dim(\text{Coker}(L_h))$. However, it's a well known fact that the divergence operator is injective on the space of traceless 2-tensors on a 2-sphere, specifically, $\ubar\Xi_\gamma=\{0\}$. So Proposition \ref{p3} implies $L_\gamma$ is surjective, and since $\pi$ is a bundle isomorphism, $\pi\circ L_\gamma$ is also surjective. Consequently, $\text{Coker}(L_h) = \{0\}$, and therefore $\ubar\Xi_h = \{0\}$. 

\end{proof}
\begin{corollary}
Given $h\in\Gamma(\text{Sym}(T^\star\mathbb{S}^2\otimes T^\star\mathbb{S}^2))$ such that $\mathcal{P}(h,h)>0$, then,
$$\{a\in\Gamma(\Xi_h)|\nabla\cdot a = 0\}=\{0\}.$$
\end{corollary}
\begin{proof}
For a given metric $\gamma$, in a local coordinate neighborhood $(x^i,\mathcal{U})$, we recall the volume 2-form on the 2-sphere is given by $\varepsilon_\gamma:=\sqrt{\det\gamma}dx^1\wedge dx^2$. We also recall the famous identity:
 \begin{equation}\label{i0}
 \varepsilon^{ij}\varepsilon_{kl} =\delta^i_k\delta^j_l-\delta^j_k\delta^i_l.
 \end{equation}
Given $a\in\text{Sym}(T^\star\mathbb{S}^2\otimes T^\star\mathbb{S}^2)$, we define the map $\varepsilon: \text{Sym}(T^\star\mathbb{S}^2\otimes T^\star\mathbb{S}^2) \to \text{Sym}(T^\star\mathbb{S}^2\otimes T^\star\mathbb{S}^2)$, by $\varepsilon(a)_{ij} = a_{kl}\varepsilon_i\,^k\varepsilon_j\,^l$. This induces a bundle isomorphism since $\varepsilon\circ\varepsilon = Id$. We also observe
$$\varepsilon(a) = (\tr_\gamma a)\gamma-a.$$
We recognize therefore, with the help of a coordinate neighborhood, that $\frac12\mathcal{P}(h,h)(h^{-1})_{ij} = \varepsilon(h)_{ij}$. Since $\mathcal{P}(h,h)>0$ identifies $h\in \Gamma(\text{Sym}(T^\star\mathbb{S}^2\otimes T^\star\mathbb{S}^2)) $ as either globally positive, or negative definite, we also have $\mathcal{P}(h^{-1},h^{-1})>0$ (in other words, matrix inversion leaves $\mathcal{S}_-,\mathcal{S}_+$ invariant). Now for $a\in\Xi_h$, it therefore follows that $0=\langle\varepsilon(h),a\rangle =\langle h,\varepsilon(a)\rangle = \langle (h^{-1})^{-1},\varepsilon(a)\rangle=\tr_{h^{-1}}(\varepsilon(a))$, namely $\varepsilon(\Xi_h) = \Xi_{h^{-1}}$. So for $b\in\Gamma(\Xi_h)$ satisfying $\nabla\cdot b = 0$, we observe $0=\nabla\cdot(\varepsilon(\varepsilon(b))) = \nabla\cdot((\tr_\gamma (\varepsilon(b))\gamma-\varepsilon(b))=d\tr_\gamma(\varepsilon(b))-\nabla\cdot\varepsilon(b)$. By Proposition \ref{p4}, $\varepsilon(b)\in\ubar\Xi_{h^{-1}}=\{0\}$.

\end{proof}
 \subsection{Openness I}
This section is largely an adaptation of the openness argument due to Lin-Wang in \cite{linwang1} to our more general setting. We present it here not only to incorporate our earlier results, but for the reader to gain an intimate understanding of the procedure that will be involved in later sections.\\\\ 
\indent We assume the existence of an embedding, $B:I_0\times \mathbb{S}^n\to\mathcal{M}$, for an interval $I_0 = (s_0-\epsilon,s_0+\epsilon)$, such that $0<\epsilon<s_0$. We wish to identify a diffeomorphism of $I_0\times\mathbb{S}^n$ with an annulus in $\mathbb{R}^{n+1}$. It will help to introduce a global coordinate system $(z^\mu)$ for $\mathbb{R}^{n+1}$, along with the standard Euclidean norm $|\vec{z}|:=\sqrt{\sum_\mu(z^\mu)^2}$. We take the standard embedding to the $s_0$-radius sphere $\mathbb{S}^n\hookrightarrow_{\iota}\mathbb{R}^{n+1}$, whereby $|\vec{z}|\circ\iota(p) = s_0$ for $p\in\mathbb{S}^n$. This then identifies a diffeomorphism onto the annulus, $\mathcal{A}:=\{\vec{z}\in\mathbb{R}^{n+1}|\sqrt{\sum_\mu(z^\mu)^2}\in I_0\}$, through the mapping $\Phi:I_0\times\mathbb{S}^n\to \mathcal{A}$, whereby $\Phi(s,p) = s\frac{\vec{z}}{|\vec{z}|}\circ\iota(p)\in\mathcal{A}$. The resulting map $\tilde B = B\circ \Phi^{-1}:\mathcal{A}\to\mathcal{M}$ allows us to pull-back the metric $g$ to $\mathcal{A}$, we will denote this pull-back $\sigma = \tilde B^\star(g)$.\\ 
\indent As in the previous section, we again assume the existence of a mapping $F:I\times\mathbb{S}^n\to\mathcal{M}$, such that $F_\lambda:\mathbb{S}^n\hookrightarrow\mathcal{M}$ for each $\lambda\in I$, moreover, we assume that $F$ factors through $\mathcal{A}$. Specifically, we have a mapping $R:I\times\mathbb{S}^n\to \mathcal{A}$ such that $F = \tilde B\circ R$, inducing a family of embeddings $\{R_\lambda:\mathbb{S}^n\hookrightarrow\mathcal{A}\}$. It will be helpful to distinguish $R_0$ so we denote instead $r:=R_0$. We recall that $dF(\partial_\lambda)|_{F_0(\mathbb{S}^n)}\in\Gamma(T^\perp F_0(\mathbb{S}^n))$, inducing also that $\upsilon:=dR(\partial_\lambda)|_{r(\mathbb{S}^n)}\in\Gamma(T^\perp r(\mathbb{S}^n))$. We also recall the `second fundamental form', $h: = \pounds_{\partial_\lambda}(F^\star(g))|_{\{\lambda = 0\}}=\pounds_{\partial_\lambda}(R^\star(\sigma))|_{\{\lambda = 0\}}$.\\
\indent Both embeddings $\mathbb{S}^n\hookrightarrow_{F_0}(\mathcal{M},g)$ and $\mathbb{S}^n\hookrightarrow_{r}(\mathcal{A},\sigma)$ induce the same metric $\gamma$ on $\mathbb{S}^n$. In this section, our strategy will be to deal directly with $(\mathcal{A},\sigma)$ to investigate when sufficiently small perturbations of the metric $\gamma$, denoted $\tilde\gamma$, can be isometrically embedded into $(\mathcal{A},\sigma)$, thus $(\mathcal{M},g)$. At every point of $q\in\mathcal{A}\subset\mathbb{R}^{n+1}$, we have a natural isomorphism between $\mathbb{R}^{n+1}$ and $T_q\mathbb{R}^{n+1}$. 
Therefore, it will be convenient when the context is clear, albeit an abuse of notation, to not distinguish between points, vectors, or coordinate representations thereof. We will signify this abuse with the use of an $\vec{a}$\textit{rrow}.\\
\indent We observe in a local coordinate neighborhood $(x^i,\mathcal{U})$ of $p\in\mathbb{S}^n$:
 $$r^\star(\sigma) = \gamma\iff\sigma_{\mu\nu}(\vec{r})\frac{\partial r^\mu}{\partial x^i}\frac{\partial r^\nu}{\partial x^j}=\gamma_{ij}.$$  
Following the Nirenberg approach \cite{nirenberg1}, if $\tilde\gamma$ represents a perturbation of $\gamma$, we wish to investigate the existence of some vector field $\vec{y}:\mathbb{S}^n\to\mathbb{R}^{n+1}$, such that
 $$\sigma_{\mu\nu}(\vec{r}+\vec{y})\frac{\partial(r^\mu+y^\mu)}{\partial x^i}\frac{\partial (r^\nu+y^\nu)}{\partial x^j}=\tilde{\gamma}_{ij}.$$
Writing, $\vec{y} = \tau^i\vec{r}_i+\phi\vec{\upsilon}$, whereby $\vec{r}_i:=\frac{\partial r^\mu}{\partial x^i}\partial_\mu$, we decompose this expression:
\begin{align*}
\tilde{\gamma}_{ij}-\gamma_{ij} &= \sigma_{\mu\nu}(\vec{r})(r^\mu_i y^\nu_j+r^\nu_j y^\mu_i)+\sigma_{\mu\nu}(\vec{r})y^\mu_i y^\nu_j+\Big(\sigma_{\mu\nu}(\vec{r}+\vec{y})-\sigma_{\mu\nu}(\vec{r})\Big)(r^\mu_i r^\nu_j+r^\mu_i y^\nu_j+r^\mu_j y^\nu_i+y^\mu_i y^\nu_j)\\
&=(\tau^m,_{j}\gamma_{mi}+\tau^m,_{i}\gamma_{mj})+\tau^m\sigma_{\mu\nu}(\vec{r})(r^\mu_i r^\nu_{mj}+r^\nu_j r^\mu_{mi})+\phi\sigma_{\mu\nu}(\vec{r})(r^\mu_i \upsilon^\nu_j+r^\nu_j \upsilon^\mu_i)\\
&\qquad+ \sigma_{\mu\nu}(\vec{r})y^\mu_i y^\nu_j+\Big(\sigma_{\mu\nu}(\vec{r}+\vec{y})-\sigma_{\mu\nu}(\vec{r})\Big)(r^\mu_i r^\nu_j+r^\mu_i y^\nu_j+r^\nu_j y^\mu_i+y^\mu_i y^\nu_j),
\end{align*}
having used $\vec{\upsilon}\in\Gamma(T^\perp r(\mathbb{S}^n))$ in the second equality. Concentrating on the second term of the second equality:
\begin{align*}
\tau^m\sigma_{\mu\nu}(\vec{r})(r^\mu_i r^\nu_{mj}+r^\nu_j r^\mu_{mi})&=\tau^m\sigma_{\mu\nu}(\vec{r})\partial_m(r^\mu_i r^\nu_j)=\tau^m\gamma_{ij,m}-\tau^m\sigma_{\mu\nu,\eta}(\vec{r})r^\mu_i r^\nu_j r^\eta_m\\
&=\tau^m\gamma_{ij,m}-\sigma_{\mu\nu,\eta}(\vec{r})r^\mu_i r^\nu_j y^\eta+\phi\sigma_{\mu\nu,\eta}(\vec{r})r^\mu_i r^\nu_j \upsilon^\eta.
\end{align*}
We now observe,
\begin{align*}
h_{ij}&=\pounds_{\partial_\lambda}(R^\star(\sigma))(\partial_i,\partial_j)\circ r\\
&=\partial_\lambda((\sigma_{\mu\nu}\circ R)R^\mu_i R^\nu_j)\circ r\\
&=\sigma_{\mu\nu,\eta}(\vec{r})r^\mu_i r^\nu_j R^\eta_\lambda\circ r+\sigma_{\mu\nu}(\vec{r})r^\nu_j R^\mu_{\lambda i} \circ r+\sigma_{\mu\nu}(\vec{r})r^\mu_i R^\nu_{\lambda j}\circ r\\
&=\sigma_{\mu\nu,\eta}(\vec{r})r^\mu_i r^\nu_j\upsilon^\eta+\sigma_{\mu\nu}(\vec{r})(\upsilon^\mu_i r^\nu_j+r^\mu_i\upsilon^\nu_j),
\end{align*}
putting everything together, we conclude:
\begin{align*}
\tilde{\gamma}_{ij}-\gamma_{ij} &= \Big(\nabla_i\tau_j+\nabla_j\tau_i+\phi h_{ij}\Big)+ \sigma_{\mu\nu}(\vec{r})y^\mu_i y^\nu_j-\partial_\eta(\sigma_{\mu\nu})r^\mu_i r^\nu_j y^\eta\\
&\qquad+\Big(\sigma_{\mu\nu}(\vec{r}+\vec{y})-\sigma_{\mu\nu}(\vec{r})\Big)(r^\mu_i r^\nu_j+r^\mu_i y^\nu_j+r^\nu_j y^\mu_i+y^\mu_i y^\nu_j).
\end{align*}
From the Taylor approximation theorem:
\begin{align*}	
\sigma_{\mu\nu}(\vec{r}+\vec{y}) &= \sigma_{\mu\nu}(\vec{r})+y^\eta\int_0^1\partial_\eta\sigma_{\mu\nu}(\vec{r}+t\vec{y})dt\\
&= \sigma_{\mu\nu}(\vec{r})+\partial_\eta\sigma_{\mu\nu}(\vec{r})y^\eta+ y^\eta y^\xi\int_0^1(1-t)\partial^2_{\mu\nu}\sigma_{\eta\xi}(\vec{r}+t\vec{y})dt
\end{align*}
from which we conclude, in the notation of Li-Wang,
\begin{align}\label{e10}
\nabla_i\tau_j+\nabla_j\tau_i +\phi h_{ij} &= \tilde{\gamma}_{ij}-\gamma _{ij}-\sigma_{\mu\nu}(\vec{r})y^\mu_i y^\nu_j-y^\eta y^\xi r^\mu_i r^\nu_j F_{\mu\nu\eta\xi}(\vec{r},\vec{y})\\
&\qquad-G_{\mu\nu\eta}(\vec{r},\vec{y})y^\eta(r^\mu_i y^\nu_j+r^\nu_j y^\mu_i+y^\mu_i y^\nu_j)\notag,
\end{align}
whereby,
$$F_{\mu\nu\eta\xi}(\vec{r},\vec{y}):=\int_0^1(1-t)\partial^2_{\eta\xi}\sigma_{\mu\nu}(\vec{r}+t\vec{y})dt,\,\,\,G_{\mu\nu\eta}(\vec{r},\vec{y}):=\int_0^1\partial_\eta\sigma_{\mu\nu}(\vec{r}+t\vec{y})dt.$$
We may alternatively denote equation (\ref{e10}): $\nabla_i\tau_j+\nabla_j\tau_i +\phi h_{ij}=\tilde q(\vec{y},\nabla\vec{y})_{ij}$. Regarding the existence of $\vec{y} = \tau^i\vec{r}_i+\phi\vec{\upsilon}$, we have:
\begin{theorem}\label{t1}
Consider,
\begin{align*}
h\in C^{2,\beta}(\text{Sym}(T^\star\mathbb{S}^n\otimes T^\star\mathbb{S}^n)),\,\,\, \gamma &\in C^{3,\beta}(\text{Sym}(T^\star\mathbb{S}^n\otimes T^\star\mathbb{S}^n)),\,\,\,\tilde\gamma \in C^{2,\beta}(\text{Sym}(T^\star\mathbb{S}^n\otimes T^\star\mathbb{S}^n)),\\
\vec{\upsilon}&\in C^{2,\beta}(\mathbb{S}^n,\mathbb{R}^{n+1}),\,\,\,\vec{r}\in C^{3,\beta}(\mathbb{S}^n,\mathbb{R}^{n+1}).
\end{align*}
Then, provided $\mathcal{P}(h,h)>0$, there exists $\epsilon,\delta>0$, such that any $||\tilde\gamma-\gamma||_{2,\beta}\leq \epsilon$  gives rise to unique $\vec{y}\in C^{2,\beta}(\mathbb{S}^n,\mathbb{R}^{n+1})$, $||\vec{y}||_{2,\beta}\leq \delta$, $\tau(\vec{y})\in C^{2,\beta}(T^\star\mathbb{S}^n)$, and $a(\vec{y})\in C^{1,\beta}(\text{Sym}(T^\star\mathbb{S}^n\otimes T^\star\mathbb{S}^n))$ such that:
\begin{align*}
\tilde\gamma&+a(\vec{y})= (r+y)^\star(\sigma),\\
\vec{y} &= d\vec{r}(\tau^\# (\vec{y}))+\phi(\vec{y})\vec{\upsilon},
\end{align*}
whereby:
\begin{itemize}
\item $\mathcal{P}(a,h)$ = 0, $\nabla\cdot a = d\tr_\gamma a.$
\item $\phi(\vec{y}) = \frac{\mathcal{P}(\tilde q(\vec{y},\nabla\vec{y})-2\text{Sym}(\nabla\tau(\vec{y})),h)}{\mathcal{P}(h,h)}$, for $q(\vec{y},\nabla\vec{y}) = \tilde q(\vec{y},\nabla\vec{y}) -\frac{\mathcal{P}(\tilde q(\vec{y},\nabla\vec{y}),h)}{\mathcal{P}(h,h)}h$, and $\tilde q(\vec{y},\nabla\vec{y})$ in (\ref{e10}).
\end{itemize}

\end{theorem}
\begin{proof}
As in \cite{li2020}, we construct a contraction mapping. We do so by taking any $\vec{z}\in C^{2,\beta}(\mathbb{S}^n,\mathbb{R}^{n+1})$, then for $\tilde q(\vec{z},\nabla\vec{z})\in C^{1,\beta}(\text{Sym}(T^\star\mathbb{S}^n\otimes T^\star\mathbb{S}^n))$, we solve the equation:
$$\mathcal{L}_h(\tau) = \nabla\cdot q(\vec{z},\nabla\vec{z}) - d\tr_\gamma q(\vec{z},\nabla\vec{z})\in C^{0,\beta}(T^\star\mathbb{S}^n),$$
yielding $\tau(\vec{z})\in C^{2,\beta}(T^\star\mathbb{S}^n)$ by elliptic regularity. As in Theorem \ref{t0}, for any $\tau\in \text{Ker}(\mathcal{L}_h)^\perp$:
$$||\tau||_{2,\beta}\leq C(||h||_{1,\beta},||\gamma||_{2,\beta})||\mathcal{L}_h(\tau)||_{0,\beta}.$$
\indent So, for a solution $\tau(\vec{z})\in\text{Ker}(\mathcal{L}_h)^\perp\cap C^{2,\beta}(T^\star\mathbb{S}^n)$ we observe the estimate
\begin{equation}\label{e11}
||\tau(\vec{z})||_{2,\beta}\leq C(||h||_{1,\beta},||\gamma||_{2,\beta})||\tilde q(\vec{z},\nabla\vec{z})||_{1,\beta}.
\end{equation}
We recall, along with $\phi(\vec{z}) = \frac{\mathcal{P}(\tilde q(\vec{z},\nabla\vec{z})-2\text{Sym}(\nabla\tau(\vec{z})),h)}{\mathcal{P}(h,h)}\in C^{1,\beta}(\mathbb{S}^n)$, the  pair $(\tau(\vec{z}),\phi(\vec{z}))$ solves the linearized embedding equation: 
$$2\text{Sym}(\nabla\tau(\vec{z}))+\phi(\vec{z})h = \tilde q(\vec{z},\nabla\vec{z})+a(\vec{z}),$$
upto some unique $a(\vec{z})\in C^{1,\beta}(\text{Sym}(T^\star\mathbb{S}^n\otimes T^\star\mathbb{S}^n))$ satisfying $\mathcal{P}(a,h)=0$, $\nabla\cdot a = d\tr_\gamma a$. Currently, we have a mapping $\mathcal{C}:C^{2,\beta}(\mathbb{S}^n,\mathbb{R}^{n+1})\to C^{1,\beta}(\mathbb{S}^n,\mathbb{R}^{n+1})$, given by
\begin{align*}
\mathcal{C}(\vec{z})&= d\vec{r}(\tau^\#(\vec{y}))+\phi(\vec{z})\vec{\upsilon}\\
&=\tau(\vec{z})^i\vec{r}_i+\phi(\vec{z})\vec{\upsilon},
\end{align*}
in local coordinates. Our result will follow as soon as we can show $\mathcal{C}$ is in-fact a contraction mapping on the appropriate domain. Therefore, our first goal is to show $\phi(\vec{z})\in C^{2,\beta}(\mathbb{S}^n)$ so that $\mathcal{C}: C^{2,\beta}(\mathbb{S}^n,\mathbb{R}^{n+1})\to C^{2,\beta}(\mathbb{S}^n,\mathbb{R}^{n+1})$. In order to do so, we use a variant of Nirenberg's trick that generalizes the procedure noticed by Li-Wang in \cite{li2020}. \\\\
\indent We recall for a variation of the metric $\gamma$, denoted $\delta\gamma$, the Scalar Curvature varies according to
$$
\delta R =\langle \text{Ric},\delta\gamma\rangle+\nabla\cdot\nabla\cdot(\delta\gamma - (\tr_\gamma\delta\gamma)\gamma).$$
We temporarily assume smooth data $\{\tau,\phi,\tilde q,a\}$, and take $2\text{Sym}(\nabla\tau)+\phi h = \delta\gamma = \tilde q+a$, for some $a\in\ubar\Xi_h$. The following equation results:
$$\nabla\cdot\nabla\cdot(\phi((\tr_\gamma h)\gamma-h)) = 2\nabla\cdot\nabla\cdot(\text{Sym}(\nabla\tau)-(\nabla\cdot\tau)\gamma)+\nabla\cdot\nabla\cdot((\tr_\gamma\tilde q)\gamma-\tilde q).$$
We calculate:
\begin{align*}
2\nabla\cdot\nabla\cdot\text{Sym}(\nabla\tau)&=\gamma^{ij}\gamma^{mn}\nabla_i\nabla_m(\nabla_j\tau_n+\nabla_n\tau_j)\\
&=\gamma^{ij}\gamma^{mn}\nabla_i\nabla_m\nabla_j\tau_n+\gamma^{ij}\gamma^{mn}\nabla_m\nabla_i\nabla_n\tau_j+\gamma^{ij}\gamma^{mn}( R_n\,^r\,_{mi}\nabla_r\tau_j+R_j\,^r\,_{mi}\nabla_n\tau_r)\\
&=2\gamma^{ij}\gamma^{mn}\nabla_i(\nabla_j\nabla_m\tau_n+R_{nrjm}\tau^r)\\
&=2\Delta(\nabla\cdot\tau)-2\nabla\cdot(\text{Ric}(\tau)).
\end{align*}
From this we conclude that $\phi$ satisfies the equation:
\begin{equation}\label{e12}
\nabla\cdot\nabla\cdot(\phi\Big((\tr_\gamma h)\gamma-h\Big)) = 2\nabla\cdot(\text{Ric}(\tau))+\nabla\cdot\nabla\cdot\Big((\tr_\gamma \tilde q)\gamma-\tilde q\Big).
\end{equation}
If we regard (\ref{e12}) as an equation on $\phi$, multiplying with the appropriate sign to induce $\tr_\gamma h>0$, we see this equation observes the principal symbol:
$$\Big((\tr_\gamma h)|\xi|^2-h(\xi,\xi)\Big)\phi\geq\Big((\tr_\gamma h)-|h|\Big)|\xi|^2\phi.$$
Since $\mathcal{P}(h,h)>0$ gives $(\tr_\gamma h)>|h|$, we conclude that (\ref{e12}) is a uniformly elliptic equation acting on $\phi$. We wish to use (\ref{e12}) on our data, $\{\tau(\vec{z}),\phi(\vec{z}),\tilde q(\vec{z},\nabla\vec{z})\}$, to improve the regularity of $\phi(\vec{z})$. Before we do so, we can already see that the operation induced upon $\tau(\vec{z})$ will yield a term $\nabla\cdot(\text{Ric}(\tau(\vec{z})))\in C^{0,\beta}(\mathbb{S}^n)$ (since $\gamma\in C^{3,\beta}(\text{Sym}(T^\star\mathbb{S}^n\otimes T^\star\mathbb{S}^n))$. However, since $\vec{z}\in C^{2,\beta}(\mathbb{S}^n,\mathbb{R}^{n+1})$, it appears that the operation induced upon $\tilde q(\vec{z},\nabla\vec{z})\in C^{1,\beta}(\text{Sym}(T^\star\mathbb{S}^n\otimes T^\star\mathbb{S}^n))$ is ill-defined. What saves us here can be observed with the help of some $\vec{y}\in C^3(\mathbb{S}^n,\mathbb{R}^{n+1})$. Specifically, the operation acting on $\tilde q(\vec{y},\nabla\vec{y})$ does not result in any derivatives of $\vec{y}$ larger than two. To see this we take a local coordinate neighborhood $(x^i,\mathcal{U})$, and observe from (\ref{e10}), $\tilde q(\vec{y},\nabla\vec{y})_{ij} = S_{\mu\nu}^1(\vec{y})\big(y^\mu_i R^\nu_j(\vec{y})+y^\mu_j R^\nu_i(\vec{y})\big)+S_{\mu\nu}^2(\vec{y})y^\mu_i y^\nu_j+S^3_{ij}(\vec{y})$ for some symmetric 2-tensors $S^1,S^2$ on $\mathbb{R}^{n+1}$, and $S^3$ on $\mathbb{S}^n$. The highest order derivatives on $\vec{y}$ in the quantity $\nabla\cdot\nabla\cdot((\tr_\gamma\tilde q(\vec{y},\nabla\vec{y}))\gamma - \tilde q(\vec{y},\nabla\vec{y}))$ take the form:
\begin{enumerate}
\item Terms contracted with $S^1$:
\begin{align*}
&\gamma^{mn}\gamma^{kl}\Big(\gamma^{ij}\gamma_{nl}\partial_m\partial_k\big(S^1_{\mu\nu}(\vec{y})(y^\mu _iR^\nu _j(\vec{y})+y_j^\mu R^\nu _i(\vec{y}))\big)-\partial_m\partial_k\big(S^1_{\mu\nu}(\vec{y})(y^\mu _nR^\nu _l(\vec{y})+y_l^\mu R^\nu _n(\vec{y}))\big)\Big)\\
&=2\gamma^{mk}\gamma^{ij} S^1_{\mu\nu}(\vec{y})(\partial_m\partial_k y^\mu_i)R^\nu_j(\vec{y})-\gamma^{mn}\gamma^{kl}S^1_{\mu\nu}(\vec{y})\big((\partial_m\partial_k y^\mu_n)R^\nu_l(\vec{y})+(\partial_m\partial_k y^\mu_l) R^\nu_n(\vec{y})\big)\\
&\qquad\qquad+R_1(\nabla^2\vec{y},\nabla\vec{y},\vec{y})\\
&=2\gamma^{mk}\gamma^{ij} S^1_{\mu\nu}(\vec{y})(\partial_m\partial_k y^\mu_i)R^\nu_j(\vec{y})-\gamma^{mn}\gamma^{kl}S^1_{\mu\nu}(\vec{y})\big((\partial_m\partial_ny^\mu_k)R^\nu_l(\vec{y})+(\partial_k\partial_ly^\mu_m)R^\nu_n(\vec{y})\big)\\
&\qquad\qquad +R_1(\nabla^2\vec{y},\nabla\vec{y},\vec{y})\\
&=R_1(\nabla^2\vec{y},\nabla\vec{y},\vec{y})
\end{align*}
\item Terms contracted with $S^2$:
\begin{align*}
&\gamma^{mn}\gamma^{kl}\Big(\gamma^{ij}\gamma_{nl}\partial_m\partial_k\big(S^2_{\mu\nu}(\vec{y})y^\mu _iy^\nu_j\big)-\partial_m\partial_k\big(S^2_{\mu\nu}(\vec{y})y^\mu _ny^\nu _l\big)\Big)\\
&=2\gamma^{mk}\gamma^{ij}S^2_{\mu\nu}(\vec{y})(\partial_m\partial_k y^\mu_i) y^\nu_j-\gamma^{mn}\gamma^{kl}S^2_{\mu\nu}(\vec{y})\big((\partial_m\partial_k y^\mu_n)y^\nu_l+y^\nu_n(\partial_m\partial_k y^\nu_l)\big)\\
&\qquad\qquad+R_2(\nabla^2\vec{y},\nabla\vec{y},\vec{y})\\
&= 2\gamma^{mk}\gamma^{ij}S^2_{\mu\nu}(\vec{y})(\partial_m\partial_k y^\mu_i) y^\nu_j-\gamma^{mn}\gamma^{kl}S^2_{\mu\nu}(\vec{y})\big((\partial_m\partial_n y^\mu_k)y^\nu_l+y^\nu_n(\partial_k\partial_l y^\nu_m)\big)\\
&\qquad\qquad+R_2(\nabla^2\vec{y},\nabla\vec{y},\vec{y})\\
&= R_2(\nabla^2\vec{y},\nabla\vec{y},\vec{y}).
\end{align*}
\end{enumerate}
It follows for $\tilde q = \tilde q(\vec{y},\nabla\vec{y})$, that the right hand of the equality in equation (\ref{e12}) is an expression in terms of $\{\tau,\vec{y}\}$ with no higher than second derivatives.\\
\indent In order to translate the above analysis onto our data $\{\tau(\vec{z}),\phi(\vec{z}), \tilde q(\vec{z},\nabla\vec{z})\},$ we will use the well known fact than any element of our index $(\mathfrak{n},\beta)$-H\"{o}lder spaces can be approximated by a smooth counterpart				 within a less regular, say $||\cdot||_{\mathfrak{n},\beta'}$, norm for any $0<\beta'<\beta$. Consequently, we take a smooth sequence $\{\vec{z}_n\}\subset C^\infty(\mathbb{S}^n,\mathbb{R}^{n+1})$ converging in $C^{2,\beta'}(\mathbb{S}^n,\mathbb{R}^{n+1})$ to $\vec{z}\in C^{2,\beta}(\mathbb{S}^2,\mathbb{R}^3)$. With the improved regularity on terms of the sequence $\{\vec{z}_n\}$, we obtain a sequence $\{\tau'_n\}\subset C^{3,\beta'}(T^\star\mathbb{S}^n)\cap\text{Ker}(\mathcal{L}_h)^\perp$ with interior Schauder estimates giving $||\tau'_n||_{3,\beta'}\leq C(||h||_{2,\beta'},||\gamma||_{3,\beta'})||\tilde q(\vec{z}_n,\nabla\vec{z}_n)||_{2,\beta'}$. From the estimate (\ref{e11}): 
$$||\tau'_n-\tau'_m||_{2,\beta'}\leq C(||h||_{1,\beta'},||\gamma||_{2,\beta'
})||\tilde q(\vec{z}_n,\nabla\vec{z}_n)-\tilde q(\vec{z}_m,\nabla\vec{z}_m)||_{1,\beta'},$$
and we obtain a Cauchy sequence. It follows that $\tau'_n\to\tau(\vec{z})$ as $n\to\infty$ in $C^{2,\beta'}(T^\star\mathbb{S}^n)$. Each $\tau'_n$ yields an associated $\phi'_n\in C^{2,\beta'}(\mathbb{S}^n)$, and by the convergence of $\tau'_n$ we have $\phi'_n\to\phi(\vec{z})$ in $C^{1,\beta'}(\mathbb{S}^n)$. For each $n\in\mathbb{N}$, the data $\{\tau'_n,\phi'_n,\tilde q(\vec{z}_n,\nabla\vec{z}_n)\}$ satisfies equation (\ref{e12}) and we obtain the interior Schauder estimates:
\begin{align*}
||\phi'_n||_{2,\beta'}&\leq C(||h||_{2,\beta'},||\gamma||_{3,\beta'})\Big(||\tau'_n||_{1,\beta'}+|| \nabla\cdot\nabla\cdot\Big((\tr_\gamma \tilde q(\vec{z}_n,\nabla\vec{z}_n))\gamma-\tilde q(\vec{z}_n,\nabla\vec{z}_n)\Big)||_{0,\beta'}+||\phi'_n||_{0,\beta'}\Big)\notag\\
&\leq C(||h||_{2,\beta'},||\gamma||_{3,\beta'})\Big(||\tilde q(\vec{z}_n,\nabla\vec{z}_n)||_{1,\beta'}+ || \nabla\cdot\nabla\cdot\Big((\tr_\gamma \tilde q(\vec{z}_n,\nabla\vec{z}_n))\gamma-\tilde q(\vec{z}_n,\nabla\vec{z}_n)\Big)||_{0,\beta'}\Big).
\end{align*}
By linearity of (\ref{e12}) and the analysis on derivatives above, we also obtain from this estimate (similarly as for $\{\tau'_n\}$) that $\{\phi'_n\}$ is a Cauchy sequence in $C^{2,\beta'}(\mathbb{S}^n)$. It follows that the limit $\phi(\vec{z})\in C^{2,\beta'}(\mathbb{S}^n)$ satisfies (\ref{e12}). Since all coefficients are of regularity $C^{0,\beta}(\mathbb{S}^n)$ we can improve the regularity to $\phi(\vec{z})\in C^{2,\beta}(\mathbb{S}^n)$ along with the estimate:
\begin{equation}\label{e13}
||\phi(\vec{z})||_{2,\beta} \leq C(||h||_{2,\beta},||\gamma||_{3,\beta})\Big(||\tilde q(\vec{z},\nabla\vec{z})||_{1,\beta}+ || \nabla\cdot\nabla\cdot\Big((\tr_\gamma \tilde q(\vec{z},\nabla\vec{z}))\gamma-\tilde q(\vec{z},\nabla\vec{z})\Big)||_{0,\beta}\Big),
\end{equation}
where we understand the final term above as an element of $C^{0,\beta}(\mathbb{S}^n)$ since it holds no derivatives of $\vec{z}$ larger than two. We conclude $\mathcal{C}:C^{2,\beta}(\mathbb{S}^{n},\mathbb{R}^{n+1})\to C^{2,\beta}(\mathbb{S}^{n},\mathbb{R}^{n+1}) $ given by: 
$$\mathcal{C}(\vec{z}) = \tau(\vec{z})^i\vec{r}_i+\phi(\vec{z})\vec{\upsilon}$$
is well defined. Moreover, from (\ref{e10},\ref{e11},\ref{e13}):
\begin{align}\label{e14}
||\mathcal{C}(\vec{z})||_{2,\beta}&=||\tau(\vec{z})^i\vec{r}_i+\phi(\vec{z})\vec{\upsilon}||_{2,\beta}\notag\\
&\leq C(||h||_{2,\beta},||\vec\upsilon||_{2,\beta},||\gamma||_{3,\beta},||\vec{r}||_{3,\beta})\Big(||\gamma-\tilde\gamma||_{2,\beta}+||\vec{z}||^2_{2,\beta}+ ||\vec{z}||^3_{2,\beta}+||\vec{z}||^4_{2,\beta}\Big).
\end{align} 
Since all exponents of $||\vec{z}||_{2,\beta}$ are at least two, we observe, provided $||\gamma-\tilde\gamma||_{2,\beta}\leq\epsilon,||\vec{z}||_{2,\beta}\leq\delta$ for sufficiently small $\epsilon,\delta$, that $\mathcal{C}(B_\delta)\subseteq B_\delta$ (for $B_\delta \subset C^{2,\beta}(\mathbb{S}^n,\mathbb{R}^{n+1})$, the closed ball of radius $\delta$). We also have (see \cite{li2020}, Lemma 14):
\begin{align}\label{e15}
\Big(q(\vec{z}_1,\nabla\vec{z}_1)&-q(\vec{z}_2,\nabla\vec{z}_2)\Big)_{ij}=\notag\\
&-\sigma_{\mu\nu}(\vec{r})\Big[((z_1^\mu)_i-(z_2^\mu)_i)(z_1^\nu)_j+(z_2^\mu)_i\Big((z_1^\nu)_j-(z^\nu_2)_j\Big)\Big]\notag\\
&-(z_1^\rho-z_2^\rho)\Big(\int_0^1\frac{\partial F_{\mu\nu\eta\xi}}{\partial y^\rho}(t\vec{z}_1+(1-t)\vec{z}_2)dt\Big)z_1^\eta z_1^\xi r^\mu_i r^\nu_j\notag\\
&+F_{\mu\nu\eta\xi}(\vec{z}_2)(z_1^\eta+z_2^\eta)(z_1^\xi-z_2^\xi)r_i^\mu r_j^\nu\notag\\
&-(z_1^\rho-z_2^\rho)\Big(\int_0^1\frac{\partial G_{\mu\nu\eta}}{\partial y^\rho}(t\vec{z}_1+(1-t)\vec{z}_2)dt\Big)z_1^\eta\Big(r^\mu_i(z_1^\nu)_j+r^\nu_j(z_1^\mu)_i+(z_1^\mu)_i(z_1^\nu)_j\Big)\notag\\
&+G_{\mu\nu\eta}(\vec{z}_2)(z_1^\eta-z_2^\eta
)\Big(r^\mu_i(z_1^\nu)_j+r^\nu_j(z_1^\mu)_i+(z_1^\mu)_i(z_1^\nu)_j\Big)\notag\\
&+G_{\mu\nu\eta}(\vec{z}_2)z_2^\eta\Big(r_i^\mu(z_1^\nu-z_2^\nu)_j+r^\nu_j(z_1^\mu-z_2^\mu)_i\Big)\notag\\
&+G_{\mu\nu\eta}(\vec{z}_2)z_2^\eta\Big((z_1^\mu-z_2^\mu)_i(z_1^\nu)_j+(z_2^\mu)_i(z_1^\nu-z_2^\nu)_j\Big).
\end{align}
From this and (\ref{e13}), we observe:
\begin{equation}\label{e16}
||\mathcal{C}(\vec{z}_1)-\mathcal{C}(\vec{z}_2)||_{2,\beta}\leq C(||h||_{2,\beta},||\gamma||_{3,\beta})\text{max}(||\vec{z}_1||_{2,\beta},||\vec{z}_2||_{2,\beta})||\vec{z}_1-\vec{z}_2||_{2,\beta}\leq C\delta||\vec{z}_1-\vec{z}_2||_{2,\beta}.
\end{equation}
By shrinking $\delta$, if necessary, we conclude that $\mathcal{C}$ is a contraction map, giving rise to a unique fixed point by the Banach space fixed point theorem.
\end{proof}
\begin{corollary}\label{c2}
Consider, 
\begin{align*}
h\in C^{2,\beta}(\text{Sym}(T^\star\mathbb{S}^2\otimes T^\star\mathbb{S}^2)),\,\,\,\gamma & \in C^{3,\beta}(\text{Sym}(T^\star\mathbb{S}^2\otimes T^\star\mathbb{S}^2)),\,\,\, \tilde\gamma\in C^{3,\beta}(\text{Sym}(T^\star\mathbb{S}^2\otimes T^\star\mathbb{S}^2)),\\
\vec{\upsilon}&\in C^{2,\beta}(\mathbb{S}^n,\mathbb{R}^{n+1}),\,\,\,\vec{r}\in C^{3,\beta}(\mathbb{S}^2,\mathbb{R}^{3}).
\end{align*}
Then, provided $\mathcal{P}(h,h)>0$, there exists $\epsilon,\delta>0$, such that any $||\gamma-\tilde\gamma||_{2,\beta}\leq \epsilon$ gives rise to a unique $\vec{y}\in C^{2,\beta}(\mathbb{S}^2,\mathbb{R}^{3})$, $||\vec{y}||_{2,\beta}\leq \delta$, and $\tau(\vec{y})\in C^{2,\beta}(T^\star\mathbb{S}^2)$, such that:
\begin{align*}
\tilde\gamma &=(r+y)^\star(\sigma),\\
\vec{y} &= d\vec{r}(\tau^\# (\vec{y}))+\phi(\vec{y})\vec{\upsilon},
\end{align*}
whereby $q(\vec{y},\nabla\vec{y}) = \tilde q(\vec{y},\nabla\vec{y}) -\frac{\mathcal{P}(\tilde q(\vec{y},\nabla\vec{y}),h)}{\mathcal{P}(h,h)}h$, for $\tilde q(\vec{y},\nabla\vec{y})$ given by (\ref{e10}), and 
$$\phi(\vec{y}) = \frac{\mathcal{P}(\tilde q(\vec{y},\nabla\vec{y})-2\text{Sym}(\nabla\tau(\vec{y})),h)}{\mathcal{P}(h,h)}.$$
\end{corollary}
\begin{proof}
The proof comes down to showing that $a\in C^{1,\beta}(\text{Sym}(T^\star\mathbb{S}^2\otimes T^\star\mathbb{S}^2))$ satisfying $\nabla\cdot a = d\tr_\gamma a$ and $\mathcal{P}(a,h) = 0$ is trivial. We do so by taking a sequence of smooth tensors $\{h_n\},\{\gamma_n\},\{q_n\}$ converging respectively to $h,\gamma, q(\vec{y},\nabla\vec{y})$ within the appropriate index $(\mathfrak{n},\beta')$-H\"{o}lder space, for $0<\beta'<\beta$. For each $n$, Proposition \ref{p4} enforces $q_n=L_{h_n}(\tau_n)$ in Proposition \ref{p3}. Exploiting the resulting Cauchy sequence in $\{\tau_n\}$ coming from Schauder estimates as in Theorem \ref{t1} then yields $q(\vec{y},\nabla\vec{y}) = L_h(\tau)$ for some $\tau \in C^{2,\beta'}(T^\star\mathbb{S}^2)$. Since the uniqueness result of Proposition \ref{p3} holds here also, we conclude $a=0$. 
\end{proof}

We take this opportunity to prove:
\begin{proposition}\label{p5}
Given $\delta,\epsilon>0$ from Corollary \ref{c2}, we consider a continuous path of metrics $t\to \gamma_t\in C^{2,\beta}(\text{Sym}(T^\star\mathbb{S}^2\otimes T^\star\mathbb{S}^2))$, $t\in[0,b]$, such that $||\gamma_t-\gamma||_{2,\beta}\leq \epsilon$ and each $\gamma_t$ furnishes a fixed point $\vec{y}_t\in C^{2,\beta}(\mathbb{S}^2,\mathbb{R}^3))$, $||\vec{y}_t||_{2,\beta}\leq \delta$. If the map $t\to\gamma_t$ is $\mathfrak{m}$-times Fr\'{e}chet differentiable, so is $t\to\vec{y}_t$. 
\end{proposition}
\begin{proof}
We begin by showing the result for the first derivative. For fixed points $\vec{y}_{t_0},\vec{y}_{t}$, a small modification to (\ref{e16}), using our Schauder estimates yield:
$$||\vec{y}_{t_0}-\vec{y}_{t}||_{2,\beta}\leq C||\gamma_{t_0}-\gamma_{t}||_{2,\beta}+C\delta||\vec{y}_{t_0}-\vec{y}_{t}||_{2,\beta}.$$
Since $C\delta<1$, we conclude $t\to\vec{y}_t$ is continuous with bound of the form $||\vec{y}_{t_0}-\vec{y}_{t}||\leq \frac{C}{1-C\delta}||\gamma_{t_0}-\gamma_{t}||$. Moreover, 
$$\Big|\Big|\frac{\vec{y}_{t_0}-\vec{y}_{t}}{t_0-t}\Big|\Big|_{2,\beta}\leq\frac{C}{1-C\delta}\Big|\Big|\frac{\gamma_{t_0}-\gamma_t}{t_0-t}\Big|\Big|_{2,\beta},$$
so the difference quotient $||\frac{\vec{y}_{t_0}-\vec{y}_t}{t_0-t}||_{2,\beta}$ remains bounded as $t\to t_0$. From (\ref{e10}), writing $\vec{y}_t=\tau^i_t\vec{r}_i+\phi_t\vec{\upsilon}$:
\begin{align*}
&\nabla_i\Big(\frac{\tau_{t_0}-\tau_{t_1}}{t_0-t_1}-\frac{\tau_{t_0}-\tau_{t_2}}{t_0-t_2}\Big)_j+\nabla_j\Big(\frac{\tau_{t_0}-\tau_{t_1}}{t_0-t_1}-\frac{\tau_{t_0}-\tau_{t_2}}{t_0-t_2}\Big
)_i+\Big(\frac{\phi_{t_0}-\phi_{t_1}}{t_0-t_1}-\frac{\phi_{t_0}-\phi_{t_2}}{t_0-t_2}\Big)h_{ij}\\
 &=\Big(\frac{\gamma_{t_0}-\gamma_{t_1}}{t_0-t_1}-\frac{\gamma_{t_0}-\gamma_{t_2}}{t_0-t_2}\Big)_{ij}+\frac{(q(\vec{y}_{t_0},\nabla\vec{y}_{t_0})-q(\vec{y}_{t_1},\nabla\vec{y}_{t_1}))_{ij}}{t_0-t_1}-\frac{(q(\vec{y}_{t_0},\nabla\vec{y}_{t_0})-q(\vec{y}_{t_2},\nabla\vec{y}_{t_2}))_{ij}}{t_0-t_2}.
\end{align*}
If we temporarily omit spherical indices, we observe from (\ref{e15}) that the last two terms decompose into the form:
\begin{align*}
\frac{q(\vec{y}_{t_0},\nabla\vec{y}_{t_0})-q(\vec{y}_{t_k},\nabla\vec{y}_{t_k})}{t_0-t_k}&=\sum_m Q_m(\vec{y}_{t_0},\vec{y}_{t_k},\nabla\vec{y}_{t_0},\nabla\vec{y}_{t_k})_{\mu\nu}\frac{y^\mu_{t_0}-y^\mu_{t_k}}{t_0-t_k}y^\nu_{t_{l_k(m)}}\\
&=\sum_m \tilde{Q}_m(\vec{y}_{t_k},\nabla\vec{y}_{t_k})_{\mu\nu} \frac{y^\mu_{t_0}-y^\mu_{t_k}}{t_0-t_k}y^\nu_{t_{l_k(m)}},
\end{align*}
for $k=1,2$, and $l_k(m)\in\{0,k\}$. Therefore:
\begin{align*}
&\frac{(q(\vec{y}_{t_0},\nabla\vec{y}_{t_0})-q(\vec{y}_{t_1},\nabla\vec{y}_{t_1}))}{t_0-t_1}-\frac{(q(\vec{y}_{t_0},\nabla\vec{y}_{t_0})-q(\vec{y}_{t_2},\nabla\vec{y}_{t_2}))}{t_0-t_2}\\
&=\sum_m \Big(\tilde{Q}_m(\vec{y}_{t_1},\nabla\vec{y}_{t_1})_{\mu\nu}-\tilde{Q}_m(\vec{y}_{t_2},\nabla\vec{y}_{t_2})_{\mu\nu}\Big)\frac{y^\mu_{t_0}-y^\mu_{t_1}}{t_0-t_1}y^\nu_{t_{l_1(m)}}\\
&\qquad+\sum_m\tilde{Q}_m(\vec{y}_{t_2},\nabla\vec{y}_{t_2})_{\mu\nu}\Big(\frac{y^\mu_{t_0}-y^\mu_{t_1}}{t_0-t_1}-\frac{y^\mu_{t_0}-y^\mu_{t_2}}{t_0-t_2}\Big)y^\nu_{t_{l_1(m)}}\\
&\qquad\qquad+\sum_m\tilde{Q}_m(\vec{y}_{t_2},\nabla\vec{y}_{t_2})\frac{y^\mu_{t_0}-y^\mu_{t_2}}{t_0-t_2}(y^\nu_{t_{l_1(m)}}-y^\nu_{t_{l_2(m)}}),
\end{align*}
where $y^\nu_{t_{l_1(m)}}-y^\nu_{t_{l_2(m)}} \in\{ 0,\,y^\nu_{t_1}-y^\nu_{t_2}\}$.\\ 
Now $\tau_{t_k}\perp\text{Ker}(L_h)$, ($k=0,1,2$), so again from Schauder estimates and the procedure leading to (\ref{e14}):
\begin{align*}
\Big|\Big|\frac{\vec{y}_{t_0}-\vec{y}_{t_1}}{t_0-t_1}-\frac{\vec{y}_{t_0}-\vec{y}_{t_2}}{t_0-t_2}\Big|\Big|_{2,\beta}&\leq C\Big|\Big|\frac{\gamma_{t_0}-\gamma_{t_1}}{t_0-t_1}-\frac{\gamma_{t_0}-\gamma_{t_2}}{t_0-t_2}\Big|\Big|_{2,\beta}\\
&+C_\star||\vec{y}_{t_1}-\vec{y}_{t_2}||_{2,\beta}\Big|\Big|\frac{\vec{y}_{t_0}-\vec{y}_{t_1}}{t_0-t_1}\Big|\Big|_{2,\beta}\text{max}(||\vec{y}_{t_0}||,||\vec{y}_{t_1}||)\\
&+C\Big|\Big|\frac{\vec{y}_{t_0}-\vec{y}_{t_1}}{t_0-t_1}-\frac{\vec{y}_{t_0}-\vec{y}_{t_2}}{t_0-t_2}\Big|\Big|_{2,\beta} \text{max}(||\vec{y}_{t_0}||,||\vec{y}_{t_1}||)\\
&+C\Big|\Big|\frac{\vec{y}_{t_0}-\vec{y}_{t_2}}{t_0-t_2}\Big|\Big|_{2,\beta}||\vec{y}_{t_1}-\vec{y}_{t_2}||_{2,\beta},
\end{align*}
where $C_\star$ is the only ``new" constant coming from Taylor approximations of the terms $\tilde{Q}_m(\vec{y}_{t_1},\nabla\vec{y}_{t_1})-\tilde{Q}_m(\vec{y}_{t_2},\nabla\vec{y}_{t_2})$. Since $C\delta<1$, we use this estimate and prior results to conclude therefore that any sequence $\{t_m\}$ such that $t_m\to t_0$, gives rise to a Cauchy sequence $\{\frac{\vec{y}_{t_0}-\vec{y}_{t_m}}{t_0-t_m}\}_m$. Moreover, the resulting limit $\dot{\vec{y}}_{t_0}=\dot{\tau}_{t_0}^i\vec{r}_i+\dot{\phi}_{t_0}\frac{\vec{r}}{r}$ is independent of $\{t_m\}$.\\

By induction we assume the path $t\to\vec{y}_t$, is $n$-times Fr\'{e}chet differentiable with bounds:
\begin{align*}
||\vec{y}_{t}^{(i)}-\vec{y}_{t_0}^{(i)}||_{2,\beta}&\leq \frac{C}{1-C\delta}||\gamma_{t}^{(i)}-\gamma_{t_0}^{(i)}||_{2,\beta}+\sum_{m=0}^{i-1} C_{m}^i||\gamma_{t}^{(m)}-\gamma_{t_0}^{(m)}||_{2,\beta},\,\,\,1\leq i\leq n-1,\\
||\vec{y}_t^{(n)}||_{2,\beta}&\leq \frac{C}{1-C\delta}||\gamma_t^{(n)}||_{2,\beta}+\sum_{m=0}^{n-1}C_{m}^n||\gamma_t^{(m)}||_{2,\beta},
\end{align*}
where $\vec{y}_t^{(i)}$ represents the $i^{th}$ derivative of $t\to\vec{y}_t$ with respect to the norm on $C^{2,\beta}(\mathbb{S}^2,\mathbb{R}^3)$.
Taking $n$ derivatives of (\ref{e10}) we have:
\begin{align}
\nabla_i(\tau_t^{(n)})_j&+\nabla_j(\tau_t^{(n)})_i+\phi_t^{(n)}h_{ij}=(\gamma_t^{(n)})_{ij}\notag\\
&-2\Big(\sigma_{\mu\nu}(\vec{r})(y_t)^\mu_i (y_t^{(n)})^\nu_j+(y_t^{(n)})^\eta y_t^\xi r^\mu_i r^\nu_j F_{\mu\nu\eta\xi}(\vec{r},\vec{y_t})+G_{\mu\nu\eta}(\vec{r},\vec{y}_t)(y_t)^\mu_i(y^{(n)}_t)_j\Big)\notag\\
&-G_{\mu\nu\eta}(\vec{r},\vec{y}_t)\Big((y^{(n)}_t)^\eta r^\mu_i (y_t)^\nu_j+y_t^\eta r^\mu_i (y^{(n)}_t)^\nu_j + (y^{(n)}_t)^\eta r^\nu_j (y_t)^\mu_i+y_t^\eta r^\nu_j (y^{(n)}_t)^\mu_i\Big)\notag\\
&-\partial_\rho F_{\mu\nu\eta\xi}(\vec{r},\vec{y}_t)(y_t^{(n)})^\rho y_t^\eta y_t^\xi r^\mu_i r^\nu_j\notag\\
&-\partial_\rho G_{\mu\nu\eta}(\vec{r},\vec{y}_t)(y^{(n)}_t)^\rho y_t^\eta \Big( r^\mu_i(y_t)^\nu_j+r^\nu_j(y_t)^\mu_i+(y_t)^\mu_i(y_t)^\nu_j\Big)\notag\\
&+P(\vec{y}_t^{(n-1)},\vec{y}_t^{(n-2)},...,\vec{y}_t^{(1)},\vec{y}_t)_{ij},\label{e17}
\end{align}
where the polynomial $P$ is the sum of monomials of degree at least two, and coefficients consisting of $\sigma_{\mu\nu}(\vec{r}),F_{\mu\nu\eta\xi}(\vec{r},\vec{y}_t),G_{\mu\nu\eta}(\vec{r},\vec{y}_t),$ and their partial derivatives of order upto $n$. Since $\tau_t^{(n)},\tau_{t_0}^{(n)}\perp\text{Ker}(L_h)$ for any $t,t_0<b$, we observe, again suppressing spherical indices and also conflating similar terms:
\begin{align*}
2\text{Sym}\Big(\nabla(\tau^{(n)}_t-\tau^{(n)}_{t_0})\Big)&+(\phi^{(n)}_t-\phi^{(n)}_{t_0})h=\gamma_t^{(n)}-
\gamma_{t_0}^{(n)}\\
&+Q(\vec{y}_t,\nabla\vec{y}_t)_{\mu\nu}y_t^\mu\Big((y_t^{(n)})^\nu-(y_{t_0}^{(n)})^\nu\Big)\\
&+\Big(Q(\vec{y}_{t},\nabla\vec{y}_{t})_{\mu\nu}-Q(\vec{y}_{t_0},\nabla\vec{y}_{t_0})_{\mu\nu}\Big)y_t^\mu(y_t^{(n)})^\nu\\
&+Q(\vec{y}_{t_0},\nabla\vec{y}_{t_0})\Big((y_t)^\mu-(y_{t_0})^\mu\Big)(y_{t_0}^{(n)})^\nu\\
&+P(\vec{y}_t^{(n-1)},...,\vec{y}_t)-P(\vec{y}_{t_0}^{(n-1)},...,\vec{y}_{t_0}).
\end{align*}
From this, our Schauder estimates yield:
\begin{align*}
||\vec{y}_t^{(n)}-\vec{y}_{t_0}^{(n)}||_{2,\beta}&\leq C||\gamma^{(n)}_t-\gamma^{(n)}_{t_0}||_{2,\beta}+C\delta||\vec{y}_t^{(n)}-\vec{y}_{t_0}^{(n)}||_{2,\beta}+C_\star\delta||\vec{y}_t^{(n)}||_{2,\beta}||\vec{y}_t-\vec{y}_{t_0}||_{2,\beta}\\
&+C||\vec{y}_{t_0}^{(n)}||_{2,\beta}||\vec{y}_t-\vec{y}_{t_0}||_{2,\beta}+|| P(\vec{y}_t^{(n-1)},...,\vec{y}_t)-P(\vec{y}_{t_0}^{(n-1)},...,\vec{y}_{t_0}) ||_{2,\beta},
\end{align*}
we conclude with an analogous continuity bound for $t\to\vec{y}_t^{(n)}$ as for lower derivatives, given the structure of the polynomial $P$. Moreover, the ratio $\frac{||\vec{y}_t^{(n)}-\vec{y}_{t_0}^{(n)}||}{t-t_0}$ remains bounded as $t\to t_0$. It is also evident from the expression above, we have with a similar abuse of notation:
\begingroup
\allowdisplaybreaks
\begin{align*}
2\text{Sym}\Big(\nabla\Big(\frac{\tau_{t_1}^{(n)}-\tau_{t_0}^{(n)}}{t_1-t_0}&-\frac{\tau_{t_2}^{(n)}-\tau_{t_0}^{(n)}}{t_2-t_0}\Big)\Big)+ \Big(\frac{\phi^{(n)}_{t_1}-\phi^{(n)}_{t_0}}{t_1-t_0}- \frac{\phi^{(n)}_{t_2}-\phi^{(n)}_{t_0}}{t_2-t_0}\Big)h=\frac{\gamma_{t_1}^{(n)}-\gamma_{t_0}^{(n)}}{t_1-t_0}-\frac{\gamma_{t_2}^{(n)}-\gamma_{t_0}^{(n)}}{t_1-t_0}\\
&+Q(\vec{y}_{t_1},\nabla\vec{y}_{t_1})_{\mu\nu}y_{t_1}^\mu\Big(\frac{(y_{t_1}^{(n)})^\nu-(y_{t_0}^{(n)})^\nu}{t_1-t_0}-\frac{(y_{t_2}^{(n)})^\nu-(y_{t_0}^{(n)})^\nu}{t_2-t_0}\Big)\\
&+Q(\vec{y}_{t_1},\nabla\vec{y}_{t_1})\Big(y_{t_1}^\mu-y_{t_2}^\mu\Big)\frac{(y_{t_2}^{(n)})^\nu-(y_{t_0}^{(n)})^\nu}{t_2-t_0}\\
&+\Big(Q(\vec{y}_{t_1},\nabla\vec{y}_{t_1})-Q(\vec{y}_{t_2},\nabla\vec{y}_{t_2})\Big)y_{t_2}^\mu\frac{(y_{t_2}^{(n)})^\nu-(y_{t_0}^{(n)})^\nu}{t_2-t_0}\\
&+\Big(\frac{Q(\vec{y}_{t_1},\nabla\vec{y}_{t_1})-Q(\vec{y}_{t_0},\nabla\vec{y}_{t_0})}{t_1-t_0}-\frac{Q(\vec{y}_{t_2},\nabla\vec{y}_{t_2})-Q(\vec{y}_{t_0},\nabla\vec{y}_{t_0})}{t_2-t_0}\Big)y_{t_1}^\mu(y_{t_1}^{(n)})^\nu\\
&+\frac{Q(\vec{y}_{t_2},\nabla\vec{y}_{t_2})-Q(\vec{y}_{t_0},\nabla\vec{y}_{t_0})}{t_2-t_0}\Big(y_{t_1}^\mu-y_{t_2}^\mu\Big)(y_{t_1}^{(n)})^\nu\\
&+ \frac{Q(\vec{y}_{t_2},\nabla\vec{y}_{t_2})-Q(\vec{y}_{t_0},\nabla\vec{y}_{t_0})}{t_2-t_0}y_{t_2}^\mu\Big((y_{t_1}^{(n)})^\nu-(y_{t_2}^{(n)})^\nu\Big)\\
&+Q(\vec{y}_{t_0},\nabla\vec{y}_{t_0})(y_{t_0}^{(n)})^\nu\Big(\frac{y_{t_1}^\mu-y_{t_0}^\mu}{t_1-t_0}-\frac{y_{t_2}^\mu-y_{t_0}^\mu}{t_2-t_0}\Big)\\
&+\frac{P(\vec{y}_{t_1}^{(n-1)},...,\vec{y}_{t_1})-P(\vec{y}_{t_0}^{(n-1)},...,\vec{y}_{t_0})}{t_1-t_0}-\frac{P(\vec{y}_{t_2}^{(n-1)},...,\vec{y}_{t_2})-P(\vec{y}_{t_0}^{(n-1)},...,\vec{y}_{t_0})}{t_2-t_0}.
\end{align*}
\endgroup
Finally, with the expression above we now proceed similarly as the case of the first derivative. We leave it for the reader to observe, utilizing all our assumed bounds, Schauder estimates again confirm that for any sequence $\{t_m\}$ such that $t_m\to t_0$ we have a Cauchy sequence $\{\frac{\vec{y}_{t_m}^{(n)}-\vec{y}_{t_0}^{(n)}}{t_m-t_0}\}_m$. Again the limit $\vec{y}^{(n+1)}_{t_0}$ is independent of the sequence, moreover, the hypothesized bounds for $\{\vec{y}^{(n)}_t\}$ also follow from the expression above. 
\end{proof}

\section{Null Cone: Openness}
\subsection{Initial Setup}
We will now assume an ambient geometry $(\mathcal{M}^n,g)$ with non-degenerate metric $g$. For convenience we will denote $\langle X,Y\rangle: =g(X,Y)$ for $X,Y\in \Gamma(T\mathcal{M})$. This metric induces an ambient Levi-Civita connection denoted by $D$, and consequently we can calculate the ambient Riemannian curvature tensor, for any $X,Y,Z\in \Gamma(T\mathcal{M})$:
$$R_{XY}Z:=D_{[X,Y]}Z-[D_X,D_Y]Z.$$
Taking a local orthonormal frame $\{e_\mu|0\leq \mu\leq n\}\subset T_p\mathcal{M}$, we also define the ambient Einstein tensor point-wise:
$$G(X,Y):=\sum_{\mu}\langle R_{X,e_\mu}Y,e_\mu\rangle.$$
\indent We now take $\Omega\subset\mathcal{M}$ to be a smooth, orientable, and connected hypersurface within an ambient geometry $(\mathcal{M},g)$. We say $\Omega$ is a \textit{null hypersurface} whenever the induced metric $g|_\Omega$ is degenerate. Alternatively, from the orientability assumption, we can find a non-vanishing vector field $\ubar L\in\Gamma(T\Omega)$, such that $X\in\Gamma(T\Omega)$ if and only if $\langle \ubar L,X\rangle = 0$. Non degeneracy of the ambient metric $g$ therefore enforces $\text{span}(\ubar L_p) = T_p^\perp\Omega\subset T_p\Omega$, for any $p\in\Omega$. Since $\Omega$ is a hypersurface, any $p\in\Omega$ admits a neighborhood $U_p\subset \mathcal{M}$, and a function $\mathfrak{f}$ on $U_p$ (denoted $\mathfrak{f}\in\mathcal{F}(U_p)$) such that $V_p:=\Omega\cap U_p = \{\mathfrak{f}=0\}$, and the gradient $\text{grad}(\mathfrak{f})\in\Gamma(T^\perp V_p)$ is nowhere vanishing. It follows that $\ubar L = f_1\text{grad}(\mathfrak{f})$ for some smooth $f_1\neq 0$ on $V_p$, giving $\langle \text{grad}(\mathfrak{f}), \text{grad}(\mathfrak{f})\rangle|_{V_p} \equiv 0$. We say $\ubar L$ is a \textit{null vector field}. If we recall the famous identity, $D_{\text{grad}(f)} \text{grad}(f) = \frac12\text{grad}\langle\text{grad}(f), \text{grad}(f)\rangle$, it follows that $\langle X,D_{\text{grad}(\mathfrak{f})}\text{grad}(\mathfrak{f})\rangle = 0$ for any $X\in\Gamma(T V_p)$. Therefore, $D_{\text{grad}(\mathfrak{f})}\text{grad}(\mathfrak{f}) = f_2\text{grad}(\mathfrak{f})$ on $V_p$, giving in turn $D_{\ubar L}\ubar L = \kappa\ubar L$, for smooth functions $f_2,\kappa\in\mathcal{F}(V_p)$. We therefore conclude that integral curves of $\ubar L$, under a suitable re-parametrization, are null geodesics of $\mathcal{M}$. We therefore conclude that $\Omega$ is in-fact a congruence of null geodesics.\\
\indent We now assume the existence of an embedded 2-sphere, $\iota:\mathbb{S}^2\hookrightarrow\Omega$, $\iota(\mathbb{S}^2)=\Sigma_0$, such that $g|_{\Sigma_0}$ is a Riemannian metric. We also assume that any integral curve of $\ubar L$ intersects $\Sigma_0$ precisely once, forcing our geometry $(\mathcal{M},g)$ to be a Lorentzian four-manifold with co-dimension one null submanifold $\Omega\subset \mathcal{M}$ of dimension three. This gives rise to a natural submersion $\pi:\Omega\to\Sigma_0$ taking $p\in\Omega$ to the intersection with $\Sigma_0$ of the integral curve $\beta_p^{\ubar L}$ of $\ubar L$, for which $\beta_p^{\ubar L}(0) = p$. Given $\ubar L$ and a constant $s_0$, we may then construct a smooth function $s\in\mathcal{F}(\Omega)$ by imposing $\ubar L(s)=1$ and $s|_{\Sigma_0}= s_0$. For $q\in\Sigma_0$, we take the open interval $(s_-(q),s_+(q))$ to represent the range of $s$ along $\beta_q^{\ubar L}$, and denote by ${S_-}:=\sup_{\Sigma_0}s_-$, ${S_+}:=\inf_{\Sigma_0}s_+$. We now notice that the interval $(S_-,S_+)$ is non-empty. Given that $\ubar L(s)=1$, the Implicit Function Theorem implies for $t\in(S_-,S_+)$, a spacelike embedding $\Sigma_t:= \{q\in\Omega|s(q)=t\}$ and a diffeomorphism $\pi|_{\Sigma_t}:\Sigma_t\to\Sigma_0$. More generally, for any continuous function $\omega$ defined on $\Sigma_0$, provided $\sup_{\Sigma_0}\omega\leq S_+$, $\inf_{\Sigma_0}\omega \geq S_-$, we obtain a homeomorphism $\iota_\omega:\Sigma_0\to \Sigma_\omega\subset \Omega$ according to $(s_0,q)\to_{\iota_\omega}(\omega(q),q)$. We will call these spherical embeddings \textit{cross-sections} of $\Omega$. Clearly the regularity of any such $\Sigma_\omega$ is dictated by the function $\omega$. Conversely, any Riemannian spherical embedding $\Sigma\subset\Omega$ is uniquely characterized as a graph over $\Sigma_0$ with graph function $s\circ(\pi|_{\Sigma})^{-1}$. We note for $s<S_-$ or $s>S_+$, in the case that $\Sigma_s$ is non-empty, such surfaces remain smooth but may no longer be connected. Nevertheless, the collection $\{\Sigma_s\}$ gives a foliation of $\Omega$. We highlight that the foliation $\{\Sigma_s\}\subset \Omega$ is dependent upon our choice of $\ubar L$. Any other viable candidate would have to be a re-scaling of $\ubar L$, say to $\ubar L_a:=a\ubar L$, for some $a\neq 0$, whereby $D_{\ubar L_a}\ubar L_a = a(\ubar L\log |a|+\kappa)\ubar L_a$. Consequently, a foliation $\{\Sigma^a_s\}\subset \Omega$, will be called a \textit{geodesic foliation} whenever $\ubar L\log |a| = -\kappa$, equivalently, $a(s,q) = a_0(q)e^{-\int_{s_0}^s\kappa(u,q)du}$, for $q\in \Sigma_{s_0}$. In particular, if we re-scale to $\kappa\equiv 0$, namely take $\ubar L$ to be geodesic, then any other geodesic foliation $\{\Sigma_\lambda\}$ is given by $s|_{\Sigma_\lambda} = (g_1\circ\pi)\lambda+(g_2\circ\pi)$ for functions $g_1,g_2$ on $\Sigma_0$. Once again, the regularity of the geodesic foliation $\{\Sigma_\lambda\}$ depends on the regularity of the functions $g_1,g_2$.  
\begin{definition}\label{d2}
We will define the second fundamental form of $\Omega$ relative to $\ubar L$ by: 
$$\ubar\chi(X,Y):=\langle D_X\ubar L,Y\rangle$$
for $X,Y\in \Gamma(T\Omega)$. 
\end{definition}
By an abuse of notation, for a cross-section $\Sigma\subset\Omega$, we will also denote by $\ubar\chi$ the induced second fundamental form on $\Sigma$. This turns out to be a rather small abuse of notation as we will now explain. Since $X\in \Gamma(T\Omega)$ if and only if $\langle \ubar L,X\rangle=0$, and $\ubar L$ is null, we observe $\ubar L\in\Gamma(T\Omega)\cap\Gamma(T^\perp\Omega)$. So $\ubar L$ is both `tangent' and `normal' to $\Omega$. As a result, at any $p\in\Omega$:
\begin{align*}
\langle X,Y+c\ubar L\rangle &= \langle X,Y\rangle,\\
\ubar\chi(X,Y+c\ubar L)&=\ubar\chi(X,Y),
\end{align*}
for any $X,Y\in T_p\Omega$, $c\in\mathbb{R}$. So both the metric and the second fundamental form are determined `modulo pointwise contributions by $\ubar L$'. Said differently, both $g|_\Omega,\ubar\chi$ are fully determined on points of a cross-section $\Sigma$ by their restrictions $g|_\Sigma,\ubar\chi|_\Sigma$. It also follows that $g|_\Omega,\ubar\chi$ are fully determined throughout $\Omega$ by their restrictions along a foliation $\{\Sigma_s\}$ of $\Omega$, which we will denote $\gamma_s:=g|_{\Sigma_s}$, $\ubar\chi_s:=\ubar\chi|_{\Sigma_s}$. Given such a foliation, we may also construct the function $\tr\ubar\chi(p):=\tr_{\gamma_{s(p)}}\ubar\chi_{s(p)}$, called the \textit{null expansion}. We observe the null expansion is solely dependent upon our choice of $\ubar L$. For a cross-section $\Sigma\subset \Omega$, denoting $\gamma:=g|_\Sigma$, we observe $\tr\ubar\chi(q) = \tr_\gamma\ubar\chi(q)$ for any $q\in \Sigma$. This function is called the null expansion, since on $\Sigma$ we observe: 
$$\ubar\chi = \frac12\pounds_{\ubar L}\gamma,$$
so by variation of the area along $\ubar L$, it follows that $\delta_{\ubar L}dA = \tr\ubar\chi dA$. In other-words, $\tr\ubar\chi(p)$ measures the pointwise area expansion of cross-sections through the point $p$ along $\ubar L$. In the next section we will need the following propagation equations:

\begin{lemma}\label{l} Along the foliation $\{\Sigma_s\}\subset\Omega$, with data $\gamma_s,\ubar\chi_s$ associated to $\ubar L$:
\begin{align}
\pounds_{\ubar L}\gamma_s &= 2\ubar\chi_s\label{i1}\\
\pounds_{\ubar L}\ubar\chi_s &= -\ubar\alpha_s + \ubar\chi^2_s+\kappa\ubar\chi_s\label{i2}\\
\ubar L\tr\ubar\chi_s &= -\frac12(\tr\ubar\chi_s)^2 - |\hat{\ubar\chi}_s|^2 - G(\ubar L,\ubar L) +\kappa \tr\ubar\chi_s\label{i3}
\end{align}
where $\ubar\alpha_s(V,W) = \langle R_{\ubar L V}\ubar L,W\rangle$, $G(\ubar L,\ubar L):=\tr_{\gamma_s}\ubar\alpha_s$.
\end{lemma}
\begin{proof}
See, for example \cite{R,G}.
\end{proof}
\begin{definition}\label{d3}\
\begin{enumerate}
\item We will call $\Omega$ a Null Cone, if $\tr\ubar\chi>0$ throughout $\Omega$.
\item We will call a Null Cone $\Omega$ convex, if $\frac12(\tr\ubar\chi)^2>|\hat{\ubar\chi}|^2$ throughout $\Omega$.
\end{enumerate}
\end{definition}
We will now discard components of $\Omega$ satisfying $s<S_-$, or $s>S_+$. We may also re-parametrize if necessary, so that $S_->0$. If we identify $\{s_0\}\times\mathbb{S}^2$ with $\Sigma_{s_0}$, it follows that we have a global diffeomorphism $B:(S_-,S_+)\times\mathbb{S}^2\to \Omega$, given by $B(s,p) = \beta^{\ubar L}_p(s)$. It also follows from our construction in Section 2, that $(\Omega,g)$ can be equivalently viewed as $(\mathcal{A},\sigma)$, $\mathcal{A}=B_{S_+}(0)-\bar{B}_{S_-}(0)\subset\mathbb{R}^3$, whereby $B_{s} (0)(\bar{B}_s(0))$ represents the open(closed) ball of radius $s$ centered at the origin in $\mathbb{R}^3$, and $\sigma = \tilde B^\star(g|_{\Omega})$. For a differentiable function $f$, we observe that $dB(\partial_s)f|_{(s,p)} = \partial_s f(\beta^p_{\ubar L}(s)) = \ubar L(f)|_{\beta^{\ubar L}_p(s)}$. Therefore, the pull-back of $\ubar L$ to $\mathcal{A}$ is realized via $d\Phi(\partial_s)f|_{(s,p)} = \partial_s f(s\frac{\vec{z}}{|\vec{z}|}\circ\iota(p)) = \frac{z^\mu}{|\vec{z}|}(\partial_\mu f)(s\frac{\vec{z}}{|\vec{z}|}\circ\iota(p))$, equivalently, $d\tilde {B}(\frac{\vec{z}}{|\vec{z}|}) = \ubar L$. If we consider an embedding $r:\mathbb{S}^2\hookrightarrow\mathcal{A}$, the pull-back $r^\star(\sigma)$ will be Riemannian if and only if $\frac{\vec{z}}{|\vec{z}|}\circ r(p)\notin dr(T_p\mathbb{S}^2)$ for every $p\in\mathbb{S}^2$. Equivalently, $r(\mathbb{S}^2)\subset\mathcal{A}\subset\mathbb{R}^3$ must be the boundary of a star-shaped region relative to the origin in $\mathbb{R}^3$. Consequently, a natural way to construct a family of embeddings about $r:\mathbb{S}^2\to\mathcal{A}$, is to take $R:(-\epsilon,\epsilon)\times\mathbb{S}^2\to \mathcal{A}$ whereby $\vec{R}(\lambda,p) = \vec{r}(p)+\lambda\frac{\vec{z}}{|\vec{z}|}\circ r(p)$, thus $F := \tilde B\circ R$.\\ 
\indent For $\vec{r}\in C^{3,\beta}(\mathbb{S}^2,\mathbb{R}^3)$, such that the associated mapping $r:\mathbb{S}^2\to\mathcal{A}$ is star-shaped, we observe $\vec{\upsilon} = \frac{\vec{r}}{|\vec{r}|}\in C^{3,\beta}(\mathbb{S}^2,\mathbb{R}^3)$. It also follows that $h = \pounds_{\partial_\lambda}(F^\star(g))|_{\{\lambda=0\}} = 2F^\star(\ubar\chi)|_{\{\lambda=0\}}\in C^{2,\beta}(\text{Sym}(T^\star\mathbb{S}^2\otimes T^\star\mathbb{S}^2))$. As a consequence of Corollary \ref{c2}:
\begin{theorem} \label{t2}
Consider a convex Null Cone, $(\mathcal{A},\sigma)$. Consider also $\gamma\in C^{3,\beta}(\text{Sym}(T^\star\mathbb{S}^2\otimes T^\star\mathbb{S}^2))$ with isometric embedding $r:(\mathbb{S}^2,\gamma)\hookrightarrow (\mathcal{A},\sigma)$ represented by some $\vec{r}\in C^{3,\beta}(\mathbb{S}^2,\mathbb{R}^3)$. Then, there exists $\epsilon,\delta>0$, such that any $||\tilde\gamma-\gamma||_{2,\beta}\leq \epsilon$ gives rise to a unique $\vec{y}\in C^{2,\beta}(\mathbb{S}^2,\mathbb{R}^{3})$, $||\vec{y}||_{2,\beta}\leq\delta$, and $\tau(\vec{y})\in C^{2,\beta}(T^\star\mathbb{S}^2)$, such that:
\begin{align*}
\tilde\gamma &= (r+y)^\star(\sigma),\\
\vec{y} &= d\vec{r}(\tau^\# (\vec{y}))+\phi(\vec{y})\frac{\vec{r}}{|\vec{r}|},
\end{align*}
whereby $q(\vec{y},\nabla\vec{y}) = \tilde q(\vec{y},\nabla\vec{y}) -\frac{\mathcal{P}(\tilde q(\vec{y},\nabla\vec{y}),h)}{\mathcal{P}(h,h)}h$, for $\tilde q(\vec{y},\nabla\vec{y})$ given by (\ref{e10}), and 
$$\phi(\vec{y}) = \frac{\mathcal{P}(\tilde q(\vec{y},\nabla\vec{y})-2\text{Sym}(\nabla\tau(\vec{y})),h)}{\mathcal{P}(h,h)}.$$
\end{theorem}	

\subsection{Regularity}
In this section we wish to show that the isometric embedding inherits regularity directly from the metric of the embedding. For convenience, yet with an abuse of notation, we will drop the subscript in $g|_\Omega$ in this section as we will not be referencing the ambient geometry.
\begin{lemma}\label{l5}
Given the background metric $g(s,y)$ of $\Omega$, and $\omega(y)\in C^{2,\beta}(\mathbb{S}^2)$ such that $S_-\leq \omega\leq S_+$. We observe the Gauss curvature, $\mathcal{K}_\omega$, of the cross-section $\Sigma_\omega:=\{s=\omega\}$ associated to the metric $\gamma:=g(\omega(y),y)$: 
$$\mathcal{K}_\omega = -\nabla\cdot\Big(\tr\ubar\chi_\omega\nabla\omega-\ubar\chi_\omega(\nabla\omega)\Big)-(d_s\tr\ubar\chi_s-\nabla_s\cdot\ubar\chi_s)(\nabla\omega) +\mathcal{K}(\omega).$$
Where $\nabla$, $\ubar\chi_\omega$ refers to data on $\Sigma_\omega$, $\nabla_s$, $\ubar\chi_s$ to the background data on $\Sigma_s$.
\end{lemma}
\begin{proof}
From the outset, we may pick a coordinate system at a given point of $\Sigma_\omega$ so that the connection coefficients for the metric $\gamma = g(\omega(y),y)$, satisfies ${^\gamma}\Gamma^k_{ij} = 0$. From $\partial_k(g_{ij}(\omega,y)) = 2\ubar\chi_{ij}(\omega,y)\partial_k\omega+g_{ij,k}(\omega,y)$, we conclude pointwise that:
\begin{align*}
0={^\gamma}\Gamma^k_{ij} &= \frac12\gamma^{kl}\Big(-\partial_l(g_{ij}(\omega,y))+\partial_i(g_{jl}(\omega,y))+\partial_j(g_{li}(\omega,y)\Big)\\
&=-\ubar\chi_{ij}\nabla^k\omega+\ubar\chi_{j}^k\omega_i+\ubar\chi^k_i\omega_j+{^g}\Gamma^k_{ij},
\end{align*}
giving, ${^g}\Gamma^k_{ij} = \ubar\chi_{ij}\nabla^k\omega-\ubar\chi_{j}^k\omega_i-\ubar\chi^k_i\omega_j$. We observe also from (\ref{i1},\ref{i2}):
\begin{align*}
\frac12\frac{\partial^2g_{il}(\omega,y)}{\partial y^k\partial y^m} &= \partial_k(\ubar\chi_{il}\omega_m)+\ubar\chi_{il,m}\omega_k+\frac12g_{il,km}\\
&=\nabla_k(\ubar\chi_{il}\omega_m)+\nabla_m(\ubar\chi_{il}\omega_k)-\ubar\chi_{il}\nabla^2_{km}\omega-(-\ubar\alpha_{il}+\ubar\chi^2_{il}+\kappa\ubar\chi_{il})\omega_m\omega_k+\frac12g_{il,km}.
\end{align*}

Expressing the Curvature tensor in partial derivatives of the metric we obtain:
\begin{align*}
{^\gamma}{R}_{iklm}&=\frac12\Big(\frac{\partial^2g_{im}(\omega,y)}{\partial y^k\partial y^l}+\frac{\partial^2g_{kl}(\omega,y)}{\partial y^i\partial y^m}-\frac{\partial^2g_{il}(\omega,y)}{\partial y^k\partial y^m}-\frac{\partial^2g_{km}(\omega,y)}{\partial y^i\partial y^l}\Big)\\
&=\nabla_k(\ubar\chi_{im}\omega_l)+\nabla_l(\ubar\chi_{im}\omega_k)-\ubar\chi_{im}\nabla^2_{kl}\omega-(-\ubar\alpha_{im}+\ubar\chi^2_{im}+\kappa\ubar\chi_{im})\omega_k\omega_l\\
&\quad+ \nabla_i(\ubar\chi_{kl}\omega_m)+\nabla_m(\ubar\chi_{kl}\omega_i)-\ubar\chi_{kl}\nabla^2_{im}\omega-(-\ubar\alpha_{kl}+\ubar\chi^2_{kl}+\kappa\ubar\chi_{kl})\omega_i\omega_m\\
&\quad-\nabla_k(\ubar\chi_{il}\omega_m)-\nabla_m(\ubar\chi_{il}\omega_k)+\ubar\chi_{il}\nabla^2_{km}\omega+(-\ubar\alpha_{il}+\ubar\chi^2_{il}+\kappa\ubar\chi_{il})\omega_m\omega_k\\
&\quad- \nabla_i(\ubar\chi_{km}\omega_l)-\nabla_l(\ubar\chi_{km}\omega_i)+\ubar\chi_{km}\nabla^2_{il}\omega+(-\ubar\alpha_{km}+\ubar\chi^2_{km}+\kappa\ubar\chi_{km})\omega_i\omega_l\\
&+{^g}R_{iklm}+\gamma_{np}\Big({^g}\Gamma^n_{km}{^g}\Gamma^p_{il}-{^g}\Gamma^n_{kl}{^g}\Gamma^p_{im}\Big).
\end{align*}
Taking a trace over the `$il$' and `$km$' indices:
\begin{align*}
{{^\gamma}R}&=-4\nabla\cdot(\tr\ubar\chi\nabla\omega)+4\nabla\cdot(\ubar\chi(\nabla\omega))-2\ubar\chi\cdot\cdot\nabla^2\omega+2\tr\ubar\chi\Delta\omega\\
&\quad-2(-\ubar\alpha(\nabla\omega,\nabla\omega)+\ubar\chi^2(\nabla\omega,\nabla\omega)+\kappa\ubar\chi(\nabla\omega,\nabla\omega))+2(-G(\ubar L,\ubar L)+|\ubar\chi|^2+\kappa\tr\ubar\chi)|\nabla\omega|^2\\
&\quad+{^g}R+\Big(\Big|-\tr\ubar\chi\nabla\omega+2\ubar\chi(\nabla\omega)\Big|^2-\Big|-\ubar\chi_{im}\omega_q+\ubar\chi_{mq}\omega_i+\ubar\chi_{qi}\omega_m\Big|^2\Big)\\
&=-2\nabla\cdot\Big(\tr\ubar\chi\nabla\omega-\ubar\chi(\nabla\omega)\Big)-2(\nabla\tr\ubar\chi-\nabla\cdot\ubar\chi)(\nabla\omega)\\
&\quad-2 (-\ubar\alpha(\nabla\omega,\nabla\omega)+\ubar\chi^2(\nabla\omega,\nabla\omega)+\kappa\ubar\chi(\nabla\omega,\nabla\omega))+2(-G(\ubar L,\ubar L)+|\ubar\chi|^2+\kappa\tr\ubar\chi)|\nabla\omega|^2\\
&\quad+{^g}R+\Big((\tr\ubar\chi)^2|\nabla\omega|^2-4\tr\ubar\chi\ubar\chi(\nabla\omega,\nabla\omega)+6\ubar\chi^2(\nabla\omega,\nabla\omega)-3|\ubar\chi|^2|\nabla\omega|^2\Big).
\end{align*}
We now observe from (\ref{i2},\ref{i3})
\begin{align*}
\nabla\tr\ubar\chi &= (\frac{d}{ds}\tr\ubar\chi)\nabla\omega+\nabla_s\tr\ubar\chi_s=-(G(\ubar L,\ubar L)+|\ubar\chi|^2-\kappa\tr\ubar\chi)\nabla\omega+\nabla_s\tr\ubar\chi_s\\
\nabla\cdot\ubar\chi &=(\pounds_{\ubar L}\ubar\chi)(\nabla\omega)+\nabla_s\cdot\ubar\chi_s+\Big(-|\ubar\chi|^2d\omega+\tr\ubar\chi\ubar\chi(\nabla\omega)-2\ubar\chi^2(\nabla\omega)\Big)\\
&=\nabla_s\cdot\ubar\chi_s-\ubar\alpha(\nabla\omega)-\ubar\chi^2(\nabla\omega)+\kappa\ubar\chi(\nabla\omega)-|\ubar\chi|^2d\omega+\tr\ubar\chi\ubar\chi(\nabla\omega),
\end{align*}
where, in the parentheses of the second equality, we corrected for the connection terms. Substituting these two identities into the expression above, and using the fact $\ubar\chi^2(\nabla\omega,\nabla\omega)-\tr\ubar\chi\ubar\chi(\nabla\omega,\nabla\omega) = \frac12\Big(|\ubar\chi|^2-(\tr\ubar\chi)^2\Big)|\nabla\omega|^2$, we obtain the desired relation up to a factor of two.
\end{proof}
\begin{proposition}\label{p6}
Suppose we have an embedding $r:\mathbb{S}^2\hookrightarrow\mathcal{A}$ into a convex Null Cone $(\mathcal{A},\sigma)$, which induces $\vec{r}\in C^{2,\beta}(\mathbb{S}^2,\mathbb{R}^3)$. Then, for $\mathfrak{n}\in\mathbb{N}$, $\mathfrak{n}\geq 2$, if $\gamma := r^\star(\sigma)\in C^{\mathfrak{n},\beta}(\text{Sym}(T^\star\mathbb{S}^2\otimes T^\star\mathbb{S}^2))$ is a Riemannian metric, we have
$$\vec{r}\in C^{\mathfrak{n},\beta}(\mathbb{S}^2,\mathbb{R}^3).$$
Moreover, $||\vec{r}||_{\mathfrak{n},\beta}\leq C(||\gamma||_{\mathfrak{n},\beta},||\vec{r}||_{1,\beta})$.
\end{proposition}
\begin{proof}
We recall the diffeomorphism $B:(S_-,S_+)\times\mathbb{S}^2\to\Omega$. We denote by $\tilde\pi:(S_-,S_+)\times\mathbb{S}^2\to\mathbb{S}^2$ the canonical projection. Given any cross-section $\Sigma\subset\Omega$, we have the diffeomorphism $\tilde\pi\circ(B^{-1}|_{\Sigma_{s_0}})\circ\pi:\Sigma\to\mathbb{S}^2$. If we choose a local coordinate chart $(x^i,\mathcal{U})$ on $\mathbb{S}^2$ this diffeomorphism also induces the same coordinate chart on $\Sigma$. For convenience, we will shortly omit any further mention of the mapping $\tilde\pi\circ(B^{-1}|_{\Sigma_{s_0}})\circ\pi $ and refer to the coordinate neighborhood $(x^i,\mathcal{U})$ as a shared chart amongst all cross-sections of $\Omega$. We also observe $\Psi:= \tilde\pi\circ(B^{-1}|_{\Sigma_{s_0}})\circ\pi\circ \tilde B\circ r = \tilde\pi\circ \Phi^{-1}\circ r\in C^{2,\beta}\mathcal{D}(\mathbb{S}^2)$, a $C^{2,\beta}$ diffeomorphism of $\mathbb{S}^2$. It follows, in the local coordinates $(s,y^i)\in I\times \mathcal{U}$ for $\Omega$, that $r:(\gamma,\mathbb{S}^2)\hookrightarrow(g,\Omega)$ takes the form:
$$g(\varphi(x),\Psi(x))_{kl}\Psi^k_{,i}\Psi^l_{,j} = \gamma_{ij},$$
or $g_{kl}\Psi^k_{,i}\Psi^l_{,j} = \gamma_{ij}$ for short, whereby $\Psi^k_{,i}(x):=\frac{\partial (y^k \circ\Psi)}{\partial x^i}(x)$, $\varphi: = |\vec{r}|$. We also observe the the function $\omega = \varphi\circ\Psi^{-1}$ uniquely identifying the cross-section $\Sigma\subset\Omega$. From the above expression, we also observe:
$$(d\Psi^{-1})^i_{j} = \gamma^{ik}\Psi^n_{,k}g_{nj},$$
having used the fact that $\Psi^i_{,j},g_{kl}$ are invertible two-by-two matrices in a local coordinate neighborhood. We calculate:
\begin{align*}
g_{kl}\Psi^k_{,im}\Psi^l_{,j}&=\gamma_{ij,m}-\frac{\partial}{\partial y^n}(g(s,y))_{kl}\Psi^n_{,m}\Psi^k_{,i}\Psi^l_{,j}-\frac{\partial}{\partial s}(g(s,y))_{kl}\varphi_{,m}\Psi^k_{,i}\Psi^l_{,j}-g_{kl}\Psi^k_{,i}\Psi^l_{,jm}\\
&=\gamma_{ij,m}-g_{kl,n} \Psi^n_{,m}\Psi^k_{,i}\Psi^l_{,j}-2\ubar\chi_{kl}\varphi_{,m}\Psi^k_{,i}\Psi^l_{,j}-g_{kl}\Psi^k_{,i}\Psi^l_{,jm}.
\end{align*}	
From a symmetric summation in $\{ijm\}$ we therefore observe:
\begin{align}	
g_{kl}\Psi^k_{,im}\Psi^l_{,j} &= \frac12\Big(\gamma_{ij,m}-\gamma_{im,j}+\gamma_{jm,i}+g_{kl,n}(\Psi^k_{,m}\Psi^l_{,i}\Psi^n_{,j}-\Psi^k_{,j}\Psi^l_{,m}\Psi^n_{,i}-\Psi^k_{,i}\Psi^l_{,j}\Psi^n_{,m})\notag\\
&\qquad\qquad+2\ubar\chi_{kl}(\varphi_{,j}\Psi^k_{,i}\Psi^l_{,m}-\varphi_{,m}\Psi^k_{,i}\Psi^l_{,j}-\varphi_{,i}\Psi^k_{,m}\Psi^l
_{,j})\Big)\notag,\\
\Psi^s_{,im}&=\frac12g^{sr}(d\Psi^{-1})^j_r \Big(\gamma_{ij,m}-\gamma_{im,j}+\gamma_{jm,i}+g_{kl,n}(\Psi^k_{,m}\Psi^l_{,i}\Psi^n_{,j}-\Psi^k_{,j}\Psi^l_{,m}\Psi^n_{,i}-\Psi^k_{,i}\Psi^l_{,j}\Psi^n_{,m})\notag\\
&\qquad\qquad+2\ubar\chi_{kl}(\varphi_{,j}\Psi^k_{,i}\Psi^l_{,m}-\varphi_{,m}\Psi^k_{,i}\Psi^l_{,j}-\varphi_{,i}\Psi^k_{,m}\Psi^l
_{,j})\Big)\notag\\
&=\frac12g^{sr}\Big((d\Psi^{-1})^j_r\big(\gamma_{ij,m}-\gamma_{im,j}+\gamma_{jm,i}\big)+\big(g_{kl,r}-g_{rk,l}-g_{lr,k}\big)\Psi^k_{,m}\Psi^l_{,i}\notag\\
&\qquad\qquad+2\big(\ubar\chi_{kl}\varphi_{,j}(d\Psi^{-1})^j_r\Psi^k_{,i}\Psi^l_{,m}
-\ubar\chi_{kr}\varphi_{,m}\Psi^k_{,i}-\ubar\chi_{kr}\varphi_{,i}\Psi^k_{,m}\big)\Big)\notag\\
&=\frac12\gamma^{nj} \big(\gamma_{ij,m}-\gamma_{im,j}+\gamma_{jm,i}\big)\Psi^s_{,n}+\frac12g^{sr} \big(g_{kl,r}-g_{rk,l}-g_{lr,k}\big)\Psi^k_{,m}\Psi^l_{,i}\notag\\
&\qquad\qquad+\ubar\chi_{kl}\varphi_{,j}\gamma^{nj}\Psi^s_{,n}\Psi^k_{,i}\Psi^l_{,m}-\ubar\chi_{kr}\varphi_{,m}g^{sr}\Psi^k_{,i}-\ubar\chi_{kr}g^{sr}\varphi_{,i}\Psi^k_{,m}\notag\\
&={^\gamma}\Gamma^n_{im}\Psi^s_{,n}-{^g}\Gamma^s_{kl}\Psi^k_{,m}\Psi^l_{,i} +\ubar\chi_{kl}\varphi_{,j}\gamma^{nj}\Psi^s_{,n}\Psi^k_{,i}\Psi^l_{,m}-\ubar\chi_{kr}\varphi_{,m}g^{sr}\Psi^k_{,i}-\ubar\chi_{kr}g^{sr}\varphi_{,i}\Psi^k_{,m}\label{e19}.
\end{align}
We conclude that the local regularity of $\Psi^i(x)$ is always a single derivative larger in comparison to $\varphi(x)$. Combining all these facts, derivatives of (\ref{e19}) gives the apriori estimate (ignoring contributions from the locally smooth metric $g$):
\begin{equation}
||\Psi||_{k+1,\beta}\leq C(||\gamma||_{k,\beta},||\varphi||_{k,\beta},||\Psi||_{1,\beta}).\label{e20}
\end{equation}
From the fact that $\vec{r}\in C^{2,\beta}(\mathbb{S}^2,\mathbb{R}^3)$, we have $\varphi\in C^{2,\beta}(\mathbb{S}^2)$, $\Psi\in C^{2,\beta}\mathcal{D}(\mathbb{S}^2)$. From (\ref{e19}), we therefore conclude that, in-fact, $\Psi\in C^{3,\beta}\mathcal{D}(\mathbb{S}^2)$. Therefore, in order to raise the regularity of the pair $(\Psi,\varphi)$, we're required to raise the regularity of $\varphi$ first. We do so using Lemma \ref{l5}:

\begin{equation}\label{e21}
\mathcal{K}\circ\Psi^{-1}=-\nabla_\omega\cdot(\tr\ubar\chi_\omega d\omega-\ubar\chi_\omega(\nabla_\omega\omega))-(d_s\tr\ubar\chi_s-\nabla_s\cdot\ubar\chi_s)(\nabla_s\omega)+\mathcal{K}_s(\omega).
\end{equation}
Since $(\Omega,g)$ is convex, we conclude that (\ref{e21}) is a quasi-linear elliptic equation for $\omega(y)$. From interior Schauder estimates, $$||\omega||_{2,\beta}\leq C(||\ubar\chi_\omega||_{1,0},||g_\omega||_{1,0})||\mathcal{K}\circ\Psi^{-1}-\mathcal{K}_s(\omega)||_{0,\beta}\leq C(||\gamma||_{2,\beta}, ||\omega||_{1,\beta},||\Psi^{-1}||_{0,\beta}).$$ 
For $n\geq 3$, we also observe $\mathcal{K}\circ\Psi^{-1}\in C^{1,\beta}(\mathbb{S}^2)$. From interior Schauder estimates we conclude $\omega\in C^{3,\beta}(\mathbb{S}^2)$, moreover,
\begin{align*}
||\omega||_{3,\beta}&\leq C(||\ubar\chi_\omega||_{2,0},||g_\omega||_{2,0})(||\mathcal{K}\circ \Psi^{-1}-\mathcal{K}_s(\omega)||_{1,\beta})\leq C(||\gamma||_{3,\beta},||\omega||_{2,\beta},||\Psi^{-1}||_{1,\beta})\\
&\leq C(||\gamma||_{3,\beta},||\omega||_{1,\beta},||\Psi^{-1}||_{1,\beta}).
\end{align*}
Since $\varphi = \omega\circ\Psi$, $(d\Psi^{-1})^i_{j} = \gamma^{ik}\Psi^n_{,k}g_{nj}$, we have 
\begin{align*}	
||\varphi||_{3,\beta}&\leq C(||\omega||_{3,\beta},||\Psi||_{3,\beta})\leq C(||\gamma||_{3,\beta},||\varphi||_{1,\beta},||\Psi||_{1,\beta}),\\
||\Psi||_{4,\beta}&\leq C(||\gamma||_{3,\beta},||\varphi||_{1,\beta},||\Psi||_{1,\beta}),
\end{align*}
giving,
$$||\vec{r}||_{3,\beta}\leq C(||\gamma||_{3,\beta},||\vec{r}||_{1,\beta}).$$
By induction, for $k\leq \mathfrak{n}$ we assume: 
\begin{align*}
||\varphi||_{k-1,\beta}&\leq C(||\gamma||_{k-1,\beta},||\varphi||_{1,\beta},||\Psi||_{1,\beta}),\\
||\Psi||_{k,\beta}&\leq C(||\gamma||_{k-1,\beta},||\varphi||_{1,\beta},||\Psi||_{1,\beta}).
\end{align*}
It follows that $\mathcal{K}\circ\Psi^{-1}\in C^{k-2,\beta}(\mathbb{S}^2)$, from interior Schauder estimates of (\ref{e21}):
\begin{align*}
||\omega||_{k,\beta}&\leq C(||\ubar\chi_\omega||_{k-1,0},||g_\omega||_{k-1,0})(||\mathcal{K}\circ\Psi^{-1}-\mathcal{K}_s(\omega)||_{k-1,\beta})\\
&\leq C(||\gamma||_{k,\beta},||\omega||_{k-1,\beta},||\Psi^{-1}||_{k-1,\beta})\\
&\leq C(||\gamma||_{k,\beta},||\omega||_{1,\beta},||\Psi||_{1,\beta},||\varphi||_{0,\beta})\\
&\leq C(||\gamma||_{k,\beta},||\varphi||_{1,\beta},||\Psi||_{1,\beta}).
\end{align*}
Again $\varphi  = \omega\circ \Psi$, followed by (\ref{e19}) implies 
\begin{align*}
||\varphi||_{k,\beta}&\leq C(||\omega||_{k,\beta},||\Psi||_{k,\beta})\leq C(||\gamma||_{k,\beta},||\varphi||_{1,\beta},||\Psi||_{1,\beta}),\\
||\Psi||_{k+1,\beta}&\leq C(||\gamma||_{k,\beta},||\varphi||_{1,\beta},||\Psi||_{1,\beta}),\\
||\vec{r}||_{k,\beta}&\leq C(||\gamma||_{k,\beta},||\vec{r}||_{1,\beta}).
\end{align*}
\end{proof}
\subsection{Openness II: Existence}
Given a continuous path of metrics $t\to\gamma_t$ on $\mathbb{S}^2$, $J$ an interval with $0,t\in J$, we know from Theorem \ref{t2}, given an isometric embedding of $(\mathbb{S}^2,\gamma_0)\hookrightarrow(\mathcal{A},\sigma)$, sufficiently small values of $t$ will ensure $(\mathbb{S}^2,\gamma_t)\hookrightarrow(\mathcal{A},\sigma)$ also, say for $t\in(-\epsilon,\epsilon)\subset J$. If the path is continuously Fr\'{e}chet differentiable, Proposition \ref{p5} also ensures the path of embeddings $t\to\vec{r}_t = \vec{r}_0+\vec{y}_t$ is continuously differentiable. If we fix some $t_0\in (-\epsilon,\epsilon)$, we can instead decompose the expression:
$$\sigma_{\mu\nu}(\vec{r}_{t_0}+\vec{z}_t)\frac{\partial (r_{t_0}^\mu+z^\mu_{t})}{\partial x^i}\frac{\partial (r_{t_0}^\nu+z^\nu_{t})}{\partial x^j} = (\gamma_t)_{ij}$$
for perturbations about the metric $\gamma_{t_0}$. Subsequently, the analogous expression for (\ref{e10}), with $\vec{z}_t=\vec{y}_{t}-\vec{y}_{t_0}=\bar{\tau}_t^i(\vec{r}_{t_0})_i+\bar{\phi}_t\frac{\vec{r}_{t_0}}{|\vec{r}_{t_0}|}=(\tau_t-\tau_{t_0})^i\vec{r}_i+(\phi_t-\phi_{t_0})\frac{\vec{r}_0}{|\vec{r}_0|}$, gives:
\begin{align*}
\nabla^{\gamma_{t_0}}_i(\bar{\tau}_t)_j+\nabla^{\gamma_{t_0}}_j(\bar{\tau}_t)_i +\bar{\phi}_t h(\vec{r}_{t_0})_{ij} &= (\gamma_{t})_{ij}-(\gamma_{t_0})_{ij}-\sigma_{\mu\nu}(\vec{r}_{t_0})({z}_t)^\mu_i (z_t)^\nu_j-z_t^\eta z_t^\xi (r_{t_0})^\mu_i (r_{t_0})^\nu_j F_{\mu\nu\eta\xi}(\vec{r}_{t_0},\vec{z}_t)\\
&\qquad-G_{\mu\nu\eta}(\vec{r}_{t_0},\vec{z}_t)z_t^\eta((r_{t_0})^\mu_i (z_t)^\nu_j+(r_{t_0})^\nu_j (z_t)^\mu_i+(z_t)^\mu_i (z_t)^\nu_j).
\end{align*}
Since $\vec{z}_t$ is Fr\'{e}chet differentiable, the same is true for its linearly independent constituents. Dividing through by $(t-t_0)$ and taking the limit as $t\to t_0$ we observe:
$$2\text{Sym}(\nabla^{\gamma_{t_0}}\dot{\bar{\tau}})+\dot{\bar\phi}h(\vec{r}_{t_0}) = \dot{\gamma}_{t_0}\iff L_{h(\vec{r}_{t_0})}(\dot{\bar\tau}) = \dot{\gamma}_{t_0}-\frac12\tr_{h(\vec{r}_{t_0})}(\dot{\gamma}_{t_0})h(\vec{r}_{t_0}),\,\,\,\dot{\bar{\phi}} = \frac{\tr_{\gamma_{t_0}}(\dot{\gamma}_{t_0})-2\nabla^{\gamma_{t_0}}\cdot\dot{\bar{\tau}}}{\tr_{\gamma_{t_0}}h(\vec{r}_{t_0})}.$$
So for the metric $\gamma_{t_0}$, we observe that the linearized embedding equation is satisfied by the derivative $(\dot{\bar{\tau}},\dot{\bar\phi})$ as expected. Beyond just $(\mathbb{S}^2,\gamma_0)$, we wish to exploit this fact in order to deduce Schauder estimates further along the path of metrics $t\to\gamma_t$. We notice there's no a priori reason to conclude $\dot{\bar\tau}\in\text{Ker}(L_{h(\vec{r}_{t_0})})^\perp$ with respect to the $L^2$ inner produce induced by $\gamma_{t_0}$, from which our Schauder estimates are derived in Theorem \ref{t1}. Fortunately, we can find a unique path of embeddings from the constraint $\dot{\bar{\tau}}\in\text{Ker}(L_{h(\vec{r}_{t_0})})^\perp$. We will prove the existence and uniqueness parts separately.
\begin{theorem}\label{t3}(Existence)
Suppose $(\mathcal{A},\sigma)$ is a convex Null Cone. For $\mathfrak{n}\geq 2$, consider also a continuous path of metrics $t\to\gamma_t\in C^{\mathfrak{n}+1,\beta}(\text{Sym}(T^\star\mathbb{S}^2\otimes T^\star\mathbb{S}^2))$, $t\in[0,b]$, with the following properties:
\begin{enumerate}
\item There exists an isometric embedding $r_0:(\mathbb{S}^2,\gamma_0)\hookrightarrow(\mathcal{A},\sigma)$,
\item $t\to\gamma_t\in C^{\mathfrak{n},\beta}(\text{Sym}(T^\star\mathbb{S}^2\otimes T^\star\mathbb{S}^2))$ is $\mathfrak{m}$-times continuously Fr\'{e}chet differentiable within $C^{\mathfrak{n},\beta}(\text{Sym}(T^\star\mathbb{S}^2\otimes T^\star\mathbb{S}^2))$.
\end{enumerate}
Then, there exists $0<\epsilon<b$, and a family of embeddings $r_t:(\mathbb{S}^2,\gamma_t)\hookrightarrow (\mathcal{A},\sigma)$, $t\in[0,\epsilon]$, with continuous associated path $t\to\vec{r}_t\in C^{\mathfrak{n}+1,\beta}(\mathbb{S}^2,\mathbb{R}^3)$ that is $\mathfrak{m}$-times continuously Fr\'{e}chet differentiable within $C^{\mathfrak{n},\beta}(\mathbb{S}^2,\mathbb{R}^3)$. Moreover, $\epsilon$ is independent of $\mathfrak{n},\mathfrak{m}$, and the first Fr\'{e}chet derivative satisfies: 
\begin{align*}
\dot{\vec{r}}_t &= d\vec{r}_t(\tau_t^\#)+\phi_t\frac{\vec{r}_t}{|\vec{r}_t|},\\
L_{h(\vec{r}_t)}(\tau_t) = \dot{\gamma}_t-&\frac12\tr_{h(\vec{r}_t)}(\dot{\gamma}_t)h(\vec{r}_t),\,\,\,\phi_t=\frac{\tr_{\gamma_t}(\dot{\gamma}_t)-2\nabla^{\gamma_t}\cdot\tau_t}{\tr_{\gamma_t}h(\vec{r}_t)},
\end{align*}
whereby $\tau_t\perp\text{Ker}(L_{h(\vec{r}_t)})$ with respect to the $L^2$ inner product induced by $\gamma_t$.
\end{theorem}
\begin{proof}
\underline{\bf Step 1: Finding $\vec{r}_t$}\\
We have the embedding $r_0:(\mathbb{S}^2,\gamma_0)\hookrightarrow(\mathcal{A},\sigma)$. For some $c_0>0$, such that $S_-<\min_{p\in\mathbb{S}^2}|\vec{r}_0|(p)-c_0$ we consider the the closed ball:
$$B_{c_0}^{\vec{r}_0}(C^{2,\beta}(\mathbb{S}^2,\mathbb{R}^3)):=\{\vec{r}\in C^{2,\beta}(\mathbb{S}^2,\mathbb{R}^3)\big|||\vec{r}-\vec{r}_0||_{2,\beta}\leq c_0\}.$$
\indent We now consider some arbitrary isometric embedding $r:(\mathbb{S}^2,\gamma)\hookrightarrow(\mathcal{A},\sigma)$, such that $\vec{r}\in B^{\vec{r}_0}_{c_0}(C^{2,\beta}(\mathbb{S}^2,\mathbb{R}^3))$. From Proposition \ref{p6}, we conclude $\vec{r}\in C^{3,\beta}(\mathbb{S}^2,\mathbb{R}^3)$, with $||\vec{r}||_{3,\beta}\leq C(c_0,\gamma)$, where $C(c_0,\gamma)$ indicates a constant only dependent upon $c_0,\gamma$. We will simply denote $\ubar\chi_{\mu\nu}:=\tilde B^\star(\ubar\chi)_{\mu\nu}$, it follows in a local coordinate neighborhood that $h_{ij}(\vec{r}) = 2\ubar\chi_{\mu\nu}(\vec{r})r^\mu_i r^\nu_j$. We also recall in dimension two:
$$(h^{-1})^{ij}(\vec{r}) = \frac{\frac12\tr_\gamma h\gamma^{ij} - \hat{h}^{ij}}{\frac14(\tr_\gamma h)^2 - \frac12|\hat{h}|^2}(\vec{r}),$$
giving $||h^{-1}(\vec{r})||_{2,\beta},||h(\vec{r})||_{2,\beta}\leq C(c_0,\gamma)$. It follows, since the path $t\to\gamma_t$ is continuous on a compact interval, that \textit{any} isometric embedding $r:(\mathbb{S}^2,\gamma_t)\hookrightarrow(\mathcal{A},\sigma)$, provided $\vec{r}\in B^{\vec{r}_0}_{c_0}(C^{2,\beta}(\mathbb{S}^2,\mathbb{R}^3))$, yields $||h^{-1}(\vec{r})||_{2,\beta},||h(\vec{r})||_{2,\beta}\leq C(c_0,\{\gamma_t\})$, where $C(c_0,\{\gamma_t\})$ indicates a constant dependent on $c_0$ and the path $\{\gamma_t\}$, but independent of $t$. We now notice from (\ref{e10}-\ref{e14}), by enlarging $C(c_0,\{\gamma_t\})$ if necessary, we obtain for -any- embedding $r:(\mathbb{S}^2,\gamma_t)\hookrightarrow(\mathcal{A},\sigma)$ such that $\vec{r}\in B_{c_0}^{\vec{r}_0}(C^{2,\beta}(\mathbb{S}^2,\mathbb{R}^3))$:
$$||\mathcal{C}_r(\vec{z})||_{2,\beta}\leq C(c_0,\{\gamma_t\})\big(||\tilde\gamma-\gamma||_{2,\beta}+||\vec{z}||^2_{2,\beta}(1+||\vec{z}||_{2,\beta}+||\vec{z}||_{2,\beta}^2)\big),$$
where $\mathcal{C}_r(\vec{z}) = \tau^i(\vec{z})\vec{r}_i+\phi(\vec{z})\frac{\vec{r}}{|\vec{r}|}$ is the mapping of Theorem \ref{t2} ``anchored" at $\vec{r}$. Also, from (\ref{e15},\ref{e16}), by enlarging $C(c_0,\{\gamma_t\})$ if necessary, we also have for $\vec{z}_1,\vec{z}_2\in C^{2,\beta}(\mathbb{S}^2,\mathbb{R}^3)$, such that $||\vec{z}_i||_{2,\beta}\leq \delta$ (say for example, any $\delta\leq c_0$):
$$||\mathcal{C}_r(\vec{z}_1)-\mathcal{C}_r(\vec{z}_2)||_{2,\beta}\leq C(c_0,\{\gamma_t\})\delta||\vec{z}_1-\vec{z}_2||_{2,\beta}.$$
\indent For an $0<\epsilon<b$ that we shall choose shortly, we now partition $[0,\epsilon]$ with $\{{^n}t_m\}_{m=0}^n$, whereby ${^n}t_m:=\epsilon\frac{m}{n}$. On the first interval $[0,{^n}t_1]$, we isometrically embed the path of metrics $\{\gamma_t\}_{t\in[0,{^n}t_1]}$ with a continuous family of embeddings resulting from the contraction mapping of Theorem \ref{t2}. As a result of our analysis in the previous paragraph, we have:
$$||{^n}\mathcal{C}_0(\vec{z})||_{2,\beta}\leq C(c_0,\{\gamma_t\}) \big(||\gamma_t-\gamma_0||_{2,\beta}+||\vec{z}||^2_{2,\beta}(1+||\vec{z}||_{2,\beta}+||\vec{z}||_{2,\beta}^2)\big),$$
whereby ${^n}\mathcal{C}_0 := \mathcal{C}_{r_{{^n}t_0}}$, the contraction mapping ``anchored" to $\vec{r}_{{^n}t_0}=\vec{r}_0$. We may now choose any $\delta>0$ such that $C(c_0,\{\gamma_t\})\delta(1+\delta+\delta^2)<1$. Since the path $t\to\gamma_t$ is continuously differentiable, $||\gamma_{t_1}-\gamma_{t_2}||_{2,\beta}\leq(\sup_t||\dot{\gamma}_t||_{2,\beta})|t_1-t_2|$. So for sufficiently large $n(\delta)\in\mathbb{N}$, we have, given any $||\vec{z}||_{2,\beta}\leq\delta$:
$$||{^n}\mathcal{C}_0(\vec{z})||_{2,\beta}\leq C(c_0,\{\gamma_t\})\sup_t||\dot{\gamma}_t||_{2,\beta}\frac{\epsilon}{n}+\ C(c_0,\{\gamma_t\})\delta^2(1+\delta+\delta^2) \leq\delta,$$
and:
$$||{^n}\mathcal{C}_0(\vec{z}_1)-{^n}\mathcal{C}_0(\vec{z}_2)||_{2,\beta}\leq C(c_0,\{\gamma_t\})\delta||\vec{z}_1-\vec{z}_2||_{2,\beta},$$
where $\delta C(c_0,\{\gamma_t\})<1$. So ${^n}\mathcal{C}_0(\vec{z})={^n}\tau(\vec{z})_0^i(\vec{r}_0)_i+{^n}\phi(\vec{z})_0\frac{\vec{r}_0}{|\vec{r}_0|}$ is a contraction map for any choice of $\gamma_t,\,t\in[0,{^n}t_1]$. Denoting the path of embeddings by $t\to{^n}\vec{r}_t$ for $t\in[0,{^n}t_1]$, we have ${^n}\vec{r}_t =\vec{r}_0+{^n}\vec{y}_t$ for the continuous path of ${^n}\mathcal{C}_0$ fixed points ${^n}\vec{y}_t:={^n}\tau_t^i(\vec{r}_0)_i+{^n}\phi_t\frac{\vec{r}_0}{|\vec{r}_0|}$. If we also denote the elliptic operator ``anchored" to the embedding ${^n}\vec{r}_{{^n}t_0}=\vec{r}_0$ by ${^n}L_0(\tau):=2\text{Sym}(\nabla\tau)-\tr_{h(\vec{r}_0)}(\nabla\tau)h(\vec{r}_0)$, it follows that ${^n}\tau_t\perp\text{Ker}({^n}L_0)$ for $t\in[0,{^n}t_1]$. Our choice of $\delta$ also gives 
\begin{align*}
||{^n}\vec{r}_{{^n}t_1}-\vec{r}_0||_{2,\beta}&=||{^n}\vec{y}_{{^n}t_1}||_{2,\beta}\leq \frac{C(c_0,\{\gamma_t\})}{1-\delta C(c_0,\{\gamma_t\})(1+\delta+\delta^2)}||\gamma_{{^n}t_1}-\gamma_0||_{2,\beta}\\
&\leq \frac{C(c_0,\{\gamma_t\})\sup_t||\dot{\gamma}_t||_{2,\beta}}{1-\delta C(c_0,\{\gamma_t\})(1+\delta+\delta^2)}\frac{\epsilon}{n}.
\end{align*}
So finally, if we choose $\epsilon$ such that
$$\epsilon \frac{C(c_0,\{\gamma_t\})\sup_t||\dot{\gamma}_t||_{2,\beta}}{1-\delta C(c_0,\{\gamma_t\})(1+\delta+\delta^2)} \leq c_0,$$
we conclude $||{^n}\vec{r}_{{^n}t_1}-\vec{r}_0||_{2,\beta}\leq \frac{c_0}{n}$. We may therefore re-apply the contraction mapping principle of Theorem \ref{t2}, with a contraction mapping ${^n}\mathcal{C}_1$, at ${^n}\vec{r}_{{^n}t_1}$. This allows us to continue to define the path $t\to {^n}\vec{r}_t$ on the interval $t\in[{^n}t_1,{^n}t_2]$, now ``anchored" to the embedding ${^n}\vec{r}_{{^n}t_1}$. Our choice of $\epsilon,\delta$ ensures, with identical estimates as above, that ${^n}\mathcal{C}_1$ is again a contraction map. By induction, on the interval $t\in[{^n}t_{m-1},{^n}t_m]$, we have: 
$${^n}\vec{r}_t ={^n}\vec{r}_{{^n}t_{m-1}}+{^n}\vec{y}_t= {^n}\vec{r}_{{^n}t_{m-1}}+{^n}\tau_t^i\big({^n}\vec{r}_{{^n}t_{m-1}}\big)_i+{^n}\phi_t\frac{{^n}\vec{r}_{{^n}t_{m-1}}}{| {^n}\vec{r}_{{^n}t_{m-1}} |},$$
for ${^n}\tau_t\perp\text{Ker}({^n}L_{m-1})$, $t\in[{^n}t_{m-1},{^n}t_m]$. Again, our choice of $\epsilon,\delta$ gives $||{^n}\vec{r}_{{^n}t_m}-\vec{r}_0||\leq\sum_{i=1}^m||{^n}\vec{y}_{{^n}t_i}||\leq c_0\frac{m}{n}\leq c_0$, and we may continue to isometrically embed all metrics within $\{\gamma_t\},\,t\in[0,\epsilon]$. As a result of Proposition \ref{p5}, we conclude that $t\to{^n}\vec{r}_t$ is a piecewise differentiable path with Fr\'{e}chet derivative, ${^n}\dot{\vec{r}}_t$, that we may take to be right semi-continuous on $t\in[0,\epsilon]$.\\ 
\indent We denote by $C([0,\epsilon],B^{\vec{r}_0}_{c_0}(C^{2,\beta}(\mathbb{S}^2,\mathbb{R}^3)))$, the space of continuous functions $t\to\vec{r}'_t\in C^{2,\beta}(\mathbb{S}^2,\mathbb{R}^3)$. The space $C([0,\epsilon],B^{\vec{r}_0}_{c_0}(C^{2,\beta}(\mathbb{S}^2,\mathbb{R}^3)))$ is itself a complete metric space when equipped with the distance function $d(\vec{r}', \vec{r}''):={\sup_t||\vec{r}'_t-\vec{r}''_t||_{2,\beta}}$. For increasing integer values $n\geq n(\delta)$, we observe our sequence $\{{^n}\vec{r}_t\}_n\subset C([0,\epsilon],B^{\vec{r}_0}_{c_0}(C^{2,\beta}(\mathbb{S}^2,\mathbb{R}^3)))$. Moreover, for any $t_1\leq t_2$, we find $m_1,m_2\in\mathbb{N}$, $m_i\leq n$, such that $t_1 \leq \epsilon\frac{m_1}{n}\leq \epsilon\frac{m_2}{n}\leq t_2$, and $|t_i-\epsilon\frac{m_i}{n}|\leq \frac{\epsilon}{n}$. Therefore, we have:
\begin{align*}
||{^n}\vec{r}_{t_1}-{^n}\vec{r}_{t_2}||_{2,\beta}&\leq||{^n}\vec{r}_{t_1}-{^n}\vec{r}_{{^n}t_{m_1}}||_{2,\beta}+||{^n}\vec{r}_{{^n}t_{m_2}}-{^n}\vec{r}_{t_2}||_{2,\beta}+\sum_{i=m_1+1}^{m_2}||{^n}\vec{y}_{{^n}t_{i}}||_{2,\beta}\\
&=||{^n}\vec{y}_{{^n}t_{m_1}}-{^n}\vec{y}_{t_1}||_{2,\beta} +||{^n}\vec{y}_{{^n}t_{m_2}}-{^n}\vec{y}_{t_2}||_{2,\beta}+\sum_{i=m_1+1}^{m_2}||{^n}\vec{y}_{{^n}t_i}||_{2,\beta}\\
&\leq \frac{c_0}{\epsilon}\Big((\epsilon\frac{m_1}{n}-t_1)+(t_2-\epsilon\frac{m_2}{n})+\epsilon\frac{m_2-m_1}{n}\Big)=\frac{c_0}{\epsilon}(t_2-t_1).
\end{align*}
So the sequence of embeddings, $\{{^n}\vec{r}_t\}_n$, satisfy the boundedness and equicontinuity hypotheses of the Arz\`{e}la-Ascoli Theorem. It follows that a subsequence converges in $C([0,\epsilon],B^{\vec{r}_0}_{c_0}(C^{2,\beta}(\mathbb{S}^2,\mathbb{R}^3))) $ to a Lipshitz function, $t \to \vec{r}_t\in B_{c_0}^{\vec{r}_0}(C^{2,\beta}(\mathbb{S}^2,\mathbb{R}^3))$, representing a family of isometric embeddings, $r_t:(\mathbb{S}^2,\gamma_t)\hookrightarrow(\mathcal{A},\sigma)$. \\\\
\underline{\bf Step 2: $t\to\vec{r}_t\in C^{\mathfrak{n}+1,\beta}(\mathbb{S}^2,\mathbb{R}^3)$ is continuous}\\
From Proposition \ref{p6} we know $\vec{r}_t\in C^{\mathfrak{n}+1,\beta}(\mathbb{S}^2,\mathbb{R}^3)$ for each $t\in[0,\epsilon]$. We will argue instead that the pair $(\varphi_t,\Psi_t)$ from Proposition \ref{p6} are continuous in $t$. Since $t\to \vec{r}_t\in C^{2,\beta}(\mathbb{S}^2,\mathbb{R}^3)$ is continuous we conclude that $t\to\Psi_t\in C^{2,\beta}\mathcal{D}(\mathbb{S}^2)$, and $t\to\varphi_t\in C^{2,\beta}(\mathbb{S}^2)$ are as well. Taking a derivative of (\ref{e19}), we therefore conclude $t\to\Psi_t\in C^{3,\beta}\mathcal{D}(\mathbb{S}^2)$ is also continuous. For convenience, if we denote (\ref{e20}) by $E_\omega(\omega) = \mathcal{K}\circ\Psi^{-1} - \mathcal{K}_s(\omega)$, from which we deduce that $E_\omega$ is a quasi-linear elliptic operator dependent upon $\omega$. For $\omega_t:=\varphi_t\circ\Psi_t^{-1}$, we have for some $t_0\in[0,\epsilon]$:
$$ E_{\omega_{t_0}}(\omega_{t_0} - \omega_{t}) = \mathcal{K}_{t_0}\circ\Psi^{-1}_{t_0}-\mathcal{K}_{t}\circ\Psi^{-1}_{t}-(\mathcal{K}_s(\omega_{t_0})-\mathcal{K}_s(\omega_t))- (E_{\omega_{t_0}} - E_{\omega_{t}})(\omega_{t}).$$
Recalling our analysis in Proposition \ref{p6}, since $\mathfrak{n}\geq 2$, similar interior Schauder estimates for the above elliptic equation gives:
\begin{align*}	
||\omega_{t_0}-\omega_t||_{3,\beta}&\leq C(||\gamma_{t_0}||_{3,\beta},||\omega_{t_0}||_{1,\beta},||\Psi^{-1}_{t_0}||_{1,\beta})\big(||\mathcal{K}_{t_0}\circ\Psi^{-1}_{t_0}-\mathcal{K}_t\circ\Psi^{-1}_t||_{1,\beta}\\
&\qquad +||\mathcal{K}_s(\omega_{t_0})-\mathcal{K}_s(\omega_t)||_{1,\beta} +||(E_{\omega_{t_0}}-E_{\omega_t})(\omega_t)||_{1,\beta}\big).
\end{align*}
From Proposition \ref{p6}, we know that $\sup_t||\omega_t||_{\mathfrak{n}+1,\beta}$ is bounded. We also observe that 
$$||(E_{\omega_{t_0}}-E_{\omega_t})(\omega_t)||_{1,\beta}\leq C(||g||^\text{loc}_{4,0})||\omega_{t_0}-\omega_t||_{2,\beta}||\omega_t||_{3,\beta}.$$
Since we're assuming $\mathfrak{n}\geq 2$, we therefore conclude that $\lim_{t\to t_0}||\omega_{t_0}-\omega_t||_{3,\beta} = 0$. Thus, $t\to \varphi_t = \omega_t\circ\Psi_t\in C^{3,\beta}(\mathbb{S}^2)$ is continuous. For $\mathfrak{n}\geq 3$, by induction, we assume $t\to\Psi_t\in C^{\mathfrak{n},\beta}\mathcal{D}(\mathbb{S}^2)$, and $t\to\varphi\in C^{\mathfrak{n},\beta}(\mathbb{S}^2)$ are continuous. Derivatives of (\ref{e19}) again ensures $t\to\Psi_t\in C^{\mathfrak{n}+1,\beta}\mathcal{D}(\mathbb{S}^2)$ is continuous. Similarly as above,
\begin{align*}	
||\omega_{t_0}-\omega_t||_{\mathfrak{n}+1,\beta}&\leq C(||\gamma_{t_0}||_{\mathfrak{n}+1,\beta},||\omega_{t_0}||_{1,\beta},||\Psi^{-1}_{t_0}||_{1,\beta})\big(||\mathcal{K}_{t_0}\circ\Psi^{-1}_{t_0}-\mathcal{K}_t\circ\Psi^{-1}_t||_{\mathfrak{n}-1,\beta}\\
&\qquad +||\mathcal{K}_s(\omega_{t_0})-\mathcal{K}_s(\omega_t)||_{\mathfrak{n}-1,\beta} +||(E_{\omega_{t_0}}-E_{\omega_t})(\omega_t)||_{\mathfrak{n}-1,\beta}\big),\\
||(E_{\omega_{t_0}}-E_{\omega_t})(\omega_t)||_{\mathfrak{n}-1,\beta}&\leq C(||g||^\text{loc}_{\mathfrak{n}+2,0})||\omega_{t_0}-\omega_t||_{\mathfrak{n},\beta}||\omega_t||_{\mathfrak{n}-1,\beta},
\end{align*}
and we therefore conclude that $\lim_{t\to t_0}||\omega_{t_0}-\omega_t||_{\mathfrak{n}+1,\beta} = 0$. Consequently, $t\to\varphi_t\in C^{\mathfrak{n}+1,\beta}(\mathbb{S}^2)$ is continuous, and therefore $t\to\vec{r}_t\in C^{\mathfrak{n}+1,\beta}(\mathbb{S}^2,\mathbb{R}^3)$ is as-well.\\\\
\underline{\bf Step 3: $t\to\dot{\vec{r}}_t\in C^{\mathfrak{n},\beta}(\mathbb{S}^2,\mathbb{R}^3)$ exists and is continuous}\\
For each t, we now solve the equation:
$$L_{h(\vec{r}_t)}(\tau_t) = \dot{\gamma}_t-\tr_{h(\vec{r}_t)}\dot{\gamma}_t,$$
for unique $\tau_t\perp\text{Ker}(L_{h(\vec{r}_t)})$. Since the elliptic operator $L_{h(\vec{r}_t)}$ varies continuously with $t$, we have:
\begin{lemma}\label{l7}
We obtain a continuous path $t\to\tau_t\in C^{\mathfrak{n}+1,\beta}(T^\star\mathbb{S}^2)\cap\text{Ker}(L_{h(\vec{r}_t)})^\perp$ from solving the elliptic equation:
$$L_{h(\vec{r}_t)}(\tau_t)= \dot{\gamma}_t-(\tr_{h(\vec{r}_t)}\dot{\gamma}_t)h(\vec{r}_t),\,\,\,\text{for each t}.$$
\end{lemma}

\begin{proof}
We will follow a very similar argument as in Theorem \ref{t0}. First, we recall in a local coordinate neighborhood that $h(\vec{r}_t)_{ij} = 2\ubar\chi_{\mu\nu}(\vec{r}_t)(r^\mu_t)_i(r^\nu_t)_j$, such that $h(\vec{r}_t)\in C^{\mathfrak{n},\beta}(\text{Sym}(T^\star\mathbb{S}^2\otimes T^\star\mathbb{S}^2)).$
Solving instead the $2^\text{nd}$-order elliptic equation:
$$\mathcal{L}_{h(\vec{r}_t)}(\tau_t) = \nabla\cdot(\dot{\gamma}_t-(\tr_{h(\vec{r}_t)}\dot{\gamma}_t)h(\vec{r}_t)) - d\big(\tr_{\gamma_t}\dot{\gamma}_t-(\tr_{h(\vec{r}_t)}\dot\gamma_t)(\tr_{\gamma_t}h(\vec{r}_t))\big),$$
we observe a solution $\tau_t\in C^{\mathfrak{n}+1,\beta}(T^\star\mathbb{S}^2)$ for each $t\in[0,\epsilon]$ (Theorem 5.20, \cite{giamart}). Proposition \ref{p4} then ensures also $L_{h(\vec{r}_t)}(\tau_t)= \dot{\gamma}_t-(\tr_{h(\vec{r}_t)}\dot{\gamma}_t)h(\vec{r}_t)$.\\
\indent We start by fixing some $t_0\in[0,\epsilon]$, and wish to show $\lim_{t\to t_0}||\tau_{t_0}-\tau_t||_{\mathfrak{n}+1,\beta} = 0$. In order to do so, we denote by $\pi_t$ the projection map onto $\text{Ker}(L_{h(\vec{r}_t)})$. It follows:
\begin{align*}
L_{h(\vec{r}_{t_0})}\Big(\tau_{t_0}-(\tau_t-\pi_{t_0}(\tau_t))\Big)&=L_{h(\vec{r}_{t_0})}(\tau_{t_0})-L_{h(\vec{r}_{t})}(\tau_t)+\Big(L_{h(\vec{r}_{t})}-L_{h(\vec{r}_{t_0})}\Big)(\tau_{t})\\
&=\dot{\gamma}_{t_0}-\dot{\gamma}_t+\Big(\tr_{h(\vec{r}_t)}-\tr_{h(\vec{r}_{t_0})}\Big)(\dot{\gamma}_{t_0})h(\vec{r}_{t_0})+\tr_{h(\vec{r}_t)}\Big(\dot{\gamma}_t-\dot{\gamma}_{t_0}\Big)h(\vec{r}_{t_0})\\
&\qquad+\tr_{h(\vec{r}_t)}(\dot{\gamma}_t)\Big(h(\vec{r}_{t_0})-h(\vec{r}_t)\Big)+\Big(L_{h(\vec{r}_t)}-L_{h(\vec{r}_{t_0})}\Big)(\tau_t),
\end{align*}
from continuity of $\{\dot{\gamma}_t,\vec{r}_t\}$, the $C^{\mathfrak{n},\beta}$ boundedness of $\dot{\gamma}_t$, and $C^{\mathfrak{n}+1,\beta}$ boundedness of $\{\tau_t,\vec{r}_t\}$ for all values of $t$, we conclude from Schauder estimates that $||\tau_{t_0}-\tau_t+\pi_{t_0}(\tau_t)||_{\mathfrak{n}+1,\beta}\to 0$ as $t\to t_0$. If we now choose a basis $\{\eta_i\}\subset \text{Ker}(L_{h(\vec{r}_{t_0})})\subset C^{\mathfrak{n}+1,\beta}(T^\star\mathbb{S}^2)$ (by elliptic regularity), for $1\leq i\leq 6$, and with unit $L^2$ norm, we may write $\pi_{t_0}(\tau_t) = \sum_i\langle\eta_i,\tau_t\rangle_{L^2}\eta_i=\sum_i\langle\eta_i-\pi_{t}(\eta_i),\tau_t\rangle_{L^2}\eta_i$, for $\langle\cdot,\cdot\rangle_{L^2}$ the $L^2$ inner product induced by $\gamma_{t_0}$. We have,
$$L_{h(\vec{r}_t)}(\eta_i-\pi_t(\eta_i))=L_{h(\vec{r}_t)}(\eta_i) = \Big(L_{h(\vec{r}_t)}-L_{h(\vec{r}_{t_0})}\Big)(\eta_i)$$
giving us, via Schauder estimates, that $||\eta_i-\pi_t(\eta_i)||_{\mathfrak{n}+1,\beta}\to 0$, as $t\to t_0$. We therefore see:
\begin{align*}
||\pi_{t_0}(\tau_t)||_{\mathfrak{n}+1,\beta}&\leq\sum_i\Big|\langle\eta_i-\pi_t(\eta_i),\tau_t\rangle_{L^2}\Big|||\eta_i||_{\mathfrak{n}+1,\beta}\\
&\leq C\Big(\sum_i||\eta_i-\pi_t(\eta_i)||_{\mathfrak{n}+1,\beta}||\eta_i||_{\mathfrak{n}+1,\beta}\Big)||\tau_t||_{\mathfrak{n}+1,\beta},
\end{align*}
having used the Cauchy-Schwarz inequality, and the fact that $\mathbb{S}^2$ is compact to obtain the second inequality. We therefore also conclude that $||\pi_{t_0}(\tau_t)||_{\mathfrak{n}+1,\beta}\to0$, as $t\to t_0$. It follows:
$$||\tau_{t_0}-\tau_t||_{\mathfrak{n}+1,\beta}\leq||\tau_{t_0}-\tau_t+\pi_{t_0}(\tau_t)||_{\mathfrak{n}+1,\beta}+||\pi_{t_0}(\tau_t)||_{\mathfrak{n}+1,\beta}\to 0,\,\,\text{as}\,\,\,t\to t_0.$$
\end{proof}
From Lemma \ref{l7}, we observe the continuous path $t\to\phi_t\in C^{\mathfrak{n},\beta}(\mathbb{S}^2)$ given by:
$$\phi_t: = \frac{ \tr_{\gamma_t}(\dot{\gamma}_t)-2\nabla^{\gamma_t}\cdot\tau_t}{\tr_{\gamma_t}h(\vec{r}_t)}.$$
We may now construct a path $t\to\vec{r}_t^\star\in C^{\mathfrak{n},\beta}(\mathbb{S}^2,\mathbb{R}^3)$ given by:
 $$\vec{r}_t^\star:=\vec{r}_0+\int_0^t\tau_u^i(\vec{r}_u)_i+\frac{\phi_u}{|\vec{r}_u|}\vec{r}_u du.$$
It follows that $\vec{r}_t^\star$ admits a Fr\'{e}chet derivative in the $t$-variable with respect to the norm on $C^{\mathfrak{n},\beta}(\mathbb{S}^2,\mathbb{R}^3)$, and the derivative $t\to\dot{\vec{r}}_t^\star = \tau^i_t(\vec{r}_t)_i+ \frac{\phi_t}{|\vec{r}_t|}\vec{r}_t\in C^{\mathfrak{n},\beta}(\mathbb{S}^2,\mathbb{R}^3)$ is continuous. We need to show $\vec{r}_t = \vec{r}_t^\star$.\\
\indent By a re-indexing, we have that our sequence $\{{^n}\vec{r}_t\}$ converges to $\vec{r}_t$. Each $t\to{^n}\vec{r}_t$ is piecewise differentiable with right semi-continuous Fr\'{e}chet derivative with respect to the norm on $C^{2,\beta}(\mathbb{S}^2,\mathbb{R}^3)$, given by: $$t\to{^n}\dot{\vec{r}}_t={^n}\dot{\tau}_t^i({^n}\vec{r}_{{^n}t_{m-1}})_i+{^n}\dot{\phi}_t\frac{{^n}\vec{r}_{{^n}t_{m-1}}}{| {^n}\vec{r}_{{^n}t_{m-1}} |}$$
for ${^n}\dot{\tau}_t\perp\text{Ker}(L_{h({^n}\vec{r}_{{^n}t_{m-1}})})$, on the interval $t\in[{^n}t_{m-1},{^n}t_m)$. So for fixed $t$, as $n\to\infty$ we obtain $m(n)\to\infty$ such that ${^n}t_{m(n)-1}\to t$. We observe:
\begin{align*}
L_{h(\vec{r}_t)}(\tau_t-{^n}\dot{\tau}_t)&=L_{h(\vec{r}_t)}(\tau_t)-L_{h({^n}\vec{r}_{{^n}t_{m-1}})}({^n}\dot{\tau}_t)+\Big(L_{h({^n}\vec{r}_{{^n}t_{m-1}})}-L_{h(\vec{r}_t)}\Big)({^n}\dot{\tau}_t)\\
&= \dot{\gamma}_t-(\tr_{h(\vec{r}_t)}\dot{\gamma}_t)h(\vec{r}_t)-L_{h({^n}\vec{r}_{{^n}t_{m-1}})}({^n}\dot{\tau}_t) +\Big(L_{h({^n}\vec{r}_{{^n}t_{m-1})}}-L_{h(\vec{r}_t)}\Big)({^n}\dot{\tau}_t).
\end{align*}
From continuity of individual terms in the operator $L_{h({^n}\vec{r}_{{^n}t_{m-1}})}$, and boundedness of $||{^n}\dot{\tau}_t||_{2,\beta}$, we have
$$||\Big(L_{h({^n}\vec{r}_{{^n}t_{m-1}})}-L_{h(\vec{r}_t)}\Big)({^n}\dot{\tau}_t)||_{1,\beta}\to 0,\,\,\,\text{as}\,\,\,n\to\infty.$$ 
Taking a derivative of (\ref{e10}), we also observe an expression of the form:
\begin{align*}
\nabla^{\gamma_{{^n}t_{m-1}}}_i({^n}\dot{\tau}_t)_j +\nabla^{\gamma_{{^n}t_{m-1}}}_j({^n}\dot{\tau}_t)_i+&{^n}\dot{\phi}_t h({^n}\vec{r}_{{^n}t_{m-1}})\\
&= \dot{\gamma}_t+({^n}\dot{\vec{y}}_{t}^\mu)({^n}{\vec{y}}_{t}^\nu) q_1({^n}\vec{y}_t,\nabla{^n}\vec{y}_t)_{\mu\nu ij}\\
&+ ({^n}{\vec{y}}_t)^\nu\big(({^n}\dot{\vec{y}}_t)_i^\mu q_2({^n}\vec{y}_t,\nabla{^n}\vec{y}_t)_{\mu\nu j}+ ({^n}\dot{\vec{y}}_t)_j^\mu q_2({^n}\vec{y}_t,\nabla{^n}\vec{y}_t)_{\mu\nu i}\big)\\
&+\big(({^n}\dot{\vec{y}}_t)_j^\mu ({^n}{\vec{y}}_t)_i^\nu+ ({^n}{\vec{y}}_t)_j^\mu ({^n}\dot{\vec{y}}_t)_i^\nu\big) q_3({^n}\vec{y}_t,\nabla{^n}\vec{y}_t)_{\mu\nu},
\end{align*}
it therefore follows that
\begin{align*}
|| \dot{\gamma}_t&-(\tr_{h(\vec{r}_t)}\dot{\gamma}_t)h(\vec{r}_t)-L_{h({^n}\vec{r}_{{^n}t_{m-1}})}({^n}\dot{\tau}_t)||_{1,\beta}\\
&\leq C(c_0,\{\gamma_t\})||{^n}\dot{\vec{y}}_t||_{2,\beta}|| {^n}{\vec{y}}_t ||_{2,\beta}\leq \tilde{C}(c_0,\{\gamma_t\})\frac{c_0}{b}(t-{^n}t_{m-1})\to 0,\,\,\,\text{as}\,\,\,n\to\infty.
\end{align*}
From Schauder estimates we conclude $||\tau_t-{^n}\dot{\tau}_t+\pi_t({^n}\dot{\tau}_t)||_{2,\beta}\to0$, as $n\to\infty$. Arguing as in Lemma \ref{l7}, we also see $||\pi_t({^n}\dot{\tau}_t)||_{2,\beta}\to0$, thus $||\tau_t-{^n}\dot{\tau}_t||_{2,\beta}\to 0$, as $n\to\infty$. We consequently observe 
$$||\phi_t-{^n}\dot{\phi}_t||_{1,\beta} = \Big|\Big|\frac{\tr_{\gamma_t}(\dot{\gamma}_t)-2\nabla^{\gamma_t}\cdot\tau_t}{\tr_{\gamma_t}h(\vec{r}_t)}-\frac{\tr_{\gamma_t}(\dot{\gamma}_t)-2\nabla^{\gamma_t}\cdot({^n}\dot{\tau}_t)}{\tr_{\gamma_t}h({^n}\vec{r}_t)}\Big|\Big|_{1,\beta}\to 0,\,\,\,\text{as}\,\,\, n\to\infty.$$  
Combining all these facts, we finally conclude that:
$$||\dot{\vec{r}}_t^n-\dot{\vec{r}}_t^\star||_{1,\beta}\to 0,\,\,\,\text{as}\,\,n\to\infty,\,\,\,\text{for each}\,\,t\in[0,\epsilon].$$
Since $\vec{r}^n_t,\vec{r}_t\in B_{c_0}^{\vec{r}_0}(C^{2,\beta}(\mathbb{S}^2,\mathbb{R}^3))$, and Schauder estimates enforce $\sup_t||\tau_t||_{2,\beta}$, $\sup_t||{^n}\tau_t||_{2,\beta}$ are bounded, we also observe that $\{\sup_t||\dot{\vec{r}}^n_t-\dot{\vec{r}}_t^\star||_{1,\beta}\}_n$ is a bounded sequence. Therefore, by the Dominated Convergence Theorem, 
$$||\vec{r}_t-\vec{r}_t^\star||_{1,\beta} = \lim_{n\to\infty}||{^n}\vec{r}_t-\vec{r}_t^\star||_{1,\beta}\leq\int_0^t\lim_{n\to\infty}||\dot{\vec{r}}_u^n-\dot{\vec{r}}_u^\star||_{1,\beta}du=0,\,\,\,\text{for each}\,\,t\in[0,\epsilon].$$
\underline{\bf Step 4: $t\to\vec{r}^{(m)}_t\in C^{\mathfrak{n},\beta}(\mathbb{S}^2,\mathbb{R}^3)$ exists and is continuous for $m\leq\mathfrak{m}$}\\
This step is largely analogous to the previous one. We proceed by induction assuming the existence of continuous paths 
$$t\to\vec{r}_t^{(m-i)}\in C^{\mathfrak{n},\beta}(\mathbb{S}^2,\mathbb{R}^3),\,\,\,1\leq i\leq m.$$ 
We now recall equation (\ref{e17}), and within (\ref{e17}) we substitute: $n=m$, $\vec{r}\to\vec{r}_t$, $\vec{y}_t^{(m-i)}\to\vec{r}_t^{(m-i)}$, $\vec{y}_t = 0$. As a result, we observe an equation:
$$\nabla_i^{\gamma_t}(\tau_t^{(m)})_j+\nabla_j^{\gamma_t}(\tau_t^{(m)})_i+\phi_t^{(m)}h(\vec{r}_t)_{ij} = (\gamma_t^{(m)})_{ij}+P(\vec{r}_t^{(m-1)},\vec{r}_t^{(m-2)},\dots,\vec{r}_t^{(1)},\vec{r}_t)_{ij},$$
for the unknown pair $(\tau_t^{(m)},\phi_t^{(m)})$. Consequently, for each $t$, we deduce a well defined elliptic equation:
$$L_{h(\vec{r}_t)}(\tau_t^{(m)}) = \gamma_t^{(m)}-\tr_{h(\vec{r}_t)}\gamma_t^{(m)}+P(\vec{r}_t^{(m-1)},\dots,\vec{r}_t)-\tr_{h(\vec{r}_t)}P (\vec{r}_t^{(m-1)},\dots,\vec{r}_t),$$
which yields a unique solution $\tau_t^{(m)}\perp\text{Ker}(L_{h(\vec{r}_t)})\cap C^{\mathfrak{n}+1,\beta}(T^\star\mathbb{S}^2)$. We leave the reader the completely analogous argument to Lemma \ref{l7} that identifies the continuous path $t\to\tau_t^{(m)}\in C^{\mathfrak{n}+1,\beta}(T^\star\mathbb{S}^2)$. Consequently, we have the continuous path $t\to\phi_t^{(m)}\in C^{\mathfrak{n},\beta}(\mathbb{S}^2)$, given by:
$$\phi_t^{(m)}:= \frac{ \tr_{\gamma_t}(\gamma_t^{(m)})-2\nabla^{\gamma_t}\cdot\tau_t^{(m)}}{\tr_{\gamma_t}h(\vec{r}_t)}.$$
Again, we construct a continuous path $t\to(\vec{r}_t^{(m-1)})^\star\in C^{\mathfrak{n},\beta}(\mathbb{S}^2,\mathbb{R}^3)$ given by:
$$(\vec{r}_t^{(m-1)})^\star:=\vec{r}_0^{(m-1)}+\int_0^t(\tau_u^{(m)})^i(\vec{r}_u)_i+\frac{\phi_u^{(m)}}{|\vec{r}_u|}\vec{r}_u du,\,\,\,\text{for each}\,\,t\in[0,\epsilon],$$
from which we again observe a continuous Fr\'{e}chet derivative within $C^{\mathfrak{n},\beta}(\mathbb{S}^2,\mathbb{R}^3)$. By Proposition \ref{p5}, each term in our sequence $\{{^n}\vec{r}_t\}$ is piecewise differentiable with respect to the norm on $C^{2,\beta}(\mathbb{S}^2,\mathbb{R}^3)$, with right semi-continuous Fr\'{e}chet derivatives of order up to $\mathfrak{m}$. For $m\leq \mathfrak{m}$, $t\in[{^n}t_{k-1},{^n}t_{k})$, we denote these:
$$t\to {^n}\vec{r}_t^{(m)} = ({^n}\tau_t^{(m)})^i({^n}\vec{r}_{{^n}t_{k-1}})_i+{^n}\phi_t^{(m)}\frac{{^n}\vec{r}_{{^n}t_{k-1}}}{|{^n}\vec{r}_{{^n}t_{k-1}}|},$$
whereby ${^n}\tau_t^{(m)}\perp\text{Ker}(L_{h({^n}\vec{r}_{{^n}t_{k-1}})})$. From our use of (\ref{e17}) in the construction of the path $t\to \tau_t^{(m)}$, a completely analogous argument as in the first derivative case now ensues to show $||\tau_t^{(m)}-{^n}\tau_t^{(m)}||_{2,\beta}\to 0$, and also $||\phi_t^{(m)}-{^n}\phi_t^{(m)}||_{1,\beta}\to 0$ as $n\to\infty$. We leave for the reader again the completely analogous observation that this enforces that $\vec{r}_t^{(m-1)} = (\vec{r}_t^{(m-1)})^\star$.
\end{proof}
\begin{corollary}\label{c4}
Suppose $(\mathcal{A},\sigma)$ is a convex Null Cone, and consider a smooth path of Riemannian metrics $t\to\gamma_t$, $t\in[0,b]$ such that $(\mathbb{S}^2,\gamma_0)\hookrightarrow(\mathcal{A},\sigma)$. Then, there exists $0<\epsilon<b$, and a family of smooth isometric embeddings $r_t:(\mathbb{S}^2,\gamma_t)\hookrightarrow(\mathcal{A},\sigma)$, $t\in[0,\epsilon]$. We also find smooth sections $\tau:[0,\epsilon]\times \mathbb{S}^2\to T^\star\mathbb{S}^2$, $\phi:[0,\epsilon]\times\mathbb{S}^2\to\mathbb{R}$, such that the associated mapping $\vec{r}:[0,b]\times\mathbb{S}^2\to \mathbb{R}^3$ satisfies:
	\begin{align*}
d\vec{r}(\partial_t) &= d\vec{r}(\tau^\#)+\phi\frac{\vec{r}}{|\vec{r}|},\\
L_{h(\vec{r}_t)}(\tau_t) = \dot{\gamma}_t-\frac12\tr_{h(\vec{r}_t)}&(\dot{\gamma}_t)h(\vec{r}_t),\,\,\,\phi_t=\frac{\tr_{\gamma_t}(\dot{\gamma}_t)-2\nabla^{\gamma_t}\cdot\tau_t}{\tr_{\gamma_t}h(\vec{r}_t)},
\end{align*}
whereby $\tau_t\perp\text{Ker}(L_{h(\vec{r}_t)})$ for each $t\in[0,\epsilon]$, with respect to the $L^2$ inner product induced by $\gamma_t$.
\end{corollary}
\subsection{Openness II: Uniqueness}
As a consequence of Theorem \ref{t3}, we observe a solution of the system of first order equations:
\begin{align} 
L_{h(\vec{r}_t)}(\tau_t)&=\dot{\gamma}_t-(\tr_{h(\vec{r}_t)}\dot{\gamma}_t)h(\vec{r}_t),\,\,\tau_t\perp\text{Ker}(L_{h(\vec{r}_t)})\label{e22}\\
\frac{d}{dt}(\vec{r}_t)&=\tau_t^i(\vec{r}_t)_i+\phi_t\frac{\vec{r}_t}{|\vec{r}_t|},\,\,\vec{r}_t|_{t=0}=\vec{r}_0,\label{e23} 
\end{align}
on an interval $t\in[0,\epsilon]$. We can provide an explicit expression of the path $t\to\vec{r}_t$, in order to do so we will need a result in the theory of ODE flows on compact manifold:
\begin{proposition}\label{p7}[\cite{ebin1970groups}, Theorems 3.4, 3.5]
Consider a smooth compact manifold $M$, and for $\mathfrak{n}\geq 2$, $0<\beta<1$, a continuous path $t\to\tau_t\in C^{\mathfrak{n},\beta}(T^\star M)$, with $t,0\in I$. Then, the associated flow $\psi:I\times M\to M$, with $\psi(0,\cdot) = id$, is a $C^1$ curve in $C^{\mathfrak{n},\beta}\mathcal{D}(M)$, the group of $C^{\mathfrak{n},\beta}$ diffeomorphisms of $M$.
\end{proposition}
\begin{lemma}\label{l8}
For the flow $\psi:I\times M\to M$ of a continuous path $t\to\tau_t\in C^{\mathfrak{n},\beta}(T^\star M)$ as in Proposition \ref{p7}, there exists a unique $C^1$ curve in $C^{\mathfrak{n},\beta}\mathcal{D}(M)$, $\psi^{-1}:I\times M\to M$, such that $p\to \psi^{-1}(t,p)$ is the inverse of $p\to \psi(t,p)$ for each $t$. Moreover,
$$d\psi^{-1}(\partial_t|_{(t,p)}) = -d\psi^{-1}(\tau(t,p)).$$
\end{lemma}
\begin{proof}
Using Proposition \ref{p7}, for each $t_0\in I$, we take $J:=I\cap (-\infty,t_0+\epsilon)$ if $t_0>0$ or $J:=I\cap(t_0-\epsilon,\infty)$ if $t_0<0$, and obtain a flow $\tilde\psi_{t_0}:J\times M\to M$ associated to the continuous path $t\to-\tau_{t_0-t}\in C^{k,\beta}(T^\star M)$. We then define the mapping $\psi^{-1}:I\times M\to M$ by $\psi^{-1}(t,p) := \tilde\psi_t(t,p)$. We see $p\to\psi^{-1}(t,p)$ is a $C^{\mathfrak{n},\beta}$ diffeomorphism by construction, moreover, from the fact that:
$$d\tilde{\psi}_{t_0}(\partial_t|_{(t,p)}) = -\tau(t_0-t,\tilde{\psi}_{t_0}(t,p)),\,\,\,\tilde{\psi}_{t_0}(0,p)=p,$$
we see that the curve $t\to\tilde\psi_{t_0}(t,p)$ is the reverse of the curve $t\to\psi(t,\tilde\psi_{t_0}(t_0,p))$. Therefore, $p\to\psi^{-1}(t,p)$ is the inverse of $p\to \psi(t,p)$ for each $t\in I$. To show that $\psi^{-1}$ is $C^1$, for fixed $t_0\in I$, we observe that any $p\in M$ is given uniquely as $\psi(t_0,q) = p$ for some $q\in M$. By the Implicit Function Theorem, we may then find a unique $C^1$ path $x_q:(t_0-\epsilon,t_0+\epsilon)\to M$ such that $\psi(t,x_q(t)) = p$, and $x_q(t_0) = q$. Taking a derivative, we conclude that $d\psi(\dot{x}(t_0)) = -\tau(t,p)$, and since $d\psi^{-1}|_{T M}$ is well defined, and the inverse of $d\psi|_{T M}$, we have
$$\dot{x}_q(t_0) = -d\psi^{-1}|_{T M}(\tau(t,p)).$$
Now, $\psi^{-1}(t,p)=\psi^{-1}(t,\psi(t,x_q(t)) = x_q(t)$ for $t\in (t_0-\epsilon,t_0+\epsilon)$ and fixed $p$. So we conclude that $d\psi^{-1}(\partial_t)$ exists and
$$d\psi^{-1}(\partial_t|_{(t,p)}) = -d\psi^{-1}(\tau(t,p)).$$
\end{proof}
Given a flow $\psi:I\times M\to M$ with respect to a continuous path $t\to\tau_t\in C^{\mathfrak{n},\beta}(T^\star M)$, we will refer to $\psi^{-1}$ as the \textit{inverse flow} with respect to the flow $\psi$. Finally, our analysis of ODE flows allows us to solve linear equations using an adaptation of the Method of Characteristics: 
\begin{lemma}\label{l9}
Given continuous paths $t\to\tau_t\in C^{\mathfrak{n},\beta}(T^\star M)$, $t\to\phi_t\in C^{\mathfrak{n},\beta}(M,\mathbb{R})$, $t\in I$, and initial data $f\in C^{\mathfrak{n},\beta}(M,\mathbb{R})$, we can solve the problem:
 $$dg(\partial_t|_{(t,p)})=dg(\tau(t,p))+\phi(t,p),\,\,\,g(0,p)=f(p).$$
The unique solution $t\to g_t\in C^{\mathfrak{n},\beta}(M)$, is given by:
$$g(t,p) = (f\circ\psi^{-1})(t,p)+\int_0^t\phi\big(u,\psi(u,\psi^{-1}(t,p))\big)du,$$
where $\psi: I\times M\to M$ is the flow associated to $t\to-\tau_t$.
\end{lemma}
\begin{proof}
We start by showing that the proposed solution satisfies the equation. In order to do so, we recall from Lemma \ref{l8} that:
$$d\psi^{-1}(\partial_t|_{(t,p)}) = d\psi^{-1}(\tau(t,p)).$$
It follows, for \textit{any} function $h\in C^{k,\beta}(M,\mathbb{R})$, that $d(h\circ\psi^{-1})(\partial_t|_{(t,p)}) = d(h\circ\psi^{-1})(\tau(t,p))$. Therefore, we calculate:
\begin{align*}
dg(\partial_t)&=d(f\circ\psi^{-1})(\partial_t)+\phi\big(t,\psi(t,\psi^{-1}(t,p))\big)+\int_0^td\phi\big(u,\psi(u,\psi^{-1}(t,p))\big)(\partial_t)du\\
&=d(f\circ\psi^{-1})(\tau)+\phi(t,p)+\int_0^td\phi\big(u,\psi(u,\psi^{-1}(t,p))\big)(\tau)du\\
&=dg(\tau(t,p))+\phi(t,p).
\end{align*}
For any solution $\tilde g: I\times M\to\mathbb{R}$ of our linear problem, we take 
$$G(t,p):=  \tilde{g}(t,\psi(t,p))-\int_0^t\phi\big(u,\psi(u,p)\big)du - f(p),$$
and calculate:
\begin{align*}
dG(\partial_t|_{(t,p)})&= \Big(d\tilde{g}(\partial_t)+d\tilde{g}(d\psi(\partial_t))-\phi\Big) \circ(t,\psi(t,p))\\
&=\Big(d\tilde{g}(\tau)+\phi+d\tilde{g}(-\tau)-\phi\Big) \circ(t,\psi(t,p))\\
&=0.
\end{align*}   
Since $G(0,p) = \tilde{g}(0,\psi(0,p))-f(p) = 0$ we conclude that $G\equiv 0$. Now given $t\in I$, any $q\in M$ is uniquely identified as $q=\psi(t,p)$ for some $p\in M$, and therefore $p=\psi^{-1}(t,q)$. It follows that 
$$\tilde g(t,q) = (f\circ\psi^{-1})(t,q)+\int_0^t\phi\big(u,\psi(u,\psi^{-1}(t,q))\big)du = g(t,q).$$  
\end{proof}
\begin{lemma}\label{l10}
Given continuous paths $t\to\tau_t\in C^{\mathfrak{n},\beta}(T^\star \mathbb{S}^2)$, $t\to\phi_t\in C^{\mathfrak{n},\beta}(\mathbb{S}^2)$, and initial data $\vec{r}_0\in C^{\mathfrak{n},\beta}(\mathbb{S}^2,\mathbb{R}^3)$, we solve the system:
$$d\vec{r}(\partial_t|_{(t,p)}) = d\vec{r}(\tau^\#(t,p))+\phi(t,p)\frac{\vec{r}}{|\vec{r}|},\,\,\,\vec{r}(0,p) = \vec{r}_0(p)$$
with unique solution $t\to \vec{r}_t\in C^{\mathfrak{n},\beta}(\mathbb{S}^2,\mathbb{R}^3)$ given by:
$$\vec{r}(t,p) = \Big(1+\frac{\int_0^t\phi\big(u,\psi(u,\psi^{-1}(t,p))\big)}{(|\vec{r}_0|\circ\psi^{-1})(t,p)}\Big)(\vec{r}_0\circ\psi^{-1})(t,p),$$
as long as $|\vec{r}|(t,p)>0$, whereby $\psi:I\times\mathbb{S}^2\to\mathbb{S}^2$ is the flow associated to $t\to-\tau_t$.
\end{lemma}
\begin{proof}
We start by observing that any solution of our system satisfies:
\begin{align*}
\partial_t|\vec{r}| &= \frac{1}{|\vec{r}|}\delta_{\mu\nu}(dr^\mu)(\partial_t)r^\nu = \frac{1}{|\vec{r}|}\delta_{\mu\nu}(dr^\mu)(\tau^\#)r^\nu+\frac{1}{|\vec{r}|^2}\phi\delta_{\mu\nu}r^\mu r^\nu\\
&=\tau^\#|\vec{r}|+\phi,
\end{align*}
as long as $|\vec{r}|> 0$. From Lemma \ref{l9}, we therefore conclude:
\begin{equation}\label{e24}
|\vec{r}|(t,p)=(|\vec{r}_0|\circ\psi^{-1})(t,p)+\int_0^t\phi\big(u,\psi(u,\psi^{-1}(t,p))\big)du,
\end{equation}
moreover,
$$d\big(\frac{\vec{r}}{|\vec{r}|}\big)(\partial_t)=\frac{1}{|\vec{r}|}d\vec{r}(\partial_t)-\frac{\vec{r}}{|\vec{r}|^2}\partial_t|\vec{r}| = \frac{1}{|\vec{r}|}\big(d\vec{r}(\tau^\#)+\phi\frac{\vec{r}}{|\vec{r}|}\big)-\frac{\vec{r}}{|\vec{r}|^2}\big(\tau^\#|\vec{r}|+\phi\big)=d\big(\frac{\vec{r}}{|\vec{r}|}\big)(\tau^\#).$$
Again from Lemma \ref{l9} (with $\phi\equiv0$), we conclude, $\frac{\vec{r}}{|\vec{r}|} = \frac{\vec{r}_0}{|\vec{r}_0|}\circ\psi^{-1},$ and therefore:
$$\vec{r}(t,p) = |\vec{r}|(t,p) \frac{\vec{r}_0}{|\vec{r}_0|}\circ\psi^{-1}(t,p)  = \Big(1+\frac{\int_0^t\phi\big(u,\psi(u,\psi^{-1}(t,p))\big)}{(|\vec{r}_0|\circ\psi^{-1})(t,p)}\Big)(\vec{r}_0\circ\psi^{-1})(t,p), $$
as long as $|\vec{r}|>0$. We leave for the reader the straight forward exercise of checking that this expression indeed solves our system.
\end{proof}
Finally, we're ready to show uniqueness of our path of embeddings satisfying (\ref{e22},\ref{e23}). We equip a norm acting on the pair $(\tau_t,\phi_t)$, $t\in[0,\delta]$, for any $\delta\leq \epsilon$, as given by 
$$||(\tau,\phi)||_{\infty}:=\sup_t||\tau(t,\cdot)||_{2,\beta}+\sup_t||\phi(t,\cdot)||_{2,\beta}.$$ 
From the outset, we're given a continuous path of metrics $t\to\gamma_t\in C^{3,\beta}(\text{Sym}(T^\star\mathbb{S}^2\otimes T^\star\mathbb{S}^2))$, and an initial isometric embedding $(\mathbb{S}^2,\gamma_0)\hookrightarrow (\mathcal{A},\sigma)$ represented by $\vec{r}_0\in C^{3,\beta}(\mathbb{S}^2,\mathbb{R}^3)$, that traces out the boundary of a star-shaped region in $\mathbb{R}^3$. We also assume the path $t\to\gamma_t$ is continuously Fr\'{e}chet differentiable with respect to the norm on $C^{2,\beta}(\text{Sym}(T^\star\mathbb{S}^2\otimes T^\star\mathbb{S}^2))$. We next consider two paths $t\to{^1}\vec{r}_t, t\to{^2}\vec{r}_t$ associated to two pairings $({^1}\tau_t,{^1}\phi_t), ({^2}\tau_t,{^2}\phi_t)$ each pair coupled via (\ref{e22},\ref{e23}) on $t\in[0,\epsilon]$. Our goal is to show, if both also represent paths of isometric embeddings for $t\to\gamma_t$, then, ${^1}\vec{r}_t = {^2}\vec{r}_t$, on all of $t\in[0,\epsilon]$.
\begin{theorem}\label{t4}(Uniqueness)
The path $t\to\vec{r}\in C^{\mathfrak{n},\beta}(\mathbb{S}^2,\mathbb{R}^3)$ described in Theorem \ref{t3}, is unique.
\end{theorem}
\begin{proof}
	We start by choosing antipodal points on $\mathbb{S}^2$, we can then use stereographic projections from these to arrange a covering of $\mathbb{S}^2$ by co-ordinate neighborhoods $\mathcal{U}_n,\mathcal{U}_s\subset\mathbb{S}^2$. We choose these so that they have the same image under stereographic projection, say $B_{2R}(0)\subset\mathbb{R}^2$, where $B_{2R}(0)$ represents the closed ball centered at the origin of radius $2R$. For sufficiently large $R$, we obtain a sub-cover $\{\mathcal{U}'_n,\mathcal{U}'_s\}$, subordinate to the original, with coincident projections onto $B_{R}(0)$. Now, given $\Lambda>0$ sufficiently large so that $||({^1}\tau,{^1}\phi)||_\infty,||({^2}\tau,{^2}\phi)||_\infty \leq \Lambda$, we choose $\delta$ sufficiently small so that any $t\to\tau_t$ satisfying $||\tau||_\infty \leq \Lambda$ gives rise to a flow and inverse flow satisfying $\psi(t,\cdot),\psi^{-1}(t,\cdot):\mathcal{U}_{\cdot}'\to \mathcal{U}_{\cdot}$, $t\in[0,\delta]$. In stereographic coordinates:
$$\psi(t,p)^i=p^i-\int_0^t\tau(u,\psi(u,p))^idu,\,\,\,\psi^{-1}(t,p)^i=p^i+\int_0^t\tau(t-u,\psi^{-1}(u,p))^idu.$$
To start, we need to show $\sup_t||d\psi(t,\cdot)||_{1,\beta},||d\psi^{-1}(t,\cdot)||_{1,\beta}\leq C(\Lambda,\delta)$. To this end:
$$|\vec{\psi}_i(t,p)|\leq 1+\int_0^t|\nabla\vec{\tau}(u,\psi(u,p))||\vec{\psi}_i(u,p)|du\leq 1+C\Lambda\int_0^t| \vec{\psi}_i(u,p)|du,$$
and from standard ODE differential inequalities we have:
$$| \vec{\psi}_i(t,p)| \leq(t+|\vec{\psi}_i(0,p)|)e^{C\Lambda t} = (1+t)e^{C\Lambda t}.$$
Using a coordinate patch-work we therefore conclude $\sup_t||d\psi(t,\cdot)||_{0}\leq C_1(\Lambda,\delta)$.
Similarly,
\begin{align*}
	|\vec{\psi}_{ij}(t,p)|&\leq \int_0^t |\nabla\vec{\tau}(u,\psi(u,p))||\vec{\psi}_{ij}(u,p)|+|\nabla^2\vec{\tau}(u,\psi(u,p))||\vec{\psi}_i(u,p)||\vec{\psi}_j(u,p)|du\\
	&\leq C\Lambda\big(\int_0^t|\vec{\psi}_{ij}(u,p)|du+tC_1(\Lambda,\delta)^2\big),
\end{align*}
 from which differential inequalities again yield $\sup_t||d\psi(t,\cdot)||_1\leq C_2(\Lambda,\delta)$. Finally, we have
 \begin{align*}
 |\vec{\psi}_{ij}(t,p_1) - \vec{\psi}_{ij}(t,p_2)|&\leq \int _0^t \Big(|\nabla\vec{\tau}(u,\psi(u,p_1))-\nabla\vec{\tau}(u,\psi(u,p_2))||\vec{\psi}_{ij}(u,p_1)|\\
 &+|\nabla\vec{\tau}(u,\psi(u,p_2))||\vec{\psi}_{ij}(u,p_1)-\vec{\psi}_{ij}(u,p_2)|\\
 &+|\nabla^2\vec{\tau}(u,\psi(u,p_1))-\nabla^2\vec{\tau}(u,\psi(u,p_2))||\vec{\psi}_i(u,p_1)||\vec{\psi}_j(u,p_1)|\\
 &+|\nabla^2\vec{\tau}(u,\psi(u,p_2))||\vec{\psi}_i(u,p_1)-\vec{\psi}_i(u,p_2)||\vec{\psi}_j(u,p_1)|\\
 &+	|\nabla^2\vec{\tau}(u,\psi(u,p_2))||\vec{\psi}_i(u,p_2)| |\vec{\psi}_j(u,p_1)-\vec{\psi}_j(u,p_2)|\Big)du\\
 &\leq tC\Lambda C_1(\Lambda,\delta)C_2(\Lambda,\delta)|\vec{p}_1-\vec{p}_2|\\
 &+C\Lambda\int_0^t |\vec{\psi}_{ij}(u,p_1)-\vec{\psi}_{ij}(u,p_2)| du\\
 &+tC\Lambda C_1(\Lambda,\delta)^{2+\beta}|\vec{p}_1-\vec{p}_2|^\beta\\
 &+2C\Lambda C_2(\Lambda,\delta)C_1(\Lambda,\delta)|\vec{p}_1-\vec{p}_2|.
 \end{align*}
Consequently, we observe an ODE differential inequality for the quantity $\sup_{p_1\neq p_2}\frac{|\vec{\psi}_{ij}(t,p_1) - \vec{\psi}_{ij}(t,p_2)|}{|\vec{p}_1-\vec{p}_2|^\beta}$, yielding the desired estimate through a coordinate patch-work, namely $\sup_t||d\psi(t,\cdot)||_{1,\beta}\leq C(\Lambda,\delta)$. The exact same analysis yields also $||d\psi^{-1}(t,\cdot)||_{1,\beta}\leq \tilde{C}(\Lambda,\delta)$.\\
\indent Now, from $({^1}\tau,{^1}\phi),({^2}\tau,{^2}\phi)$, we wish to estimate $\sup_t||(d({^1}\psi)-d({^2}\psi))(t,\cdot)||_{1,\beta}$, we observe: 
\begin{align*}
|{^1}\vec{\psi}(t,p)-{^2}\vec{\psi}(t,p)| &\leq \int_0^t|({^1}\vec{\tau})(u,{^1}\psi(u,p))-({^1}\vec{\tau})(u,{^2}\psi(u,p))|du\\
&+\int_0^t|({^1}\vec{\tau})(u,{^2}\psi(u,p))-{^2}\vec{\tau}(u,{^2}\psi(u,p))|du\\
&\leq C\Lambda\int_0^t|{^1}\vec{\psi}(u,p)-{^2}\vec{\psi}(u,p)|du+t||{^1}\tau-{^2}\tau||_{\infty}.
\end{align*}
Again results relating to ODE differential inequalities yield,  

$$|{^1}\vec{\psi}(t,p)-{^2}\vec{\psi}(t,p)|\leq \frac{||{^1}\tau-{^2}\tau||_\infty}{C\Lambda}\Big(\frac{1}{C\Lambda}(e^{C\Lambda \delta}-1)-\delta\Big)=O(\delta^2)||{^1}\tau-{^2}\tau||_\infty.$$
The stereographic projection is a local conformal diffeomorphism $p
\to(x^1,x^2)(p)$, with the push-forward of the standard metric given by:
$$\ring{\gamma}_{ij}dx^idx^j = \frac{4}{(1+(x^1)^2+(x^2)^2)^2} (dx^1)^2+(dx^2)^2.$$
As a result, we conclude:
$$\text{dist}({^1}\psi(t,x),{^2}\psi(t,x))\leq 2|{^1}\vec{\psi}(t,p)-{^2}\vec{\psi}(t,p)|=O(\delta^2)||{^1}\tau-{^2}\tau||_\infty.$$
Similarly: 
\begin{align*}
|({^1}\vec{\psi})_i&-({^2}\vec{\psi})_i| (t,p)\leq \int_0^t|({^1}\vec{\tau})_k(u,{^1}\psi(u,p))({^1}\psi)^k_i(u,p)-({^2}\vec{\tau})_k(u,{^2}\psi(u,p))({^2}\psi)^k_i(u,p)|du\\
&\leq\int_0^t|(\nabla({^1}\vec{\tau}))(u,{^1}\psi(u,p))|\big|({^1}\vec{\psi})_i-({^2}\vec{\psi})_i\big|du\\
&+\int_0^t\big|(\nabla({^1}\vec{\tau}))(u,{^1}\psi(u,p))-(\nabla({^2}\vec{\tau}))(u,{^2}\psi(u,p))\big||({^2}\vec{\psi})_i|du\\
&\leq C\Lambda\int_0^t |({^1}\vec{\psi})_i-({^2}\vec{\psi})_i\big|(u,p)du\\
&+C(\Lambda,\delta)\int_0^t \big|(\nabla({^1}\vec{\tau}))(u,{^1}\psi(u,p))-(\nabla({^1}\vec{\tau}))(u,{^2}\psi(u,p))|du\\
&+ C(\Lambda,\delta)\int_0^t| (\nabla({^1}\vec{\tau}))(u,{^2}\psi(u,p))-(\nabla({^2}\vec{\tau}))(u,{^2}\psi(u,p))\big| du\\
&\leq C\Lambda\int_0^t |({^1}\vec{\psi})_i-({^2}\vec{\psi})_i\big|(u,p)du+t\tilde{C}(\Lambda,\delta)||\tau_1-\tau_2||_\infty,
\end{align*}
yielding again $\sup_t||d({^1}\psi)(t,\cdot)-d({^2}\psi)(t,\cdot)||_{0} \leq O(\delta^2)||{^1}\tau-{^2}\tau||_{\infty}$. Analogously, we also obtain $\sup_t||d({^1}\psi^{-1})(t,\cdot)-d({^2}\psi^{-1})(t,\cdot)||_{0} \leq O(\delta^2)||{^1}\tau-{^2}\tau||_{\infty}$. In the notation of Proposition \ref{p6}, writing $\varphi = |\vec{r}|$, we observe from (\ref{e24}):
\begin{align*}
||({^1}\varphi-{^2}\varphi)(t,\cdot)||_{1,0}&\leq||\big(|\vec{r}_0|\circ{^1}\psi^{-1}-|\vec{r}_0|\circ{^2}\psi^{-1}\big)(t,\cdot)||_{1,0}\\
&+\int_0^t||{^1}\phi\big(u,{^1}\psi(u,{^1}\psi^{-1}(t,\cdot))\big)-{^1}\phi\big(u,{^1}\psi(u,{^2}\psi^{-1}(t,\cdot))\big)||_{1,0}du\\
&+\int_0^t||{^1}\phi\big(u,{^1}\psi(u,{^2}\psi^{-1}(t,\cdot))\big)-{^1}\phi\big(u,{^2}\psi(u,{^2}\psi^{-1}(t,\cdot))\big)||_{1,0}du\\
&+ \int_0^t ||{^1}\phi\big(u,{^2}\psi(u,{^2}\psi^{-1}(t,\cdot))\big)-{^2}\phi\big(u,{^2}\psi(u,{^2}\psi^{-1}(t,\cdot))\big)||_{1,0}du\\
&\leq C(|||\vec{r}_0|||_{2,0})O(\delta^2)||{^1}\tau-{^2}\tau||_{\infty}+t\Lambda C(\Lambda,\delta)O(\delta^2)||{^1}\tau-{^2}\tau||_\infty+t||{^1}\phi-{^2}\phi||_\infty\\
&\leq O(\delta)||({^1}\tau-{^2}\tau,{^1}\phi-{^2}\phi)||_\infty.
\end{align*}
We now note, if we apply our estimates thus far to (\ref{e19}) (noting the relationship $\Psi = \psi^{-1}$), and utilize a coordinate patchwork, we conclude with the estimate $\sup_t||d({^1}\psi^{-1})-d({^2}\psi^{-1})||_{1}\leq O(\delta)||({^1}\tau-{^2}\tau,{^1}\phi-{^2}\phi)||_\infty$. Subsequently, we can introduce this new estimate above, to conclude $||({^1}\varphi-{^2}\varphi)(t,\cdot)||_{1,\beta}\leq O(\delta) ||({^1}\tau-{^2}\tau,{^1}\phi-{^2}\phi)||_\infty$. Again, using (\ref{e19}) and a subsequent estimate as above we have: 
\begin{align*}
\sup_t||d({^1}\psi^{-1})-d({^2}\psi^{-1})||_{1,\beta}&\leq O(\delta)||({^1}\tau-{^2}\tau,{^1}\phi-{^2}\phi)||_\infty,\\
||{^1}\varphi-{^2}\varphi||_\infty&\leq O(\delta)||({^1}\tau-{^2}\tau,{^1}\phi-{^2}\phi)||_\infty.
\end{align*}
Consequently, Lemma \ref{l10} yields
$$\sup_t||{^1}\vec{r}_t-{^2}\vec{r}_t||_{2,\beta}\leq O(\delta)||({^1}\tau-{^2}\tau,{^1}\phi-{^2}\phi)||_\infty.$$
From (\ref{e22}), we leave to the reader the analogous argument as in Lemma \ref{l7} to observe from subsequent Schauder estimates:
$$||{^1}\tau-{^2}\tau||_\infty\leq C \sup_t||{^1}\vec{r}_t-{^2}\vec{r}_t||_{2,\beta}\leq O(\delta)||({^1}\tau-{^2}\tau,{^1}\phi-{^2}\phi)||_\infty. $$ 
From (\ref{e12}), we now observe the equation:
\begin{align*}
\nabla\cdot\nabla\cdot\Big(({^1}\phi_t-{^2}\phi_t)\big((\tr_{\gamma_t}h({^1}\vec{r}_t))\gamma_t-h({^1}\vec{r}_t)\big)\Big)&=\nabla\cdot\nabla\cdot\Big(({^2}\phi_t)\big(\tr_{\gamma_t}\big(h({^2}\vec{r}_t)-h({^1}\vec{r}_t)\big)\gamma_t+ h({^1}\vec{r}_t)-h({^2}\vec{r}_t)\big)\Big)\\
&\qquad+2\nabla\cdot(\mathcal{K}_t({^1}\tau_t-{^2}\tau_t)).
\end{align*}
The same analysis that proceeds (\ref{e12}) can be applied to this equation, specifically, that no more than two derivatives survive on either ${^1}\vec{r}_t,{^2}\vec{r}_t$. We therefore conclude from Schauder estimates:
\begin{align*}
||{^1}\phi_t-{^2}\phi_t||_{2,\beta}&\leq C(\Lambda,\delta)\big(||{^1}\vec{r}_t-{^2}\vec{r}_t||_{2,\beta}+||{^1}\tau_t-{^2}\tau_t||_{1,\beta}+|| {^1}\phi_t-{^2}\phi_t||_{0,\beta}\big)\\
&\leq \tilde{C}(\Lambda,\delta)\big(||{^1}\vec{r}_t-{^2}\vec{r}_t||_{2,\beta}+||{^1}\tau_t-{^2}\tau_t||_{1,\beta}\big),
\end{align*}
giving,
$$||{^1}\phi-{^2}\phi||_\infty \leq O(\delta)||({^1}\tau-{^2}\tau,{^1}\phi-{^2}\phi)||_\infty.$$
Finally, we therefore observe:
$$||({^1}\tau-{^2}\tau,{^1}\phi-{^2}\phi)||_\infty \leq O(\delta) ||({^1}\tau-{^2}\tau,{^1}\phi-{^2}\phi)||_\infty,$$
so that sufficiently small $\delta$ induces that ${^1}\tau_t = {^2}\tau_t$, ${^1}\phi_t = {^2}\phi_t$, on $t\in[0,\delta]$. By the uniqueness result of Lemma \ref{l10}, we conclude ${^1}\vec{r}_t = {^2}\vec{r}_t$ for $t\in[0,\delta]$. If both ${^1}\vec{r}_t,{^2}\vec{r}_t$ exist for $t\in[0,\epsilon]$ whereby $\delta< \epsilon$, then a simple continuity argument using our analysis above ensures that ${^1}\vec{r}_t = {^2}\vec{r}_t$ throughout $t\in[0,\epsilon]$.
\end{proof}
\begin{corollary}\label{c5}
Suppose $(\mathcal{A},\sigma)$ is a convex Null Cone. For $\mathfrak{n}\geq 2$, consider also a continuous path of metrics $t\to\gamma_t\in C^{\mathfrak{n}+1,\beta}(\text{Sym}(T^\star\mathbb{S}^2\otimes T^\star\mathbb{S}^2))$, $t\in(-b,b)$, with the following properties:
\begin{enumerate}
\item There exists an isometric embedding $r_0:(\mathbb{S}^2,\gamma_0)\hookrightarrow(\mathcal{A},\sigma)$,
\item $t\to\gamma_t\in C^{\mathfrak{n},\beta}(\text{Sym}(T^\star\mathbb{S}^2\otimes T^\star\mathbb{S}^2))$ is $\mathfrak{m}$-times continuously Fr\'{e}chet differentiable within $C^{\mathfrak{n},\beta}(\text{Sym}(T^\star\mathbb{S}^2\otimes T^\star\mathbb{S}^2))$.
\end{enumerate}
Then, there exists $0<\epsilon<b$, and a unique family of embeddings $r_t:(\mathbb{S}^2,\gamma_t)\hookrightarrow (\mathcal{A},\sigma)$, $t\in(-\epsilon,\epsilon)$, with a associated path $t\to\vec{r}_t\in C^{\mathfrak{n}+1,\beta}(\mathbb{S}^2,\mathbb{R}^3)$ that is $\mathfrak{m}$-times continuously Fr\'{e}chet differentiable within $C^{\mathfrak{n},\beta}(\mathbb{S}^2,\mathbb{R}^3)$. Moreover, $\epsilon$ is independent of $\mathfrak{n},\mathfrak{m}$, and the first Fr\'{e}chet derivative satisfies: 
\begin{align*}
\dot{\vec{r}}_t &= d\vec{r}_t(\tau_t^\#)+\phi_t\frac{\vec{r}_t}{|\vec{r}_t|},\\
L_{h(\vec{r}_t)}(\tau_t) = \dot{\gamma}_t-&\frac12\tr_{h(\vec{r}_t)}(\dot{\gamma}_t)h(\vec{r}_t),\,\,\,\phi_t=\frac{\tr_{\gamma_t}(\dot{\gamma}_t)-2\nabla^{\gamma_t}\cdot\tau_t}{\tr_{\gamma_t}h(\vec{r}_t)},
\end{align*}
whereby $\tau_t\perp\text{Ker}(L_{h(\vec{r}_t)})$ with respect to the $L^2$ inner product induced by $\gamma_t$.
\end{corollary}
\begin{proof}
On the interval $t\in[0,\epsilon)$ the result follows from Theorems \ref{t3} and \ref{t4}. To extend this to $t\in(-\epsilon,0]$ (following a shrinkage of $\epsilon$, if necessary), we simply observe Theorems \ref{t3}, \ref{t4} for the path $t\to \gamma_{-t}$, noting the continuity of upto $\mathfrak{m}$ Fr\'{e}chet derivatives of the pair $(\tau_t,\phi_t)$ in the final step of Theorem \ref{t3}.	
\end{proof}

\section{Null Cone: Closedness}
In order for us to extend our unique path of isometric embeddings from Theorem \ref{t3}, we will need the following control:
\begin{proposition}\label{p8}
We consider a path of isometric embeddings $r_t:(\mathbb{S}^2,\gamma_t)\hookrightarrow (\mathcal{A},\sigma)$, with associated path $t\to\vec{r}_t\in C^{2,\beta}(\mathbb{S}^2,\mathbb{R}^3)$ on an interval $[0,\epsilon)$, for some $\epsilon<b$, according to the the hypotheses of Theorem \ref{t3}. Provided:
$$\inf_{t\in[0,\epsilon)}\big(\inf_{\mathbb{S}^2}\varphi_t\big)\geq C_-,\,\,\,\sup_{t\in[0,\epsilon)}\big(\sup_{\mathbb{S}^2}\varphi_t\big)\leq C_+,\,\,\, 
\sup_{t\in[0,\epsilon)}||d\varphi_t||_0\leq C_0,$$
for some $C_0>0$, $S_-<C_-<C_+<S_+$, and $\varphi_t:=|\vec{r}_t|$, then the path $t\to r_t$ can be uniquely extended.
\end{proposition}
 \begin{proof}
We recall the local equation:
$$g_{kl}(\varphi_t,\Psi_t(x))(\Psi_t)^k_{,i}(\Psi_t)^l_{,j} = (\gamma_t)_{ij}.$$
Since we can find a constant $C_1>0$ such that:
$$\frac{1}{C_1}\mathring\gamma\leq \frac{1}{s^2}g_s\leq C_1\mathring\gamma$$
throughout $\{C_-\leq s\leq C_+\}\subset\Omega$, we immediately conclude via a coordinate patch-work, that $\sup_t||\varphi_t d\Psi_t||_0\leq C_2$. From the $C^0$ control of $\varphi_t$ it follows that $\sup_t||d\Psi_t||_0\leq C_3$. From (\ref{e19}) we observe $||d\Psi||_{1,\beta}\leq C_4$, and therefore $\omega = \varphi\circ\Psi^{-1}$ satisfies $\sup_t||\omega_t||_1\leq C_5$. From (\ref{e21}) we conclude, using interior Schauder estimates, that $\sup_t||\omega_t||_{2,\beta}\leq C_6$, therefore $\sup_t||\varphi_t||_{2,\beta}\leq C_7$. Since $C^{2,\beta}(\mathbb{S}^2)\subset\subset C^{2,\alpha}(\mathbb{S}^2)$ for any $\alpha<\beta$, we have a convergent sequence $\{\varphi_{t_n}\}\in C^{2,\alpha}(\mathbb{S}^2)$ for some $\{t_n\}\subset[0,\epsilon)$ such that $t_n\to\epsilon$, say $\varphi_{t_n}\to \varphi_\star\in C^{2,\alpha}(\mathbb{S}^2)$. Moving to a subsequence if necessary, we conclude from (\ref{e19}) that $\Psi_{t_n}\to \Psi_\star\in C^{2,\alpha}(\mathcal{D})$. By continuity, we conclude that $(\varphi_\star,\Psi_\star)$ represents a $C^{2,\alpha}$ isometric embedding $r_\star:(\mathbb{S}^2,\gamma_\epsilon)\hookrightarrow(\mathcal{A},\sigma)$. By Proposition \ref{p6}, $\vec{r}_\star\in C^{\mathfrak{n}+1,\beta}(\mathbb{S}^2,\mathbb{R}^3)$. By the uniqueness of Corollary \ref{c5} 	the result follows.
 \end{proof}

\subsection{Additional Setup}
We will be able to gain the control needed to apply Proposition \ref{p8} by imposing asymptotic constraints on $\Omega$. We will describe these in this section.\\\\
\indent From this point forth, we will assume a rescaling $\ubar L_a = a\ubar L$, $a>0$, exists such that:
 $$S_+(a)=\infty,\,\,\,\ubar L\log a = -\kappa\iff D_{\ubar L_a}\ubar L_a = 0.$$
For convenience, we will continue to refer to $\ubar L$ instead of $\ubar L_a$. Consequently, our foliation $\{\Sigma_s\}$ becomes a geodesic foliation.
 \begin{definition}\label{d4}
 We say $\Omega$ satisfies the null energy condition provided:
 $$G(\ubar L,\ubar L)\geq0.$$
\end{definition}
The null energy condition is a natural physical constraint (clearly independent of our choice of re-scaling) motivated by the fact that the Einstein tensor is proportional to the stress energy tensor in the theory of general relativity. If we couple the null energy equation with equation (\ref{i3}), also known as the famous \textit{optical Raychaudhuri equation}, we observe the following:
\begin{lemma}\label{l11}
If $\Omega$ satisfies the null energy condition, then $\Omega$ is a Null Cone provided $\tr\ubar\chi>0$ on a single cross-section $\Sigma\subset\Omega$.	
\end{lemma}
\begin{proof}
Assume for a contradiction that $\tr\ubar\chi(p)<0$ for some $p\in \Omega$. From equation (\ref{i3}):
$$\frac{d}{ds}\frac{1}{\tr\ubar\chi_s} = \frac12+\frac{|\hat{\chi}_s|^2+G(\ubar L,\ubar L)}{\tr\ubar\chi_s^2}\geq \frac12\implies \frac{1}{\tr\ubar\chi_s}-\frac{1}{\tr\ubar\chi(p)}\geq \frac{1}{2}(s-s(p)),$$
as long as $\tr\ubar\chi_s<0$. From the fact that $S_+=\infty$, we observe the existence of some $s_\star$ such that:
$$\frac{1}{\tr\ubar\chi_s}\to 0,\,\,\,\text{as}\,\,\,s\to s_\star^-,$$
causing $\tr\ubar\chi_s\to -\infty$ as $s\to s_\star^-$. This contradicts the fact that $S_+ = \infty$ since $\Omega$ ceases to be smooth at $s = s_\star$. So we conclude that $\tr\ubar\chi\geq 0$ throughout $\Omega$. From our equation above, we have for any $p\in \Sigma$:
$$\frac{1}{\tr\ubar\chi_s}\geq \frac{1}{2}(s-s_0)+\frac{1}{\tr\ubar\chi(p)}>0,$$
for all $s\geq s_0$, so that $\tr\ubar\chi_s>0$. Moreover, 
$$\infty>\frac{1}{\tr\ubar\chi(p)}-\frac{1}{2}(s-s_0)\geq\frac{1}{\tr\ubar\chi_s} ,$$
for all $s\leq s_0$, so again $\tr\ubar\chi_s>0$.
\end{proof}
We observe for $\Omega$ satisfying the null energy condition, that:
$$\tr\ubar\chi_s\leq \frac{1}{\frac{1}{2}(s-s_0)+\frac{1}{\tr\ubar\chi_{s_0}}}< \frac{2}{s-s_0}$$
\indent Given a cross-section $\Sigma\subset\Omega$, and $v\in T_q(\Sigma)$ we may extend $v$ along the generator $\beta_q^{\ubar L}$ according to the equation:
\begin{align*}
\dot{V}(s) &= D_{V(s)}\ubar L\\
V(0)&=v.
\end{align*} 
As a result, we have $\frac{d}{ds}{\langle V(s),\ubar L\rangle}=\frac12V(s)\langle\ubar L,\ubar L\rangle+\langle V(s),D_{\ubar L}\ubar L\rangle=\kappa\langle V(s),\ubar L\rangle$, and since $v\in T_p\Omega\iff \langle \ubar L|_p,v\rangle = 0$, it follows that $\langle V(0),\ubar L_p\rangle=0$. By uniqueness, we conclude $\langle V(s),\ubar L\rangle=0$ for all $s$. As a result, any section $W\in\Gamma(T\Sigma)$ extends througout $\Omega$ satisfying $[\ubar L,W]=0$. Along each generator, $0=[\ubar L,W]s = \ubar L(Ws)=\frac{d}{ds}(Ws)$, so that $Ws|_\Sigma=0$ ensures $Ws = 0$ throughout $\Omega$. We conclude, $W|_{\Sigma_s}\in\Gamma(T\Sigma_s)$, and denote by $E(\Sigma)\subset\Gamma(T\Omega)$ the set of such extensions off of $\Sigma$ along $\ubar L$.  We also note that linear independence is preserved along generators, allowing us to extend a local basis $\{X_1,X_2\}\subset\Gamma(TU)$ ($\mathcal{U}\subset \Sigma$), throughout $\pi^{-1}(\mathcal{U})$. Given a local basis extension $\{X_i\}\subset E(\Sigma_{s_0})$ along $\{\Sigma_s\}$, we impose the following assumptions:
\begin{align}
s^{-2}(\gamma_s)_{ij} &= O(1)\tag{A1}\label{a1}\\
\sup_{\mathbb{S}^2}|\mathfrak{r}_s^2\mathcal{K}_s|,\frac{\text{diam}(\Sigma_s)}{\mathfrak{r}_s}&=O(1)\tag{A2}\label{a2}\\
a_s:=(s-s_0)^2\tr\ubar\chi_s-2(s-s_0)&=O(1)\tag{A3}\label{a3}\\
s^{1+\delta}||\ubar\alpha_s||_{1,\beta},s^{2+\delta}(\pounds_{\ubar L}\ubar\alpha_s)_{ij} &= O(1)\tag{A4}\label{a4}
\end{align}
whereby $4\pi\mathfrak{r}_s^2:=|\Sigma_s|$, $0<\delta<1$.\\
\indent It's a straight forward calculation from (\ref{i3}) to show that
$$\frac{da_s}{ds} = -\frac12\frac{a^2}{(s-s_0)^2}-(s-s_0)^2(|\hat{\ubar\chi}_s|^2+G(\ubar L,\ubar L)-\kappa\tr\ubar\chi_s).$$
So given the case of an affine foliation $\{\Sigma_s\}$ and assuming the null energy condition, from the fact that $\frac{da_s}{ds}\leq 0$, $a(s_0) = 0$, we deduce from (\ref{a3}) that $0\leq -a_s< \alpha^2$ for some $\alpha$. We conclude:
$$\frac{2}{s-s_0}-\frac{\alpha^2}{(s-s_0)^2}\leq \tr\ubar\chi_s <\frac{2}{s-s_0}.$$
This allows us to specify the decay of the inverse metric $\gamma^{-1}$. Since we have from (\ref{i1}) that 
$$\frac{d}{ds}\sqrt{\frac{\det(\gamma_s)_{ij}}{\det{\mathring\gamma_{ij}}}}=\tr\ubar\chi\sqrt{\frac{\det(\gamma_s)_{ij}}{\det\mathring\gamma_{ij}}}$$
it follows that $C_1s^2\leq \sqrt{\frac{\det(\gamma_s)_{ij}}{\det\mathring\gamma_{ij}}} = \sqrt{\frac{\det{(\gamma_{s_0}})_{ij}}{\det\mathring\gamma_{ij}}}e^{\int_{s_0}^s\tr\ubar\chi dt}\leq C_2s^2$, and therefore from (\ref{a1}) we have $\gamma^{ij}=O(s^{-2})$.

\begin{lemma}\label{l12} From (\ref{a4}) we conclude $\Omega$ is convex beyond an affine leaf $\Sigma_s$ for sufficiently large $s$. In the case that $\Omega$ also satisfies the bound, 
$$|\ubar\chi_s|^2G(\ubar L,\ubar L)\geq 2\tr\ubar\chi_s\hat{\ubar\alpha_s}\cdot\cdot\hat{\ubar\chi}_s,$$
we conclude $\frac12(\tr\ubar\chi_s)^2\geq |\hat{\chi}_s|^2$
globally, and $\Omega$ is convex provided some cross-section $\Sigma\subset\Omega$ is convex.
\end{lemma}
\begin{proof}
We show that $|\hat{\ubar\chi}_s|^2 = O(s^{-4})$ which implies the first part of the lemma. From (\ref{i1}-\ref{i3}), using a frame extension $\{X_i\}\subset E(\Sigma_{s_0})$:
\begin{align*}
\frac{d}{ds}|{\hat{\ubar\chi}_s}|^2 &= \frac{d}{ds}(\gamma_s^{ij}\gamma_s^{kl}(\hat{\ubar\chi}_s)_{ik}(\hat{\ubar\chi}_s)_{jl})\\
&=-4\ubar\chi_s^{ij}\gamma_s^{kl}(\hat{\ubar\chi}_s)_{ik}(\hat{\ubar\chi}_s)_{jl}+2\gamma_s^{ij}\gamma_s^{kl}\big(\frac{d}{ds}((\ubar\chi_s)_{ik}-\frac12\tr\ubar\chi_s(\gamma_s)_{ik})(\hat{\ubar\chi}_s)_{jl}\big)\\
&=-2\tr\ubar\chi_S|\hat{\ubar\chi}_s|^2+2\gamma_s^{ij}\gamma_s^{kl}\big(-(\ubar\alpha_s)_{ik}+(\ubar\chi_s)^2_{ik}-\tr\ubar\chi_s(\ubar\chi_s)_{ik})(\hat{\ubar\chi}_s)_{jl}\big)\\
&=-2\tr\ubar\chi_s|\hat{\ubar\chi}_s|^2-2\hat{\ubar\alpha}_s\cdot\cdot\hat{\ubar\chi}_s+2\ubar\chi_s^2\cdot\cdot\hat{\ubar\chi}_s-2\tr\ubar\chi_s|\hat{\ubar\chi}_s|^2\\
&=-2\hat{\ubar\alpha}_s\cdot\cdot\hat{\ubar\chi}_s-2\tr\ubar\chi_s|\hat{\ubar\chi}_s|^2\\
&\leq s^{5+\delta}|\ubar\alpha_s|^2+|\hat{\ubar\chi}_s|^2(\frac{1}{s^{5+\delta}}-2\tr\ubar\chi_s)\\
&\leq \frac{C_1}{s^{1+\delta}}+|\hat{\ubar\chi}_s|^2(-\frac{4}{s}+\frac{C_2}{s^2})
\end{align*}
 having used (\ref{a4}) in the final inequality. It follows that:
$$|\hat{\ubar\chi}_s|^2e^{\int_{s_0}^s\frac{4}{t}-\frac{C_2}{t^2}dt}\leq \int_{s_0}^s\frac{C_3}{t^{1+\delta}}dt$$	and therefore $|\hat{\ubar\chi}_s|^2 = O(s^{-4})$. For the second part of the lemma we leave the reader the simple verification from (\ref{i3}), that $D_s:=\frac12 \tr\ubar\chi_s - \frac{|\hat{\ubar\chi}_s|^2}{\tr\ubar\chi_s}$ satisfies the equation
$$\frac{dD}{ds} = -D_s^2-\Big(\frac12+\frac{|\hat{\ubar\chi}_s|^2}{\tr\ubar\chi_s^2}\Big)G(\ubar L,\ubar L)+2\frac{\hat{\ubar\alpha}_s\cdot\cdot\hat{\ubar\chi}_s}{\tr\ubar\chi_s}+\kappa D_s.$$
So for a contradiction, assuming $D_{s(p)}<0$ for some $p\in \Omega$, under an affine foliation we have:
$$\frac{d}{ds}\frac{1}{D_s}\geq 1\implies \frac{1}{D_s}-\frac{1}{D_{s(p)}}\geq s-s(p).$$
Analogous to Lemma \ref{l11}, we conclude from $S_+ = \infty$ that $D\to-\infty$ as $s\to s_\star^-$ for some $s_\star^-$ contradicting smoothness. From our equation above, convexity of $\Omega$ follows exactly the same argument as positivity of null expansion in Lemma \ref{l11}.
\end{proof}
A direct consequence of Lemma \ref{l12} and (\ref{a1}) is the fact that an affine foliation $\{\Sigma_s\}$ satisfies $(\hat{\ubar\chi}_s)_{ij}=O(1)$ under a basis extension. We now turn to (\ref{a4}) to improve our control of $(\gamma_s,\,a_s,\,\hat{\ubar\chi}_s)$:
\begin{lemma}\label{l13}
	Along an affine foliation $\{\Sigma_s\}$, we define $\bar{\gamma}_s = s^{-2}\gamma_s$. Then, the system $(\bar{\gamma}_s,a_s , \hat{\ubar\chi}_s)$, converges to a limit $(\gamma_\infty, \theta, \hat\Xi)$ of index $(1,\beta)$-H\"{o}lder space tensors on $\mathbb{S}^2$. 
\end{lemma} 
\begin{proof}
We have the system of equations:

$$\frac{d}{ds}\begin{pmatrix}
	\bar\gamma_s\\a_s\\\hat{\ubar\chi}_s
\end{pmatrix}	
=\begin{pmatrix}
\frac{a_s}{s^2}\bar\gamma_s+\frac{2}{s^2}\hat{\ubar\chi}_s\\
-\frac{1}{2}\frac{a_s^2}{s^2}-s^2(|\hat{\ubar\chi}_s|^2+G(\ubar L,\ubar L))\\-\hat{\ubar\alpha}_s+s^2|\hat{\ubar\chi}_s|^2\bar\gamma_s	\end{pmatrix}
	.$$
From (\ref{a1},\ref{a3}), and Lemma \ref{l12}, denoting $\vec{X}:=(\bar\gamma, a,\hat{\ubar\chi})$, we have $\vec{X} = O(1)$, and therefore we observe $|\frac{d}{ds}{\vec{X}}| \leq \frac{C_1}{s^{1+\delta}}$ from (\ref{a4}) and Lemma \ref{l12} (here $|\vec{X}|$ refers to a local co-ordinate Euclidean norm). Consequently:
$$|\vec{X}(s_n)-\vec{X}(s_m)|\leq C_1\int_{s_m}^{s_n}\frac{dt}{t^{1+\delta}}.$$
From this Cauchy sequence we conclude that $\vec{X}(s)$ uniformly converges to a limit, say 
$$\displaystyle{\lim_{s\to\infty}\vec{X}(s) = (\gamma_\infty,\theta, \hat{\Xi})}.$$ Since $(\pounds_{Y}(\bar{\gamma}_s)^{-1})^{ij} = -\bar{\gamma}_s^{ik}\pounds_{Y}(\bar{\gamma}_s)_{kl}\bar{\gamma}_s^{lj}$ for any $Y\in E(\Sigma_0)$, established bounds on $\vec{X}$ yield: 
$$\frac{d}{ds}\pounds_{Y}\vec{X} = \pounds_{Y}{\frac{d}{ds}{\vec{X}}} = A_1\pounds_{Y}\vec{X}+A_2$$
having used our system of equations to obtain the final equality so that the matrices $A_1(\bar\gamma_s,a_s,\hat{\ubar\chi}_s,\ubar\alpha_s)$, $A_2(\bar{\gamma}_s,\pounds_Y\ubar\alpha_s)$ satisfy $s^{1+\delta}A_i=O(1)$. Standard ODE differential inquality results enforce that $|\pounds_Y\vec{X}|\leq C_2$ for some constant $C_2$. Therefore, $|\frac{d}{ds}\pounds_Y\vec{X}|\leq\frac{C_3}{s^{1+\delta}}$, and we again observe uniform convergence for $\pounds_Y\vec{X}$ similarly as for $\vec{X}$. For the H\"{o}lder bound we observe:
\begin{align*}
	\frac{d}{ds}\frac{\pounds_{Y}\vec{X}(p)-\pounds_{Y}\vec{X}(q)}{d(p,q)^\beta} &= A_3(p) \frac{\pounds_{Y}\vec{X}(p)-\pounds_{Y}\vec{X}(q)}{d(p,q)^\beta}+\Big(\frac{A_1(p)-A_1(q)}{d(p,q)^\beta}\pounds_{Y}\vec{X}(q)+\frac{A_2(p)-A_2(q)}{d(p,q)^\beta}\Big)\\
	&= A_1(p) \frac{\pounds_{Y}\vec{X}(p)-\pounds_{Y}\vec{X}(q)}{d(p,q)^\beta}+B_2(p,q).
\end{align*}
Here the H\"{o}lder bound in (\ref{a4}) comes into play. Combined with our control of $\pounds_{Y}\vec{X}$ we have $s^{1+\delta}\sup_{p\neq q}|B_2(p,q)| = O(1)$. We therefore observe from standard ODE differential inequalities that:
$$\Big|\frac{\pounds_{Y}\vec{X}(p)-\pounds_{Y}\vec{X}(q)}{d(p,q)^\beta}(s)\Big|\leq C_4,$$
where $C_4$ is independent of $p,q$. So taking the limit $s\to\infty$ we have our result.
\end{proof}
We will also need the following result:
\begin{proposition}\label{p9}(\cite{christodoulou2014global}, Lemma 2.3.2)\\
Consider $\Sigma:=(\mathbb{S}^2,\gamma)$, $\gamma\in C^{3,\beta}(\text{Sym}(T^\star\mathbb{S}^2\otimes T^\star\mathbb{S}^2))$. There exists a conformal transformation to a metric $e^{-2v}\gamma$, $v\in C^{3,\beta}(\mathbb{S}^2)$, such that $\mathcal{K}_{e^{-2v}\gamma} = 1$. Moreover, the conformal factor $e^v$ can be chosen such that the quantities:
$$C_m:=\inf_{\mathbb{S}^2}\mathfrak{r}e^{-v},\,\,\sup_{\mathbb{S}^2}\mathfrak{r}e^{-v},\,\,\sup_{\mathbb{S}^2}(e^{2v}|\nabla e^{-v}|),\,\,\Big(\int_{\mathbb{S}^2}e^{(2-3p)v}|\nabla^2e^{-v}|^pdA\Big)^{\frac{1}{p}}$$
depend only on:
$$k_M:=\sup_{\mathbb{S}^2}\mathfrak{r}^2|\mathcal{K}|,\,\,\frac{\text{diam}(\Sigma)}{\mathfrak{r}},$$
whereby $4\pi\mathfrak{r}^2:=|\Sigma|$.	Moreover, if $\mathcal{K}>0$, denoting $k_m:=\inf_{\mathbb{S}^2}(\mathfrak{r}^2\mathcal{K})$, then these quantities depend only on $\{k^{-1}_m,k_M\}$.
\end{proposition}
In-fact, the proof of this result follows from a careful construction of a bounded function $v\in L^\infty(\mathbb{S}^2)$. Then, observing that $v$ satisfies the elliptic equation:
$$\Delta v = \mathcal{K}-e^{-2v}$$
in the sense of distributions, the right-hand side of the above equation belongs to $L^p(\mathbb{S}^2)$ for every $p\in(0,\infty)$. Subsequently, elliptic regularity gives $v\in C^{3,\beta}(\mathbb{S}^2)$, since $\mathcal{K}\in C^{1,\beta}(\mathbb{S}^2)$. We have enough regularity to construct a differentiable exponential map associated to $\gamma\in C^{3,\beta}(\text{Sym}(T^\star\mathbb{S}^2\otimes T^\star\mathbb{S}^2))$ and the famous Killing-Hopf Theorem ensures that we can construct an isometry $\Phi_\gamma:\mathbb{S}^2\to\mathbb{S}^2$ such that $\Phi_\gamma^\star(\mathring\gamma) = e^{-2v}\gamma$, with $\mathring\gamma$ the standard round metric on $\mathbb{S}^2$, and say $\Phi_\gamma(p_N) = p_N$ whereby $p_N\in\mathbb{S}^2$ represents the north pole. Since the function $u:=v\circ\Phi_\gamma^{-1}$ satisfies:
$$\mathring\Delta u = 1-(\mathcal{K}\circ\Phi^{-1}_\gamma)e^{2u},$$
we can apply Calderon-Zygmund $L^p$ estimates on the standard 2-sphere to duduce:
\begin{equation}\label{e25}
\int_{\mathbb{S}^2}|\ring\nabla^2u|^p+|\ring\nabla u|^p+|u-\bar{u}|^p \ring{dA}\leq C(p)\int_{\mathbb{S}^2}|1-\mathcal{K}e^{2u}|^p\ring{dA}\leq C(k_M,C_m),
\end{equation}
  whereby $\bar{u}:=\frac{1}{4\pi}\int u \ring{dA}$. In other-words, the solution $u-\bar{u}$ belongs to the Sobolev space $W^{2,p}(\mathbb{S}^2)$ consisting of functions with finite $L^p(\mathbb{S}^2)$ norm for up to two derivatives in the distributional sense. From the Sobolev embedding theorem, for sufficiently large $p\in\mathbb{R}$, we observe $C^{1,\beta}(\mathbb{S}^2)\hookrightarrow W^{2,p}(\mathbb{S}^2)$, meaning:
  $$||u-\bar{u}||^p_{1,\beta}\leq C_1(p) \int_{\mathbb{S}^2}|\ring\nabla^2u|^p+|\ring\nabla u|^p+|u-\bar{u}|^p \ring{dA}\leq C_2(p)\int_{\mathbb{S}^2}|1-\mathcal{K}e^{2u}|^p\ring{dA}\leq \tilde{C}(k_M,C_m). $$
If we now consider our foliation $\{\Sigma_s\}$ with smooth path $s\to\gamma_s$, within $\Omega$, this foliation gives rise to a family of smooth functions $\{u_s\}$ combining Proposition \ref{p9} and elliptic regularity. From (\ref{a2}), we therefore conclude that:
$$||u_s-\bar{u}_s||_{1,\beta} = O(1).$$
We also observe from Jensen's inequality:
  $$\frac{e^{2\bar{u}_s}}{\mathfrak{r}^2_s}\leq \frac{1}{4\pi}\int_{\mathbb{S}^2}\frac{e^{2u_s}}{\mathfrak{r}_s^2}\ring{dA}\leq 1,$$
  whereby $4\pi\mathfrak{r}^2_s:=|\Sigma_s|$. In-fact, we observe that $c\mathfrak{r}_s^2\leq e^{2\bar{u}_s}\leq \mathfrak{r}_s^2$ for some $c>0$. We show this by contradiction. If not, then we observe a sequence $\{s_n\}$ such that $\frac{e^{2\bar{u}_{s_n}}}{\mathfrak{r}^2_{s_n}} \to 0$. By the famous Kondrachov embedding theorem we have $W^{2,p}(\mathbb{S}^2)\subset\subset W^{1,q}(\mathbb{S}^2)$ for any $q\in[1,\infty)$. From (\ref{a2}), and our bound in (\ref{e25}) we may re-index to a sub-sequence such that $u_{s_n}-\bar{u}_{s_n}\to u_\star\in W^{1,q}(\mathbb{S}^2)$. Since terms in the sequence $\{u_{s_n}-\bar{u}_{s_n}\}$ solve the equation:
  $$\mathring\Delta (u_{s_{n}}-\bar{u}_{s_n}) = 1-\frac{e^{2\bar{u}_{s_{n}}}}{\mathfrak{r}^2_{s_{n}}}(\mathfrak{r}^2_{s_{n}}\mathcal{K}_{s_{n}})e^{2(u_{s_{n}}-\bar{u}_{s_{n}})},$$
it follows that $u_\star$ solves the equation $\mathring\Delta u_\star = 1$ in the distributional sense. From elliptic regularity $u_\star$ is consequently a strong solution and smooth. This contradicts the divergence theorem since $0=\int_{\mathbb{S}^2}\mathring\Delta u_\star \ring{dA} = \int_{\mathbb{S}^2}\ring{dA} = 4\pi.$\\
\indent  From a calculation analogous to (\ref{e19}) (with $\ubar\chi\equiv 0$), we conclude in a local coordinate neighborhood:
\begin{align}
(\bar{\gamma}_s)_{ij}&=\Big(\frac{e^{2\bar{u}_s}}{s^2}\Big)(e^{2(u_s-\bar{u}_s)}\mathring\gamma)_{kl}(\Phi_s(x))(\Phi_s)^k_{,i}(\Phi_s)^l_{,j}\label{E1},\\
	(\Phi_s)^k_{,ij} &= {^{\bar{\gamma}_s}}\Gamma^n_{ij}(\Phi_s)^k_{,n}-({^{e^{2(u_s-\bar{u}_s)}\mathring\gamma}}\Gamma^k_{nm}\circ\Phi_s)(\Phi_s)^n_{,i}(\Phi_s)^m_{,j}\label{E2},
\end{align}
whereby $\Phi_s:=\Phi_{\gamma_s}$, and $\bar{\gamma}_s = \frac{1}{s^2}\gamma_s$. From Lemma \ref{l13} we observe that $\frac{1}{s^2}\mathfrak{r}_s^2\to \frac{1}{4\pi}|\Sigma_\infty|$ as $s\to\infty$, it follows therefore that (\ref{E1}) provides a bound $||d\Phi_s||_{0} = O(1)$. Consequently, (\ref{E2}) provides $||d\Phi_s||_{1,\beta} = O(1)$. In summary:
\begin{lemma}\label{l14}
  Given our foliation $\{\Sigma_s\}$ of $\Omega$, we observe a smooth pairing $(u_{s},\Phi_s)$ such that $\gamma_s = \Phi_s^\star(e^{2u_s}\mathring\gamma)$, whereby $\mathring\gamma$ is the standard round sphere on $\mathbb{S}^2$. Moreover,
  $$\frac{e^{\bar{u}_{s}}}{s},||u_s-\bar{u}_s||_{1,\beta},||d\Phi_s||_{1,\beta} = O(1),$$
whereby $\bar{u}_s = \frac{1}{4\pi}\int_{\mathbb{S}^2}u_s\ring{dA}$. Also, $c_1\leq \frac{e^{\bar{u}_{s}}}{s} \leq c_2$, for some $c_1,c_2>0$. 
  \end{lemma} 
\subsection{Closedness}
In this section we wish to show that any metric, $\gamma\in C^{\mathfrak{n},\beta}(\text{Sym}(T^\star\mathbb{S}^2\otimes T^\star\mathbb{S}^2))$, $\mathfrak{n}\geq 3$, can be isometrically embedded in $\Omega$ up to a rescaling. From Proposition \ref{p9}, we follow a similar argument as in Lemma \ref{l14} to obtain $u_\gamma\in C^{\mathfrak{n},\beta}(\mathbb{S}^2)$, $\Phi_\gamma\in \mathcal{D}^{\mathfrak{n}+1,\beta}(\mathbb{S}^2)$ (as a result of (\ref{E2})), such that $\gamma = \Phi_\gamma^\star(e^{2u_\gamma}\mathring\gamma)$. We now take the path of metrics: 
$$t\to\gamma(s,t):=\Phi_s^\star(e^{2(t(u_\gamma+\bar{u}_s-u_s)+u_s)}\mathring\gamma)\in C^{\mathfrak{n},\beta}(\text{Sym}(T^\star\mathbb{S}^2\otimes T^\star\mathbb{S}^2)),\,\,\,t\in[0,1],$$
and we observe that this path takes us from $\gamma_s$, at $t=0$, to $(\Phi_\gamma^{-1}\circ\Phi_s)^\star(e^{2\bar{u}_s}\gamma)$, at $t=1$. So if we're successful in isometrically embedding this latter metric, we will have isometrically embedded the metric $e^{2\bar{u}_s}\gamma$ by pre-composing the resulting isometric embedding with the appropriate diffeomorphism. We also observe the Fr\'{e}chet derivative of our path $\gamma(s,t)$, is given by:
$$\frac{d}{dt}\gamma(s,t) = \Big(2(u_\gamma+\bar{u}_s-u_s)\circ\Phi_s\Big)\gamma(s,t).$$
Starting from any leaf $\Sigma_s\subset\Omega$, $s\geq s_0$, Theorem \ref{t3} admits an isometric embedding of the path $t\to\gamma(s,t)$ for some interval $t\in[0,\epsilon(s)]$, whereby $0<\epsilon(s)\leq 1$. Our goal is to show that sufficiently large $s$ accommodates an interval with $\epsilon(s)=1$.
\begin{proposition}\label{p10}
	For fixed $s$, consider the path $t\to \gamma(s,t)\in C^{\mathfrak{n},\beta}(\text{Sym}(T^\star\mathbb{S}^2\otimes T^\star\mathbb{S}^2))$ given by: 
	$$\gamma(s,t):=\Phi_s^\star(e^{2(t(u_\gamma+\bar{u}_s-u_s)+u_s)}\mathring\gamma),$$ 
with associated isometric embedding $\vec{r}(s,t)\in C^{\mathfrak{n},\beta}(\mathbb{S}^2,\mathbb{R}^3)$, $t\in[0,\epsilon(s)]$. Then, denoting $\varphi(s,t):=|\vec{r}(s,t)|$, $u(s,t):= [t(u_\gamma+\bar{u}_s-u_s)+u_s]\circ\Phi_s$ we observe, for a fixed constant $c_0>0$, all $t\in[0,\epsilon(s)]$ giving rise to the estimate:
	$$||\varphi(s,t)e^{-u(s,t)}-se^{-u_s\circ\Phi_s}||_{1,\beta}\leq c_0,$$
also admit a constant, $C_\star(c_0)$, such that: 
$$|| \varphi(s,t)e^{-u(s,t)}-se^{-u_s\circ\Phi_s}||_{1,\beta}\leq C_\star(c_0) e^{-\bar{u}_s}.$$
\end{proposition}
\begin{proof}
We start similarly as in Proposition \ref{p8}, in a local coordinate neighborhood:
$$g_{kl}(\varphi(s,t)(x),\Psi(s,t)(x))\Psi(s,t)^k_{,i}\Psi(s,t)^l_{,j} = \gamma(s,t)_{ij}.$$
From Lemma \ref{l13} we observe that $\bar{\gamma}_s$ converges pointwise, moreover, $\frac{\det{\bar{\gamma}_s}_{ij}}{\det\mathring\gamma_{ij}}$ converges away from zero. It follows that we can find constants $C_1,C_2>0$ such that:
	$$\frac{1}{C_1}\ring\gamma\leq \frac{1}{s^2}\gamma_s\leq C_1\ring\gamma,\,\,\,\frac{1}{C_2}\ring\gamma\leq\frac{1}{s}\ubar\chi_s\leq C_2\ring\gamma.$$
	From our expression above, we therefore have:
	$$\big(\varphi e^{-u(s,t)}\big)^2\big(\frac{g}{\varphi^2}\big)_{kl} \Psi(s,t)^k_{,i}\Psi(s,t)^l_{,j} = \Phi_s^\star(\mathring\gamma)_{ij}.$$
Using Lemma \ref{l14} we also find a constant $C_3>0$ such that $\frac{1}{C_3}\ring\gamma\leq \Phi_s^\star(\mathring\gamma)\leq C_3\ring\gamma$, giving $||d\Psi(s,t)||_0\leq C_4(c_0)$. If we denote by $\tilde\gamma(s,t):=e^{-2\bar{u}_s}\gamma(s,t)$, then from (\ref{e19}):
\begin{align*}
	\Psi^k_{,ij}&= {^{\tilde\gamma}}\Gamma^n_{ij}\Psi^k_{,n}-{^{\bar{\gamma}_s}}\Gamma^k_{nm}\circ(\varphi(x),\Psi(x))\Psi^n_{,i}\Psi^m_{,j} +\big(\varphi e^{-u(s,t)}\big)\big(\frac{\ubar\chi_{nm}}{\varphi}\big)\big(\varphi_{,r} e^{-u(s,t)}\big)\Phi_s^\star(\mathring\gamma)^{lr}\Psi^k_{,l}\Psi^n_{,i}\Psi^m_{,j}\\
	&\qquad-\big(\frac{e^{u(s,t)}}{\varphi}\big)\big(\frac{\ubar\chi _{nm}}{\varphi}\big) (\varphi^2g ^{kn})\big(\varphi_{,i} e^{-u(s,t)}\big)\Psi^m_{,j}-\big(\frac{e^{u(s,t)}}{\varphi}\big)\big(\frac{\ubar\chi_{nm}}{\varphi}\big) (\varphi^2g ^{kn})\big(\varphi_{,j}e^{-u(s,t)}\big)\Psi^m_{,i}.
\end{align*}
Within a local coordinate neighborhood we see $\varphi_ie^{-u(s,t)} = (\varphi e^{-u(s,t)})_i+(\varphi e^{-u(s,t)})(u(s,t) - \bar{u}_s)_i$, so from Lemmata \ref{l13}, \ref{l14} and our hypothesis it follows that $||d\Psi(s,t)||_{1,\beta}\leq C_5(c_0)$. We also see that $\tilde{h}:=\frac{1}{\varphi} h = \frac{2}{\varphi}\Psi^\star(\ubar\chi)$ satisfies $||\tilde h||_{1,\beta}\leq C_6(c_0)$. We now observe $L_h(\tau) = L_{\tilde h}(\tau)$, which we transform into a second order elliptic operator using $\tilde\gamma(s,t)$:
$$\tilde{\mathcal{L}}_{\tilde h}(\tau) := \tilde\nabla\cdot L_{\tilde h}(\tau) - d\tr_{\tilde\gamma}\tilde L_{\tilde h}(\tau) = e^{2\bar{u}_s}\mathcal{L}_h(\tau).$$
Since our path of isometric embeddings is governed by the elliptic equation $\tilde{\mathcal{L}}_{\tilde h}(\tau) = \tilde\nabla\cdot\tilde q- d\tr_{\tilde\gamma}\tilde q$, whereby $\tilde{q}: = \dot{\gamma}-\frac12(\tr_{\tilde{h}}\dot\gamma)\tilde{h}$, we wish to estimate $||\tilde q||_{1,\beta}$ in order to use Schauder estimates to gain control of $||\tau||_{2,\beta}$. In order to do so, we start by observing that $\tilde q = 2\dot{u}(s)(\gamma - \frac12\tr_{\tilde h}\gamma\tilde h)$, whereby $\dot{u}(s):=(u_\gamma+\bar{u}_s-u_s)\circ\Phi_s$. For convenience, we temporarily denote $\Psi^\star(\ubar\chi)_{ij}$ simply by $\ubar\chi_{ij}$, and $\gamma(s,t)_{ij}$ by $\gamma_{ij}$, to see:
\begin{align*}
	\gamma_{ij} - \frac12(\tr_{\ubar\chi}\gamma)\ubar\chi_{ij}&=\gamma_{ij} - \frac12\Big(\frac{\frac12(\tr\ubar\chi)\gamma^{mn}-\hat{\ubar\chi}^{mn}}{\frac14(\tr\ubar\chi)^2-\frac12|\hat{\ubar\chi}|^2}{\gamma}_{mn}\Big)\ubar\chi_{ij}\\
	&=\gamma_{ij}-\frac{\frac12\tr\ubar\chi}{\frac14(\tr\ubar\chi)^2-\frac12|\hat{\ubar\chi}|^2}(\frac12\tr\ubar\chi\gamma_{ij}+\hat{\ubar\chi}_{ij})\\
	&=-\frac12\frac{|\hat{\ubar\chi}|^2\gamma_{ij}+\tr\ubar\chi\hat{\ubar\chi}_{ij}}{\frac14(\tr\ubar\chi)^2-\frac12|\hat{\ubar\chi}|^2}\\
	&=-\frac{2}{\tr\ubar\chi}\Big(1+ \frac{2e^{-2u(s,t)}(\varphi^4|\hat{\ubar\chi}|^2)}{(ae^{-u(s,t)}+2\varphi e^{-u(s,t)})^2-2 e^{-2u(s,t)}(\varphi^4|\hat{\ubar\chi}|^2)}\Big)\hat{\ubar\chi}_{ij}\\
	&\qquad\qquad\qquad\qquad-\frac{2(\varphi^4|\hat{\ubar\chi}|^2)}{(ae^{-u(s,t)}+2\varphi e^{-u(s,t)})^2-2 e^{-2u(s,t)}(\varphi^4|\hat{\ubar\chi}|^2)}\Phi_s^\star(\mathring\gamma)_{ij}.
	\end{align*}

We now use Lemma \ref{l13} and the second part of (\ref{a4}) to observe that both the functions $(s^4|\hat{\ubar\chi}_s|^2)\circ\big(\varphi(s,t)(x),\Psi(s,t)(x)\big)$, and $a_s\circ\big(\varphi(s,t)(x),\Psi(s,t)(x)\big)$ have $C^{1,\beta}(\mathbb{S}^2)$ bounds dependent only on $c_0$. Moreover, $\Psi^\star(\hat{\ubar\chi})\in C^{1,\beta}(\text{Sym}(T^\star\mathbb{S}^2\otimes T^\star\mathbb{S}^2))$ also exhibits a norm bound dependent only on $c_0$. We also have:
$$\frac{2}{\tr\ubar\chi}\circ(\varphi(x),\Psi(x)) = \frac{s}{1+\frac{a}{s}}\circ(\varphi(x),\Psi(x)) = e^{\bar{u}_s}\Big(\frac{e^{-u(s,t)}\varphi}{1+\frac{a e^{-u(s,t)}}{\varphi e^{-u(s,t)}}} e^{u(s,t)-\bar{u}_s}\Big),$$
and from Lemma \ref{l14} that $||\dot{u}||_{1,\beta}\leq C_7$. So combining all of these facts we conclude that $||\tilde{q}||_{1,\beta}\leq e^{\bar{u}_s}C_8(c_0)+C_9(c_0)$. Solving the equation $\tilde{\mathcal{L}}_{\tilde h}(\tau) = \tilde\nabla\cdot\tilde q- d(\tr_{\tilde\gamma}\tilde q)$ therefore yields via elliptic estimates that, $||\tau||_{2,\beta}\leq e^{\bar{u}_s}C_{10}(c_0)+C_{11}(c_0)$. Denoting $\tilde\tau:=e^{-\bar{u}_s}\tau$, we also observe:
$$\phi = \frac{\tr_{\gamma(s,t)}\dot{\gamma}(s,t) - 2\nabla\cdot\tau}{\tr_{\gamma(s,t)}h}=\frac{2\dot{u}}{\tr\ubar\chi}(1-e^{-2\bar{u}_s}\tilde\nabla\cdot\tau)=\dot{u}\varphi\frac{1-e^{-\bar{u}_s}\tilde\nabla\cdot\tilde\tau}{1+\frac{a e^{-u(s,t)}}{\varphi e^{-u(s,t)}}}.$$
Consequently, we observe from Lemma \ref{l10}:
\begin{align*}
\partial_t(e^{-u(s,t)}\varphi-se^{-u_s\circ\Phi_s}) &= \partial_t(e^{-u(s,t)}\varphi)\\
& =\big(-\dot{u}e^{-u(s,t)}\varphi+e^{-u(s,t)}(\tau^\#\varphi+\phi)\big)	\\
&=\tau^\#(e^{-u(s,t)}\varphi-se^{-u_s\circ\Phi_s})\\
&\qquad+e^{-u(s,t)}(\phi-\dot{u}\varphi\big)+(\varphi e^{-u(s,t)})\tau^\#(u(s,t)-\bar{u}_s)\\
&\qquad\qquad-(se^{-u_s\circ\Phi_s})\tau^\#(u_s\circ\Phi_s-\bar{u}_s).
\end{align*}
Denoting: 
\begin{align*}
	f&:= e^{-u(s,t)}\varphi-se^{-u_s\circ\Phi_s}\\
	\tilde\phi&:= e^{-u(s,t)}(\phi-\dot{u}\varphi\big)+(\varphi e^{-u(s,t)})\tau^\#(u(s,t)-\bar{u}_s)-(se^{-u_s\circ\Phi_s})\tau^\#(u_s\circ\Phi_s-\bar{u}_s),
\end{align*}
we observe the equation $\partial_tf = \tau^\#(f)+\tilde\phi$, for which Lemma \ref{l9} yields the unique solution:
$$f(t,x) = \int_0^t\tilde\phi(\lambda,\psi(\lambda,\psi^{-1}(t,x)))d\lambda,$$
whereby the flow $\psi$, and its inverse flow $\psi^{-1}$ are associated to to the vector-field $\tau^\#$, specifically, $\psi^{-1} = \Psi$. In a local co-ordinate neighborhood we also observe, $(\tau^\#)^i=e^{-\bar{u}_s}\tilde\gamma^{ij}\tilde\tau_j$ so that our estimate on $||\tau||_{2,\beta}$ yields $||\tilde\phi||_{1,\beta} = C_{12}(c_0)e^{-\bar{u}_s}$. Combined with our estimates on $d\Psi(s,t)$, we see:
$$||f(t,x)||_{1,\beta}\leq \int_0^t|| \tilde\phi(\lambda,\psi(\lambda,\psi^{-1}(t,x))) ||_{1,\beta} d\lambda\leq tC_{13}(c_0)e^{-\bar{u}_s}\leq C_{13}(c_0)e^{-\bar{u}_s}.$$
\end{proof}
\begin{corollary}\label{c6}
For $s$ sufficiently large, we can isometrically embed the path $t\to\gamma(s,t)$, for $t\in[0,1]$, to completion within $\Omega$.
\end{corollary}
\begin{proof}
For any $c_0>0$ as in Proposition \ref{p10}, we have:
$$||\varphi(s,t)e^{\bar{u}_s-u(s,t)} - se^{\bar{u}_s-u_s\circ\Phi_s}||_{1,\beta}=||\big(\varphi(s,t)-se^{t(u_\gamma+\bar{u}_s-u_s)\circ\Phi_s}\big)e^{\bar{u}_s-u_s\circ\Phi_s}||_{1,\beta}\leq C_\star(c_0).$$
It is evident from Lemma \ref{l14} that we obtain a bound $C_1(c_0)$ such that: 
$$\sup_t||d\varphi(s,t)||_0,\,\,\sup_t(\sup_{\mathbb{S}^2}\varphi(s,t))\leq sC_1(c_0).$$
Lemma \ref{l14} also provides a constant $C_2>0$ such that $\sup_{\mathbb{S}^2}|\bar{u}_s-u_s\circ\Phi_s|\leq C_2$. Therefore, for any $se^{-C_2+\inf(u_\gamma)}-C_\star e^{C_2}:=C_-> S_-$ we obtain $\inf_t(\inf_{\mathbb{S}^2}\varphi(s,t))\geq C_-$ satisfying the desired conditions to apply Proposition \ref{p8}. Since these bounds continue to hold independently of $t$, if $\epsilon(s)<1$ in Proposition \ref{p11}, we may extend the path $t\to\vec{r}(s,t)$ throughout the interval $t\in[0,1]$.
\end{proof}
Consequently:
\begin{theorem}\label{t5}
	Up-to a sufficiently large scale, any metric $\gamma\in C^{\mathfrak{n},\beta}(\text{Sym}(T^\star\mathbb{S}^2\otimes T^\star\mathbb{S}^2))$, $\mathfrak{n}\geq 3$, can be isometrically embedded within $\Omega$.
\end{theorem}
\subsection{The Foliation}
From Theorem \ref{t5}, we conclude with the isometric embedding of the metric $e^{2\bar{u}_s}\gamma\in C^{\mathfrak{n},\beta}(T^\star\mathbb{S}^2\otimes T^\star\mathbb{S}^2))$, $\mathfrak{n}\geq 3$, for sufficiently large $s$, which we will denote by $\vec{r}_s\in C^{3,\beta}(\mathbb{S}^2,\mathbb{R}^3)$. Moreover, from Proposition \ref{p9} (after the necessary coordinate change) we observe:
$$\big|\big||\vec{r}_s|e^{-\bar{u}_s}-(se^{-\bar{u}_s})e^{(u_\gamma+\bar{u}_s-u_s)\circ\Phi_\gamma}\big|\big|_{1,\beta}\leq \tilde{C}_\star e^{-\bar{u}_s}\implies \big|\big||\vec{r}_s|e^{-\bar{u}_s}\big|\big|_{1,\beta}\leq \tilde{\tilde{C}}_\star.$$
In this section we wish to show, for sufficiently large $s$, we are able to isometrically embed the path of metrics $t\to  e^{t+2\bar{u}_s}\gamma$, $t\in [0,\infty)$, for the given $\gamma\in C^{\mathfrak{n},\beta}(\text{Sym}(T^\star\mathbb{S}^2\otimes T^\star\mathbb{S}^2))$. We also wish to show that this induces a foliation upon $\Omega$ in a neighborhood of infinity. 
\begin{proposition}\label{p11}
For sufficiently large $s$, consider the isometric embedding $\vec{r}_s$ associated to the metric $e^{2\bar{u}_s}\gamma\in C^{\mathfrak{n},\beta}(T^\star\mathbb{S}^2\otimes T^\star\mathbb{S}^2))$, $\mathfrak{n}\geq 3$. For the path $t\to\gamma(s,t):=e^{t+2\bar{u}_s}\gamma$, we take the associated path of isometric embeddings $\vec{r}(s,t)\in C^{\mathfrak{n},\beta}(\mathbb{S}^2,\mathbb{R}^3)$, whereby $\vec{r}(s,0) = \vec{r}_s$, $t\in [0,T(s)]$. Then, denoting $\varphi(s,t):=|\vec{r}(s,t)|$ we observe, for a fixed $c_0>0$, all $t\in [0,T(s)]$ giving rise to the estimate:
$$\big|\big|\varphi(s,t) e^{-\frac{t}{2}-\bar{u}_s} - |\vec{r}_s|e^{-\bar{u}_s}\big|\big|_{1,\beta}\leq c_0,$$
also admit a constant, $C_{\star\star}(c_0)$, such that:
$$||\varphi(s,t) e^{-\frac{t}{2}-\bar{u}_s}-|\vec{r}_s|e^{-\bar{u}_s}||_{1,\beta}\leq C_{\star\star}(c_0)e^{-\bar{u}_s}.$$
\end{proposition}
\begin{proof}
Our argument starts identically as in Proposition \ref{p10} except with $u(s,t),\bar{u}_s$ both replaced by $\frac{t}{2}+\bar{u}_s$, and $\Phi_s^\star(\mathring\gamma),\tilde\gamma$ both replaced by the metric $\gamma$. Consequently, we also observe $\dot{u} = \frac12$, and an estimate $||\tau||_{2,\beta}\leq C_1(c_0)e^{\frac{t}{2}+\bar{u}_s}+C_2(c_0)$. Denoting:
\begin{align*}
f&:=e^{-\frac{t}{2}-\bar{u}_s}\varphi-|\vec{r}_s|e^{-\bar{u}_s}\\
\tilde\phi&:=	e^{-\frac{t}{2}-\bar{u}_s}(\phi-\frac{1}{2}\varphi\big)+e^{-\bar{u}_s}\tau^\#(|\vec{r}_s|),
\end{align*}
we observe the equation $\partial_tf = \tau^\#(f)+\tilde\phi$, whereby $(\tau^\#)^i = e^{-\frac{t}{2}-\bar{u}_s}\gamma^{ij}\tilde\tau_j$, $\tilde\tau:= e^{-\frac{t}{2}-\bar{u}_s} \tau$. It follows that $||\tilde\phi||_{1,\beta}\leq C_3(c_0)e^{-\frac{t}{2}-\bar{u}_s}$ and therefore, with $\psi^{-1} = \Psi$ we observe from Lemma \ref{l9}:
$$||f(t,x)||_{1,\beta}\leq \int_0^t|| \tilde\phi(\lambda,\psi(\lambda,\psi^{-1}(t,x))) ||_{1,\beta}d \lambda\leq C_4(c_0)e^{-\bar{u}_s}\int_0^te^{-\frac{\lambda}{2}}d\lambda = C_5(c_0)e^{-\bar{u}_s}.$$
\end{proof}
We will be able to use Proposition \ref{p11} to show the existence of a foliation of $\Omega$ in a neighborhood of infinity. Moreover, we recall that a geodesic foliation, $\{\Sigma_\lambda\}$, of $\Omega$ differs from $\{\Sigma_s\}$ via:
$$s|_{\Sigma_\lambda} = (f_1\circ\pi)\lambda+f_2\circ\pi$$
for some functions $f_1,f_2$ on $\Sigma_{s_0}$. 
\begin{definition}\label{d5}
	We will say a family of cross-sections $\{\Sigma_\lambda\}$, $\lambda\geq \lambda_0$, of $\Omega$ form an asymptotically $C^{\mathfrak{n},\beta}$ geodesic foliation, whenever:
	$$s|_{\Sigma_\lambda} = (f_1\circ\pi) \lambda+f_2,$$
	whereby, identifying $\Sigma_{s_0}$ with $\mathbb{S}^2$, $f_1,f_2\circ(\pi|_{\Sigma_\lambda})^{-1}\in C^{\mathfrak{n},\beta}(\mathbb{S}^2)$, and $|| f_2\circ(\pi|_{\Sigma_\lambda})^{-1} ||_{\mathfrak{n},\beta}\leq C$, for some constant $C$ independent of $\lambda$.
\end{definition}

\begin{theorem}\label{t6}
	For $s$ sufficiently large and $\gamma\in C^{\mathfrak{n},\beta}(\text{Sym}(T^\star\mathbb{S}^2\otimes T^\star\mathbb{S}^2))$, $\mathfrak{n}\geq3$, we can isometrically embed the path $t\to \gamma(s,t):=e^{t+2\bar{u}_s}\gamma$ in $\Omega$, for $t\in[0,\infty)$. Moreover, this path of embeddings forms a foliation of $\Omega$ in a neighborhood of infinity that is an asymptotically geodesic $C^{1,\beta}$-foliation.
\end{theorem}
\begin{proof}
Here, for any $c_0>0$ as in Proposition \ref{p11}, we have:
$$||\varphi(s,t) - |\vec{r}_s|e^{\frac{t}{2}}||_{1,\beta}\leq C_{\star\star}(c_0)e^{\frac{t}{2}},\,\,\,|||\vec{r}_s|||_{1,\beta}\leq \tilde{\tilde{C}}_\star e^{\bar{u}_s}.$$
We conclude $\sup_t||d\varphi(s,t)||_0,\sup_t(\sup_{\mathbb{S}^2}\varphi(s,t))\leq (C_{\star\star}(c_0)+\tilde{\tilde{C}}_\star e^{\bar{u}_s})e^{\frac{t}{2}}$. In order to gain control of $\inf_t(\inf_{\mathbb{S}^2}\varphi(s,t))$ from below, we analyze the function $\omega(s,t) = \varphi(s,t)\circ \Psi^{-1}(s,t)$. \\
\indent For $(\tau,\phi)(s,t)$ as in Proposition \ref{p11}, from Lemma \ref{l8} we know $d\Psi^{-1}(s,t)(\partial_t) = d\psi(\partial_t) = -\tau(s,t)$, and from Lemma \ref{l10}, that $\partial_t\varphi(s,t) = \tau^\#(\varphi)+\phi$. It follows therefore that $\partial_t\omega = \partial_t\varphi\circ\Psi^{-1}+d\varphi(d\Psi^{-1}(\partial_t)) = \tau^\#(\varphi)\circ\Psi^{-1}+\phi\circ\Psi^{-1} - \tau^\#(\varphi)\circ\Psi^{-1} = \phi\circ\Psi^{-1}$. We also see, borrowing from Proposition \ref{p11}:
$$\phi\circ\Psi^{-1}  =\frac{1}{\tr\ubar\chi}\Big(1-\nabla\cdot\tau\Big)\circ\Psi^{-1} = \frac{\omega}{2}\Big(\frac{1}{1+\frac{a e^{-\frac{t}{2}-\bar{u}_s}}{\omega e^{-\frac{t}{2}-\bar{u}_s}}}\Big)\Big(1-e^{-\frac{t}{2}-\bar{u}_s}\tilde\nabla\cdot\tilde\tau\Big)\circ\Psi^{-1}.$$
We therefore have $\sup_{\mathbb{S}^2}|\partial_t\omega - \frac{\omega}{2}|\leq C_1(c_0)$. Since $\omega(s,0) = |\vec{r}_s|$, and:
$$\sup_{\mathbb{S}^2}\Big||\vec{r}_s|e^{-\bar{u}_s}-(se^{-\bar{u}_s})e^{(u_\gamma+\bar{u}_s-u_s)\circ\Phi_\gamma}\Big|\leq \tilde{C}_\star e^{-\bar{u}_s},$$
we observe from Lemma \ref{l14} that $\omega(s,0)$ grows at least linearly in $s$. Consequently, $\partial_t\omega(s,0)>0$ whenever $|\vec{r}_s|>2C_1.$ From the differential inequality, $\partial_t\omega\geq \frac{\omega}{2}-C_1$, we conclude $\partial_t\omega(s,t)>0$ whenever $\omega(s,0)>2C_1$. Giving $\inf_t(\inf_{\mathbb{S}^2}\varphi(s,t))>2C_1$. We may now apply Proposition \ref{p8} to extend the path $t\to\vec{r}(s,t)$ of Proposition \ref{p11} to $t\in[0,\infty)$. Moreover, 
$$\omega(s,t)\geq (|\vec{r}_s|-2C_1)e^{\frac{t}{2}}+2C_1,$$
so we foliate a neighborhood of infinity. Again from Proposition \ref{p11}:
$$\partial_t(e^{-\frac{t}{2}}\omega) = e^{-\frac{t}{2}}(\phi\circ\Psi^{-1}-\frac12\omega)=e^{-\frac{t}{2}}\phi'\circ\Psi^{-1},$$
whereby $||\phi'\circ\Psi^{-1}||_{1,\beta}\leq C_2.$ We then solve to obtain, for $t_n\leq t_m$ (suppressing dependence on $s$):
$$||e^{-\frac{t_m}{2}}\omega(t_m,x)-e^{-\frac{t_n}{2}}\omega(t_n,x)||_{1,\beta}\leq \int _{t_n}^{t_m}e^{-\frac{\lambda}{2}}||\phi'(\lambda,x)||_{1,\beta}d\lambda\leq 2C_2(e^{-\frac{t_n}{2}}-e^{-\frac{t_m}{2}}).$$
So from any Cauchy sequence, $\{t_n\}$, we observe a convergent sequence $\{e^{-\frac{t_n}{2}}\omega(t_n,x)\}\subset C^{1,\beta}(\mathbb{S}^2)$ with limit $\omega_\infty\in C^{1,\beta}(\mathbb{S}^2)$, independent of $\{t_n\}$. Moreover,
$$||e^{-\frac{t}{2}}\omega(t,x)-\omega_\infty||_{1,\beta}\leq2C_2e^{-\frac{t}{2}},$$
so for a change of parameter $t\to \mathfrak{t}=e^{\frac{t}{2}}$, we observe $\omega(\mathfrak{t},x) = \omega_\infty\mathfrak{t}+f(\mathfrak{t},x)$ whereby  the functions $\omega_\infty,f\circ(\pi|_{\Sigma_\mathfrak{t}})^{-1}\in C^{1,\beta}(\mathbb{S}^2)$, and $||f\circ(\pi|_{\Sigma_\mathfrak{t}})^{-1}||_{1,\beta}\leq 2C_2$.
\end{proof}

\section{Acknowledgements} 
This material is based upon work supported by the National Science Foundation under Award No. 1703184. The author would like to thank Mu-Tao Wang for introducing him to this problem and his continued support, as well as Siyuan Lu for helpful conversations related to this work.
\newpage
\section{Appendix}
We wish to present a proof of Proposition \ref{p4} without using Theorem \ref{t0}. This argument is an adaptation of the original argument by Li-Wang \cite{li2020}, Theorem 11.\\\\
For a given metric $\gamma$ on a 2-sphere, we return to the volume 2-form, $\varepsilon_\gamma$, and identity (\ref{i0}). Contracting this identity either with itself, or a traceless 2-tensor, $T$ we observe:
 \begin{align}
 \varepsilon_i\,^l\varepsilon_{lj}&=-\gamma_{ij}\label{o1}\\
 \varepsilon_{ik}T^{kl}\varepsilon_{lj}&=T_{ij}\label{o2}.
 \end{align}
 Any tensor indices raised using $h^{-1}$ we will denote with a $\bar{bar}$. We have for a $h$-traceless 2-tensor $a:$
 \begin{align*}
 \frac12\mathcal{P}(h,h)(h^{-1})^{ij}(h^{-1})^{kl}a_{jl}&=\frac12\mathcal{P}(h,h)\bar{a}^{ik}\\
 &=\frac12\mathcal{P}(h,h)\varepsilon_{h}^{im}a_{mn}\varepsilon_{h}^{nk}\\
 &=	\varepsilon^{im}a_{mn}\varepsilon^{nk}.
 \end{align*}
In dimension two, with the use of the volume form we can instead equivalently characterize the bundle: 
$$\ubar\Xi_h=\{a\in\Gamma(\text{Sym}(T^\star\mathbb{S}^2\otimes T^\star\mathbb{S}^2))|\varepsilon^{ij}\nabla_ia_{jk}=0\}.$$
This follows from (\ref{i0}) since 
$$\varepsilon^{ij}\varepsilon^{kl}\nabla_i a_{jk} = (\gamma^{ik}\gamma^{jl}-\gamma^{jk}\gamma^{il})\nabla_i a_{jk} = (\nabla\cdot a)^l-\nabla^l\tr_\gamma a.$$
  Finally, we recall the \textit{Hodge Star} isomorphism acting on 1-forms:
 $$(\star\tau)_i = \varepsilon^j\,_i\tau_j.$$
 \begin{lemma}[\cite{li2020}, Lemma 7]\label{la}
Given $a\in\ubar\Xi_h$, and $\tau\in\text{Ker}(L_h)$, then the 1-form
 	$$\eta = a(\star\tau)$$
 	is trivial. 
 \end{lemma}
\begin{proof}
We will work within a local coordinate neighborhood $(x^i,\mathcal{U})$, and start by showing that $\eta$ is closed, and therefore given the cohomology of $\mathbb{S}^2$, exact.
\begin{align*}
d\eta&=\partial_l(\varepsilon^{ij}\tau_ia_{jk})dx^l\wedge dx^k\\
&=\nabla_l(\varepsilon^{ij}\tau_ia_{jk}) dx^l\wedge dx^k\\
&=(\varepsilon^{ij}a_{jk}\nabla_l\tau_i+\varepsilon^{ij}\tau_i\nabla_la_{jk})\varepsilon^{lk}(\sqrt{\text{det}\gamma}dx^1\wedge dx^2).	\end{align*}
	Analyzing the terms in turn:
	\begin{enumerate}
	\item $(\varepsilon^{ij}a_{jk}\varepsilon^{lk})\nabla_{l}\tau_i=-\mathcal{P}(h,h)\bar{a}^{il}\nabla_l\tau_i=-\frac12\mathcal{P}(h,h)\bar{a}^{il}(\nabla_l\tau_i+\nabla_i\tau_l)=\tr_h(\nabla\tau) \mathcal{P}(h,h)\bar{a}^{il}h_{il}=0$,
	\item $\varepsilon^{ij}\tau_i(\varepsilon^{lk}\nabla_l a_{jk})=\varepsilon^{ij}\tau_i(\varepsilon^{lk}\nabla_l a_{kj})=0$.
	\end{enumerate}
Therefore, $\eta = df$ for some function $f$. As a result
\begin{align*}
\nabla^2_{ij}f&=\nabla_i(\varepsilon^{kl}\tau_{k} a_{lj})\\
&=\varepsilon^{kl}\nabla_i\tau_k a_{lj}+\varepsilon^{kl}\tau_k \nabla_i a_{lj}\\
(h^{-1})^{ij}\nabla^2_{ij}f&=(\varepsilon^{kl}(h^{-1})^{ij}a_{lj})\nabla_i\tau_k+\varepsilon^{kl}\tau_k((h^{-1})^{ij}\nabla_la_{ij})\\
&=\sqrt{\mathcal{P}(h,h)}(\bar{\varepsilon}^{kl}\bar{a}^i\,_l)\nabla_i\tau_k-\varepsilon^{kl}\tau_ka_{ij}\nabla_l(h^{-1})^{ij}.
\end{align*}
Combining identities (\ref{o1},\ref{o2}), we observe that $\bar{\varepsilon}^{kl}\bar{a}^i\,_l$ is a symmetric contravariant 2-tensor. As a result, $\bar{\varepsilon}^{kl}\bar{a}^i\,_l\nabla_i\tau_k = \frac12\bar\varepsilon^{kl}\bar{a}^i\,_l(\nabla_i\tau_k+\nabla_k\tau_i) = -\frac12\tr_h(\nabla\tau)\bar\varepsilon^{kl}\bar{a}^i\,_lh_{ik} = -\frac12\tr_h(\nabla\tau)\bar{\varepsilon}^{kl}a_{kl}=0.$ Since $a_{lj}$ is $h$-traceless, we have $a_{lj}\bar{a}^{j}\,_{m}=\frac12|a|^2_{h}h_{lm}$. If we assume $a\neq0$, we can therefore replace $\tau_j$ with $f_j$:
$$f_k\bar{a}^k\,_m = \frac12\varepsilon^{kl}\tau_k|a|^2_{h}h_{lm}\implies f_k\frac{2\bar{a}^{kl}}{|a|^2_{h}}=\varepsilon^{kl}\tau_k.$$
We conclude that
$$(h^{-1})^{ij}\nabla^2_{ij}f + \Big(2\frac{\bar{a}^{kl}}{|a|^2_{h}}a_{ij}\nabla_l(h^{-1})^{ij}\Big)f_k=0.$$
Since the above coefficients remain bounded as $a_{ij}\to 0$ in any coordinate neighborhood, moreover, $(h^{-1})^{ij}\nabla_{ij}f = -\varepsilon^{kl}\tau_ka_{ij}\nabla_l(h^{-1})^{ij} = 0 = f_k$, whenever $a_{ij}=0$, we are free to (arbitrarily) extend the coefficients such that $f$ satisfies an equation:
$$(h^{-1})^{ij}\nabla_{ij}f + R^kf_k=0$$ 
for some bounded $R^k$ on the 2-sphere. By the strong Maximum Principle we therefore have $df\equiv 0$ on $\mathbb{S}^2$.
\end{proof}

\begin{proposition}[\cite{li2020}, Theorem 11]
Given $h\in\Gamma(\text{Sym}(T^\star\mathbb{S}^2\otimes T^\star\mathbb{S}^2))$ such that $\mathcal{P}(h,h)>0$, then the operator $L_h:\Gamma(T^\star\mathbb{S}^2)\to\Gamma(\Xi_h)$ is surjective, equivalently $\ubar\Xi_h=\{0\}$, and $\dim(\text{Ker}(L_h))=6$.
\end{proposition}
\begin{proof}
We argue identically as in the beginning of Proposition \ref{p4} to conclude that $\text{index}(L_h) = \text{index}(\pi\circ L_\gamma) = \dim(\text{Ker}(L_{\gamma}))$. By conformal invariance of $L_\gamma$ we also conclude $\dim(\text{Ker}(L_\gamma)) = \dim(\text{Ker}(L_{\mathring\gamma})) = 6$. This implies $\dim(\text{Ker}(L_h))\geq 6$, and therefore $\text{Ker}(L_h)\neq\{0\}$. In the neighborhood of any point on our 2-sphere, we may choose isothermal co-ordinates $((x,y),\mathcal{U})$ so that $h=\psi^2((dx^1)^2+(dy)^2)$, for some $0<\psi\in\mathcal{F}(\mathcal{U})$. As a result, $L_h(\tau)=0$ is equivalent to the system of first order elliptic PDE:
\begin{align*}
\partial_x\begin{pmatrix}
\tau_x\\\tau_y	
\end{pmatrix}
+\begin{pmatrix}
0&-1\\
1&0	
\end{pmatrix}
\partial_y\begin{pmatrix}
	\tau_x\\\tau_y
\end{pmatrix}
=\begin{pmatrix}
	\Gamma^x_{xx}-\Gamma^x_{yy}&\Gamma^y_{xx}-\Gamma^y_{yy}\\
	2\Gamma^x_{xy}&2\Gamma^y_{xy}
\end{pmatrix}
\begin{pmatrix}
	\tau_x\\\tau_y
\end{pmatrix}
\end{align*}
where $\tau = \tau_xdx+\tau_ydy$, and $\Gamma^i_{jk}$ are the Christoffel symbols associated to $((x,y),\mathcal{U})$. Denoting the left-hand side by $L(\vec{\tau})$, we conclude that $|L(\vec{\tau})|\leq C|\vec{\tau}|$ ($|\vec{x}|$, the local Euclidean norm). The Strong Unique Continuation Principle of first order elliptic systems (see, for example Theorem 2.1, pg 151 of \cite{okaji2001strong}) implies $\tau$ cannot vanish on any open subset of $\mathcal{U}$ without vanishing throughout. As a result, given any non-trivial $\tau\in Ker(L_h)$, the collection of points $\mathcal{V}_\tau:=\{p\in\mathbb{S}^2|\tau_p\neq0\}$ are topologically dense within $\mathbb{S}^2$. Lemma \ref{la} therefore ensures that any $a\in\ubar\Xi_h$ is pointwise singular. Specifically, $0=\frac{\det a}{\det h}=-\frac12|a|^2_{h}$, enforcing $a\equiv 0$. So $\ubar\Xi_h=\{0\}$, and by Proposition \ref{p3}, $L_h$ is surjective. Consequently, $\dim(\text{Ker}(L_h))=6$.
\end{proof}
 
\bibliographystyle{abbrv}
%\cleardoublepage
\normalbaselines %Fixes spacing of bibliography
%\addcontentsline{toc}{section}{Bibliography} %adds Bibliography to your table of contents
\bibliography{foliation6.bbl} %your bibliography file - change the path if needed/
\end{document}